\documentclass[11pt]{article}
\usepackage{amsmath,amsthm,amsfonts,amssymb}
\usepackage[all]{xy}
\usepackage[usenames]{color} 

\newcounter{enumitemp}
\newenvironment{enumeratecontinue}{
 \setcounter{enumitemp}{\value{enumi}}
 \begin{enumerate}
 \setcounter{enumi}{\value{enumitemp}}
}
{
 \end{enumerate}
}

\newcommand\pref[1]{(\ref{#1})}

\newtheorem{thm}{Theorem}[section]

\newtheorem{theorem}[thm]{Theorem}

\newtheorem{lemma}[thm]{Lemma}

\newtheorem{corollary}[thm]{Corollary}
\newtheorem{proposition}[thm]{Proposition}
\newtheorem*{proposition*}{Proposition}

\newtheorem{fact}[thm]{Fact}

\theoremstyle{definition}

\newtheorem{definition}[thm]{Definition} 

\newtheorem*{defn*}{Definition}
\newtheorem{defns}[thm]{Definitions}

\newtheorem{remark}[thm]{Remark}
\theoremstyle{remark}

\newcounter{remarks}
{\paragraph*{Remarks}\smallskip
 \begin{list}{\arabic{remarks}. }{\usecounter{remarks}%
 \setlength{\leftmargin}{0in}%
 \setlength{\rightmargin}{0in}%
 \setlength{\labelsep}{0pt}%
 \setlength{\labelwidth}{0pt}%
 \setlength{\listparindent}{0pt}%
 }
}
{
\end{list}
}

\newcommand\from\colon
\newcommand\inv{{-1}}
\newcommand\subgroup{<}
\newcommand\infinity\infty
\newcommand\na{\text{na}}
\newcommand\supp{\text{supp}}
\newcommand\disjunion\coprod

\DeclareMathOperator{\Fix}{Fix}
\DeclareMathOperator{\Per}{Per}

\DeclareMathOperator{\cl}{cl}

\DeclareMathOperator\core{core}

\newcommand{\N}{{\mathbb N}}

\newcommand{\Z}{{\mathbb Z}}
\newcommand{\C}{{\mathcal C}}

\newcommand{\E}{{\mathcal E}}

\newcommand{\Out}{\mathsf{Out}}
\newcommand{\Aut}{\mathsf{Aut}}
\newcommand{\Inn}{\mathsf{Inn}}

\newcommand{\Stab}{\mathsf{Stab}}

\newcommand{\F}{\mathcal F}
\newcommand{\rtt}{relative train track map}

\renewcommand\L{\mathcal L}
\newcommand\LS{LS}
\def\B{\mathcal B}
\newcommand{\A}{\mathcal A}

\newcommand{\fG} {f : G \to G}
\newcommand{\ti} {\tilde}
\newcommand{\iNp} {indivisible Nielsen path}

\newcommand{\eg}{EG}
\newcommand{\noneg}{NEG}
\renewcommand\neg\noneg

\newcommand{\wt}{\widetilde}

\newcommand{\ct}{CT}
\newcommand{\cts}{CTs}

\newcommand{\comment}[1]{}

\newcommand\BH{\cite{BestvinaHandel:tt}}
\newcommand\BookZero{\cite{BFH:laminations}}
\newcommand\BookOne{\cite{BFH:TitsOne}}
\newcommand\BookTwo{\cite{BFH:TitsTwo}}

\newcommand\recognition{\cite{FeighnHandel:recognition}}

\DeclareMathOperator\interior{int}
\newcommand\bdy\partial
\newcommand\geod[2]{\overleftrightarrow{#1#2}}

\newcommand\intersect\cap
\newcommand\union\cup
\newcommand\<\langle
\renewcommand\>\rangle
\newcommand\meet\wedge
\newcommand\composed{\circ}
\newcommand\cross\times
\newcommand\restrict{\bigm |}
\newcommand\wh{\widehat}
\newcommand\inject\hookrightarrow
\newcommand\reals{\mathbf{R}}
\DeclareMathOperator\Length{Length}
\newcommand\abs[1]{\left|#1\right|}
\newcommand\Id{\text{Id}}

\DeclareMathOperator\CV{\mathcal{X}}
\DeclareMathOperator\rank{rank}
\DeclareMathOperator\BCC{BCC} 
 
 \newcommand\surjection\twoheadrightarrow
 
\newcommand\nc[1]{\<\!\< #1 \>\!\>}
\DeclareMathOperator\CAT{CAT}
\newcommand\suchthat{\bigm|}
\newcommand\hyp{\mathbf{H}}
\DeclareMathOperator\MCG{\mathcal{MCG}}
\DeclareMathOperator\Homeo{Homeo}

\title{Subgroup classification in $\Out(F_n)$}

\author{Michael Handel and Lee Mosher}

\begin{document}

\maketitle

\begin{abstract}
For any subgroup $H$ of $\Out(F_n)$, either $H$ has a finite index subgroup that fixes the conjugacy class of some proper, nontrivial free factor of $F_n$, or $H$ contains a fully irreducible element $\phi$, meaning that no positive power of $\phi$ fixes the conjugacy class of any proper, nontrivial free factor of $F_n$.
\end{abstract}

\tableofcontents

\break

\section{Introduction} 

Define a subgroup $H \subgroup \Out(F_n)$ to be \emph{reducible} if there exists a nontrivial free factorization $F_n = A_1 * \cdots * A_k * B$ such that $H$ preserves the set of conjugacy classes of the subgroups $A_1,\ldots,A_k$. If $H$ is not reducible then it is \emph{irreducible}, and if every finite index subgroup of $H$ is irreducible then $H$ is \emph{fully irreducible}. Note that for any finite index subgroup $H_0 \subgroup H$, $H$ is fully irreducible if and only if $H_0$ is fully irreducible. Reducible and irreducible are defined in the same manner for an individual element $\phi \in \Out(F_n)$, and $\phi$ is \emph{fully irreducible}\footnote{also known in the literature as ``irreducible with irreducible powers'' or ``iwip''} if $\phi^k$ is irreducible for every integer $k \ge 1$, equivalently the cyclic subgroup generated by $\phi$ is fully irreducible.

\begin{theorem} \label{thm:main}
Every fully irreducible subgroup of $\Out(F_n)$ contains a fully irreducible element.
\end{theorem}

\paragraph{Comparison with Mapping Class Groups.} In the mapping class group of a surface~$S$, an element is reducible if it preserves some essential curve system on~$S$, and a subgroup $H$ is reducible if there exists an essential curve system on $S$ which is preserved by every element of $H$. Thurston classified mapping classes by showing that each one is either finite order, reducible, or pseudo-Anosov, the latter meaning that it has an invariant pair of geodesic laminations with strong filling and attraction properties. The analogous classification of elements of $\Out(F_n)$ is given in \cite{BFH:laminations}, where it is proved that if $\phi \in \Out(F_n)$ is fully irreducible then $\phi$ has an invariant pair of laminations with strong filling and attraction properties. 

In the context of mapping class groups, subgroup classification was proved by Ivanov in \cite{Ivanov:subgroups}: every subgroup is either finite, reducible, or contains a pseudo-Anosov element. In particular, every finite index subgroup of an infinite irreducible subgroup is irreducible, a stronger conclusion for mapping class groups than the conclusion for $\Out(F_n)$ in Theorem~\ref{thm:main}. But a stronger conclusion of this fashion is false in $\Out(F_n)$, because there exist infinite order irreducible mapping classes having a positive power which is not irreducible, and hence there exist infinite irreducible subgroups which contain no fully irreducible elements.

\subsection*{Outline of the contents and the proof of Theorem~\ref{thm:main}} 

Section~\ref{SectionPrelim} contains a detailed and self-contained account of preliminary material taken from the literature on $\Out(F_n)$ and relative train tracks. In particular, following \recognition, we focus on \emph{rotationless} outer automorphisms, and on their representation by \emph{completely split improved relative train tracks} or \emph{\cts}. We also review attracting laminations and geometric strata.

Section~\ref{SectionSingular} introduces \emph{singular lines} of a certain class of outer automorphisms, which are generalizations to $\Out(F_n)$ of singular lines for pseudo-Anosov surface homeomorphisms. Among other results, in Lemma~\ref{LemmaSingularLine} we explain how to characterize singular lines very explicitly using a \ct.

Section~\ref{SectionVertexGroups} introduces the study of \emph{vertex group systems} in $F_n$. These are a special kind of \emph{subgroup system} which in general is just a finite set of conjugacy classes of finitely generated subgroups of $F_n$; see Subsection~\ref{SectionLams}. A vertex group system is, by definition, the system of maximal elliptic subgroups of some very small, minimal, $\reals$-tree action of $F_n$. The main result here is Proposition~\ref{PropVDCC}, a descending chain condition for vertex groups, whose proof was suggested to us by Mark Feighn. 

\bigskip

In order to outline the remainder of the paper and explain the methods of proof, we recall briefly some facts about $\Out(F_n)$. Recall from \BookOne\ that $\phi \in \Out(F_n)$ has \emph{exponential growth} if it satisfies any of the following equivalent properties: there exists a conjugacy class in $F_n$ whose cyclically reduced length grows exponentially under iteration by $\phi$; some (every) relative train track representative of $\phi$ has an \eg\ stratum; the set $\L(\phi)$ of attracting laminations is nonempty. Recall also that there is natural bijection between $\L(\phi)$ and $\L(\phi^\inv)$, the \emph{dual pairing} under which $\Lambda^+ \in \L(\phi)$ and $\Lambda^- \in \L(\phi^\inv)$ correspond if and only if they have the same free factor support in $F_n$; see Subsection~\ref{SectionLams} for a review.

Theorem~\ref{thm:main} is proved by the following steps. Given a fully irreducible subgroup $H \subgroup \Out(F_n)$ one produces a fully irreducible element of $H$ in steps. The starting point is to quote the main result of \BookTwo, a Kolchin type theorem which produces an element of $H$ having exponential growth. Step~1 is to use that element to produce another element $\phi \in H$ of exponential growth, and an attracting lamination $\Lambda \in \L(\phi)$, such that every nontrivial conjugacy class in $F_n$ is weakly attracted to $\Lambda$ under iteration by $\phi$ (see Proposition~\ref{PropUniversallyAttracting} below for a fuller statement of Step~1).  Step~2 is to use \emph{that} element to produce a fully irreducible element of $H$. The proofs of these two steps share common methods for understanding the action of $\Out(F_n)$ on geodesics: weak attraction theory, developed in Section~\ref{SectionWeakAttraction}; and ping-pong arguments, developed in Section~\ref{SectionPingPong}. 

The outer automorphisms and laminations that are produced in Step~1 of the above outline play a central role in this paper --- we can think of a lamination $\Lambda$ satisfying the conclusion of Step~1 as being ``universally attracting''. Outer automorphisms that have a universally attracting lamination are a considerably broader class than those which are fully irreducible, and yet they retain many analogies with pseudo-Anosov surface homeomorphisms. In particular, the theory of singular lines developed in Section~\ref{SectionSingular} of the paper is applicable to any outer automorphism that has a universally attracting lamination.

We turn now to a fuller outline of the proof of Theorem~\ref{thm:main}. 

\bigskip

Section~\ref{SectionWeakAttraction} contains a significant generalization of the weak attraction theory which was originally introduced in Section~6 of \BookOne\ for purposes of proving the Tits alternative. Given $\phi \in \Out(F_n)$ and a dual lamination pair $\Lambda^\pm$ with $\Lambda^+ \in \L(\phi)$ and $\Lambda^- \in \L(\phi^\inv)$, the main problem of weak attraction theory is to analyze the dynamics of the action of $\phi$ on the set of geodesics, in order to determine which geodesics are weakly attracted to $\Lambda^+$ under iteration by $\phi$ and/or to $\Lambda^-$ under iteration by $\phi^\inv$. Theorem~6.0.1 of \BookOne\ is a special weak attraction result that applies only in the special case of topmost lamination pairs and birecurrent geodesics, but here we drop the ``topmost'' and ``birecurrent'' hypotheses, developing weak attraction theory in full generality. 

In Subsection~\ref{SectionAsubNA} we construct the \emph{nonattracting subgroup system} $\A_\na\Lambda^\pm$, which is the focus of weak attraction theory. Subsection~\ref{SectionNASysIsVertexSys} contains the proof that $\A_\na\Lambda^\pm$ is a vertex group system. Our two main weak attraction results are stated in Subsection~\ref{SectionWAResults}, and their proofs are in Subsections~\ref{SectionNonattrFullHeight} and~\ref{SectionNonattrGeneral}. The first weak attraction result, Proposition~\ref{prop:WA1}, is the basis of the ping pong arguments used in Step 1 above. It says in part that for each geodesic $\gamma$, at least one of the following holds: $\gamma$ is weakly attracted to $\Lambda^+$ under iteration of $\phi$; or $\gamma$ is carried by $\A_\na\Lambda^\pm$; or the weak closure of $\gamma$ contains~$\Lambda^-$. The full statement Proposition~\ref{prop:WA1} also contains a uniform version of these conclusions. The second result, Proposition~\ref{prop:WA2}, is the basis of ping pong arguments used to prove Step~2 above. It applies under the stronger hypothesis that the nonattracting subgroup system $\A_\na\Lambda^\pm$ is trivial, equivalently $\Lambda^+$ (and $\Lambda^-$) is a ``universally attracting'' lamination. The conclusions are stronger as well, saying that for any geodesic $\gamma$ exactly one of the following holds: $\gamma$ is weakly attracted to $\Lambda^+$ under iteration of $\phi$; or $\gamma$ is a generic leaf of some lamination in $\L(\phi^\inv)$ or one of the singular leaves of $\phi^\inv$ studied in Section~\ref{SectionSingular}.

Section~\ref{transversality} contains a result about the vertex group systems $\A_\na \Lambda^\pm$ which will be important in carrying out the ping pong arguments of Section~\ref{SectionPingPong}. This result says roughly that the subgroup of $\Out(F_n)$ that stabilizes the vertex group system $\A_\na\Lambda^\pm$ also virtually stabilizes some natural free factor system. This result is trivial in the case that strata corresponding to $\Lambda^\pm$ are nongeometric, for then $\A_\na\Lambda^\pm$ is itself a free factor system, as is shown in Subsection~\ref{SectionNASysIsVertexSys}. But in the case that the corresponding strata are geometric, the proof requires a delicate transversality argument. 

Section~\ref{SectionPingPong} contains the statements and proofs of our ping pong arguments, in particular Proposition~\ref{PropSmallerComplexity} in Subsection~\ref{SectionPingPongArgument}. The idea of Proposition~\ref{PropSmallerComplexity} is that if one is given $\phi,\psi \in \Out(F_n)$ and lamination pairs $\Lambda^\pm_\phi \in \L^\pm(\phi)$, $\Lambda^\pm_\psi \in \L^\pm(\psi)$ which are sufficiently independent in an appropriate sense, then for large $m,n$ one can show that $\xi = \psi^m \phi^n$ has a lamination pair $\Lambda^\pm_\xi$ which is in some sense ``more complicated'' than either of $\Lambda^\pm_\phi$ or $\Lambda^\pm_\psi$, perhaps by attracting more geodesics, or perhaps by filling up more of the group. 

Subsection~\ref{SectionProofUnivAttr} also contains the application of ping pong arguments to the proof of the following central result:

\begin{proposition}[Constructing Universally Attracting Laminations]
\label{PropUniversallyAttracting} If $H \subgroup \Out(F_n)$ is fully irreducible then there exists $\phi \in H$ and a lamination pair $\Lambda^\pm \in \L^\pm(\phi)$ whose nonattracting vertex group system $\A_\na\Lambda^\pm$ is trivial. 
\end{proposition}

\noindent
As mentioned above, the proof starts by applying the main result of \BookTwo\ to obtain some element $\phi \in H$ of exponential growth. Choosing a lamination pair $\Lambda^\pm$ for $\phi$, the proof proceeds by induction on the ``complexity'' of the vertex group system $\A_\na\Lambda^\pm$, the base case being of course when $\A_\na\Lambda^\pm$ is trivial.  The induction step is carried out by applying ping-pong arguments. First one conjugates $\phi$ to an element $\psi \in H$ with a lamination pair $\Lambda^\pm_\psi$ that is sufficiently independent from the given lamination pair $\Lambda^\pm_\phi$.  Then, assuming that neither of $\A_\na\Lambda^\pm_\phi$, $\A_\na\Lambda^\pm_\psi$ is trivial, our ping pong argument Proposition~\ref{PropSmallerComplexity} applies to show that for all sufficiently large positive $m,n$, the element $\xi = \phi^m \psi^n$ has a lamination pair $\Lambda^\pm_\xi$ such that the vertex group system $\A_\na\Lambda^\pm_\xi$ is strictly nested in both $\A_\na\Lambda^\pm_\phi$ and $\A_\na\Lambda^\pm_\psi$. The descending chain condition for vertex group systems, Proposition~\ref{PropVDCC}, finishes the proof.


Section~\ref{SectionLooking} contains the proof of Theorem~\ref{thm:main}, an outline of which is found at the beginning of that section, with the ping-pong argument Proposition~\ref{PropSmallerComplexity} again playing a key role.

\subsection*{Further results} 

In the sequel to this paper, still in preparation, we will prove the following relative version of Theorem~\ref{thm:main}: 

\begin{theorem}
For any subgroup $H \subgroup \Out(F_n)$ and any proper free factor $A$ of $F_n$, if $H$ fixes the conjugacy class of $A$, and if no finite index subgroup of $H$ fixes the conjugacy class of any proper free factor that properly contains $A$, then there exists $\phi \in H$ such that no nonzero power of $\phi$ fixes the conjugacy class of any proper free factor that properly contains $A$.
\end{theorem}

This relative version can be combined with the results of Bestvina and Feighn in \cite{BestvinaFeighn:HypComplex}, to obtain applications to the 2nd bounded cohomology of an arbitrary subgroup of $\Out(F_n)$.

\paragraph{Acknowledgements.} The first author was supported in part by NSF grant DMS0706719. The second author was supported in part by NSF grant DMS0706799.

\section{Preliminaries}
\label{SectionPrelim}

In this section we give self-contained account of basic results about $\Out(F_n)$ needed in this paper, citing outside sources for the reader to verify the correctness of statements. With that intent, we try to be as succinct as possible. We shall state numerous ``Facts'' which are either cited from or proved quickly using results from the literature, and a few ``Lemmas'' whose proofs require a bit more care. 

\subsection{$F_n$ and its automorphisms, trees, lines, and subgroups.}
\label{SectionTheBasics}

For $n \ge 2$ let $R_n$ denote the rose of rank $n$, a graph with one vertex and $n$ directed edges labelled $E_1,\ldots,E_n$. Let $F_n$ denote the rank $n$ free group $\pi_1(R_n)$, having $E_1,\ldots,E_n$ as a free basis. We use the notation $[\cdot]$ to denote conjugacy classes of elements and of subgroups of $F_n$. Let $\C$ be the set of conjugacy classes of nontrivial elements of~$F_n$. 

Let $\bdy F_n$ denote the Gromov boundary of~$F_n$, identified with the Cantor set of ends of $\wt R_n$, the universal covering tree of~$R_n$. For any finitely generated subgroup $G \subgroup F_n t \bdy F_n$ denote its boundary --- by definition this is the set of accumulation points in $\bdy F_n$ of the orbit of some (any) element of $\wt R_n \union \bdy F_n$ under the action of $\Phi$ on the compactified space $\wt R_n \union \bdy F_n$, and it is equal to the closure of the set of ideal endpoints of axes of elements of $G$ acting on the tree $\wt R_n$.

\paragraph{Outer automorphisms and outer space.} Let $\Out(F_n) = \Aut(F_n) / \Inn(F_n)$ denote the outer automorphism group of the free group $F_n$ of rank $n$. The action of $\Aut(F_n)$ on elements and subgroups of $F_n$ induces an action of $\Out(F_n)$ on conjugacy classes of elements and subgroups. 

The Culler-Vogtmann outer space $\CV_n$ and its natural compactification $\overline\CV_n$ can be understand in terms of minimal actions of $F_n$ on $\reals$-trees for which no point or end of the tree is fixed by the whole action, what we refer to as \emph{$F_n$-trees}. Each $F_n$-tree determines an element of $\reals^\C$ using translation lengths as coordinates, and a theorem of Culler-Morgan \cite{CullerMorgan:Rtrees} states that two $F_n$-trees $T,T'$ determine the same point of $\reals^\C$ if and only if there is an isometry $T \mapsto T'$ that conjugates the $F_n$ actions; it follows that $T,T'$ determine the same point of projective space $P\reals^\C$ if and only if there is a homothety $T \mapsto T'$ that conjugates the $F_n$ actions. For each $F_n$-tree $T$, the \emph{stabilizer} of $A \subset T$ is $\Stab(A) = \{g \in F_n \suchthat g \cdot x = x \,\,\,\text{for all}\,\,\, x \in A\}$, and the \emph{fixed set} of $G \subset F_n$ is $\Fix(G) = \{x \in T \suchthat g \cdot x = x \,\,\,\text{for all}\,\,\, g \in G\}$.

Outer space $\CV_n$ is defined to be the space of free simplicial $F_n$ trees modulo conjugation by homothety, with topology induced by the translation length injection $\CV_n \inject P \reals^\C$. 

An $F_n$-tree $T$ is \emph{small} if for each nondegenerate arc $\alpha$ the subgroup $\Stab(\alpha)$ is either trivial or cyclic, from which it follows that for each $x \in T$ the subgroup $\Stab(x)$ has finite rank \cite{GJLL:index}. A small $F_n$-tree $T$ is \emph{very small} \cite{CohenLustig:verysmall} if for each triod $\tau \subset T$ the subgroup $\Stab(\tau)$ is trivial (a triod being a finite simplicial tree with one vertex of valence~3 and three vertices of valence~1), and for each $g \in F_n$ and each $i \ge 1$ we have $\Fix(g) = \Fix(g^i)$. It follows that for each nondegenerate arc $\alpha$ the subgroup $\Stab(\alpha)$ is either trivial or maximal infinite cyclic. Also, for each $x \in T$ the subgroup $\Stab(x)$ is maximal among the set of subgroups $G \subgroup F_n$ in the commensurability class of $\Stab(x)$, meaning that $G \intersect \Stab(x)$ has finite index in $G$ and in $\Stab(x)$.

Compactified outer space $\overline\CV_n$ is the space of very small $F_n$ trees modulo conjugation by homothety, again topologized using the natural translation length injection $\overline\CV_n \inject P\reals^\C$. By a theorem of Culler-Morgan \cite{CullerMorgan:Rtrees} the space $\overline\CV_n$ is compact, and by a theorem of Bestvina and Feighn \cite{BestvinaFeighn:OuterLimits}, $\overline\CV_n$ is the closure in $P\reals^\C$ of $\CV_n$. The compact space $\bdy \CV_n = \overline \CV_n - \CV_n$ is called the \emph{boundary} of outer space.

Outer space itself can also be understood in terms of marked graphs.
A \emph{marked graph} is a graph $G$ having no vertices of valence~1 equipped with a path metric and a homotopy equivalence $\rho \from G \to R_n$. The fundamental group $\pi_1(G)$ is therefore identified with $F_n$ up to inner automorphism, and so the conjugacy classes of $\pi_1(G)$ are naturally identified with $\C$. Each conjugacy class $c$ is represented by a unique shortest closed curve in $G$, and the length of this curve for each $c$ determines an element of $\reals^\C$. Two marked graphs $\rho \from G \to R_n$, $\rho' \from G' \to R'_n$ are equivalent if there is a homothety $\theta \from G \to G'$ such that $\rho' \theta$ is homotopic to $\rho$, and this happens if and only if the two corresponding elements of $P\reals^\C$ are equal. $\CV_n$ may therefore be defined as the space of marked graphs modulo equivalence, and the two definitions of $\CV_n$ correspond naturally under the universal covering relation between free, simplicial $F_n$-trees and marked graphs.

\paragraph{Paths and circuits.} A \emph{path} in a finite graph $K$ is a locally injective, continuous map $\gamma \from J \to K$ from a closed, connected, nonempty subset $J \subset \reals$ to $K$ such that any lift to the universal cover $\ti\gamma \from J \to \wt K$ is proper; it follows that if $J$ is noncompact then $\ti\gamma$ induces an injection from the ends of $J$ to the ends of $\wt K$. Two paths are \emph{equivalent} if they are equal up to an orientation preserving homeomorphism between their domains. A path is \emph{trivial} if its domain is a point. Every nontrivial path can be expressed as a concatenation of edges and partial edges, concatenated at vertices, in one of the following forms: a \emph{finite path} is a finite concatenation $E_0 E_1 \ldots E_n$ for $n \ge 0$ where only $E_0$ and $E_n$ may be partial edges; a \emph{positive ray} is a singly infinite concatenation $E_0 E_1 E_2 \ldots$ where only $E_0$ may be partial; a \emph{negative ray} is a singly infinite concatenation $\ldots E_{-2} E_{-1} E_0$ where only $E_0$ may be partial; and a \emph{line} is a bi-infinite concatenation $\ldots E_{-2} E_{-1} E_0 E_1 E_2 \ldots$ with no partial edges. This expression is unique, up to translation of the index in the case of a line. In each case we have $E_i \ne \bar E_{i+1}$, by the locally injective requirement. In general we will denote a negative ray by an expression like $\overline R$, which means the orientation reversal of a positive ray $R$.

A \emph{circuit} in $K$ is a continuous, locally injective map of an oriented circle into~$K$. Two circuits are \emph{equivalent} if they are equal up to some orientation preserving homeomorphism of domains. Any circuit can be written as a concatenation of edges which is unique up to cyclic permutation.

Consider a continuous function $\gamma \from J \to K$, such that $J \subset \reals$ is closed, connected and nonempty, and if $J$ is noncompact then each lift $\ti\gamma$ is proper and induces an injection from the ends of $J$ to the ends of $K$. If $J$ is a compact interval then either $\gamma$ is homotopic rel endpoints to a unique nontrivial path denoted $[\gamma]$, or $\gamma$ is homotopic rel endpoints to a constant path in which case $[\gamma]$ denotes the corresponding trivial path. If $J$ is noncompact then there is a unique path $[\gamma]$ to which $\gamma$ is homotopic by a homotopy which is \emph{proper} meaning that its lifts to the universal cover $\wt K$ are all proper; equivalently, $\gamma$ and the path $[\gamma]$ have lifts to $\wt K$ with the same finite endpoints and the same infinite ends. Also, every continuous, homotopically nontrivial function $c \from S^1 \to K$ is freely homotopic to a unique circuit denoted $[c]$. 

Given a homotopy equivalence $f \from K \to K'$ between two finite graphs, for each path or circuit $\gamma$ in $K$ we denote $f_\#(\gamma) = [f(\gamma)]$.  

\begin{fact}[Bounded Cancellation Lemma \cite{Cooper:automorphisms}, \BookZero]
For any homotopy equivalence between marked graphs $f \from G \to G'$ there exists a constant $\BCC(f)$ such that for any lift $\ti f \from \wt G \to \wt G'$ to universal covers and any path $\gamma$ in $\wt G$, the path $[f_\#(\gamma)]$ is contained in the $\BCC(f)$ neighborhood of the image $f(\gamma)$. \qed
\end{fact}

Given any finite graph $K$ let $\wh \B(K)$ be the compact space of equivalence classes of circuits in $K$ and paths in $K$ whose endpoints, if any, are at vertices of~$K$, where this space is given the \emph{weak topology} with one basis element $\wh N(K,\alpha)$ for each finite path $\alpha$ in $K$, consisting of all paths and circuits having $\alpha$ as a subpath. Let $\B(K) \subset \wh\B(K)$ be the compact subspace consisting of all lines in $K$, with basis elements $N(K,\alpha) = \wh N(K,\alpha) \intersect \B(K)$.

Following \BookOne, for marked graphs $G$ we give a coordinate free description of $\B(G)$ as follows. Let $\wt\B$ be the set of two-elements subsets of $\bdy F_n$, equivalently
$$\wt\B = (\bdy F_n \cross \bdy F_n - \Delta) / (\Z/2)
$$
where $\Delta$ is the diagonal and $\Z/2$ acts by interchanging factors. We put the \emph{weak topology} on $\wt\B$, induced from the usual Cantor topology on $\bdy F_n$. The group $F_n$ acts on $\wt\B$ with compact but non-Hausdorff quotient space $\B = \wt\B / F_n$; we also refer to this quotient topology as the \emph{weak topology}. Elements of $\B$ are called \emph{lines}. A \emph{lift} of a line $\gamma \in \B$ is an element $\ti\gamma \in \wt\B$ that projects to $\gamma$, and the two elements of $\ti\gamma$ are called its \emph{endpoints}. (Notation, used only rarely: the element of $\wt\B = (\bdy F_n \cross \bdy F_n - \Delta) / (\Z/2)$ corresponding to an ordered pair $(\xi,\eta) \in \bdy F_n \cross \bdy F_n - \Delta$ is denoted $\geod{\xi}{\eta}$.)

Given any marked graph $G$, its space of lines $\B(G)$ is naturally homeomorphic to $\B$, via the natural identification between the $\bdy F_n$ and the space of ends of the universal covering $\wt G$. Given $\gamma \in \B$ the corresponding element of $\B(G)$ is called the \emph{realization} of $\gamma$ in $G$. Keeping this homeomorphism in mind, we use the term ``line'' ambiguously, to refer either to an element of $\B$ or an element of $\B(G)$ in a context where a marked graph $G$ is under discussion. This allows us to speak, for example, of the lift of an element of $\B(G)$ to an element of $\wt\B$.

The group $\Out(F_n)$ acts naturally on $\B$, induced by the natural action of $\Aut(F_n)$ on $\wt\B$ which in turn is induced by the action of $\Aut(F_n)$ on $\bdy F_n$. For any marked graphs $G,G'$ and any homotopy equivalence $f \from G \to G'$ the induced map $f_\# \from \wh\B(G) \to \wh\B(G')$ is continuous, and the restricted map $\B(G) \to \B(G')$ is a homeomorphism. With respect to the identifications $\B(G) \approx \B \approx \B(G')$, if $f$ preserves marking then $f_\# \from \B(G) \to \B(G')$ agrees with the identity map on $\B$, whereas if $G=G'$ then $f_\#$ agrees with the homeomorphism $\B \to \B$ induced by the outer automorphism associated to $f$. 

We use the adjective ``weak'' to refer to concepts associated to the weak topology, although in contexts where this is clear we sometimes drop the adjective. For example, given a subset $A \subset \B$ its weak closure is denoted $\cl(A) \subset \B$. In particular, although $\wh\B(G)$ is not Hausdorff and limits of sequences are not unique, we use the phrase ``weak limit'' of a sequence to refer to any point to which that sequence weakly converges. Also, in any context where $\phi \in \Out(F_n)$ or $f \from G \to G$ are given, if $\alpha,\beta \in \B$ and $\phi^k(\alpha)$ converges weakly to $\beta$ as $k \to +\infinity$, or if $\alpha,\beta \in \wh\B(G)$ and $f^k_\#(\alpha)$ converges weakly to~$\beta$, then we say that \emph{$\alpha$ is weakly attracted to $\beta$} (under iteration by $\phi$ or by $f_\#$).

The \emph{accumulation set} of a ray $\rho$ in $G$ is the set of lines $\ell \in \B(G)$ which are elements of the weak closure of $\rho$ in the space $\wh\B(G)$; equivalently, every finite subpath of $\ell$ occurs infinitely often as a subpath of $\rho$. Two rays in $G$ are \emph{asymptotic} if they have equal subrays, in which case their accumulation sets are equal. The asymptotic equivalence classes of rays are called \emph{ends}, and the accumulation set of an end is well-defined as the accumulation set of any representative ray.

A line $\ell \in \B(G)$ is \emph{birecurrent} if $\ell$ is in the weak closure of some (any) positive subray of $\ell$ and $\ell$ is in the weak closure of some (any) negative subray of $\ell$. 

\paragraph{Attracting laminations.} For any marked graph $G$ the natural homeomorphism $\B \approx \B(G)$ induces a natural bijection between the set of closed subsets of $\B$ and the set of closed subsets of $\B(G)$, and an element of either of these sets is called a \emph{lamination}. In particular, given a lamination $\Lambda \subset \B$, the corresponding subset of $\B(G)$ is called the \emph{realization} of $\Lambda$ in $G$. Also, we use the term lamination to refer to $F_n$-invariant, closed subsets of $\wt B$; these are in natural bijection with laminations in~$\B$, under the $F_n$-orbit map $\wt B \to B$. An element of a lamination is called a \emph{leaf} of that lamination. A lamination $\Lambda$ is \emph{minimal} if each of its leaves is dense in $\Lambda$. In particular, every leaf of a minimal lamination is birecurrent.

Associated to each outer automorphism $\phi \in \Out(F_n)$ there is a finite, $\phi$-invariant set of laminations denoted $\L(\phi)$, called the \emph{attracting laminations} of $\phi$; see \BookOne\ Section 3.1. By definition, a lamination $\Lambda$ is an attracting lamination for $\phi$ if it is the weak closure of a line $\ell$ satisfying the following: $\ell$ is a birecurrent leaf of $\Lambda$; $\ell$ has an \emph{attracting neighborhood} which is a weak neighborhood $U \subset \B$ of $\ell$ such that each line in $U$ is weakly attracted to~$\ell$ (and so also to each leaf of $\Lambda$); and no lift $\ti\ell \in \B$ is the axis of a generator of a rank one free factor of $F_n$. Any such leaf $\ell$ is called a \emph{generic leaf} of $\Lambda$. For the proof of $\phi$-invariance of $\L(\phi)$ see \BookOne\ Lemma~3.1.16 (sic, located between Definitions 3.1.5 and 3.1.7). For finiteness see Lemma 3.1.13.

\paragraph{Directions, turns, and folds.} A \emph{direction} of a marked graph $G$ at a point $x \in G$ is a germ of finite paths with initial vertex~$x$. If $x$ is not a vertex then the number of directions at $x$ equals~2. If $x$ is a vertex then the number of directions equals the valence of $x$, in fact there is a natural bijection between the directions at $x$ and the oriented edges with initial vertex $x$, and we shall often identify a direction at $x$ with its corresponding oriented edge. Let $T_x G$ be the set of directions of $G$ at~$x$ and let $TG$ be the union of $T_x G$ over all vertices $x$ of $G$.  A \emph{turn} at $x \in G$ is an unordered pair of directions $\{d,d'\} \subset T_x G$. If $d \ne d'$ then the turn is \emph{nondegenerate}, otherwise it is \emph{degenerate}. 

For any homotopy equivalence of marked graphs $f \from G \to G'$ that takes edges to nontrivial paths, there is an induced map $Df \from TG \to TG'$ taking $T_x G$ to $T_{f(x)} G'$: for each $d \in T_x G$ represented by an oriented edge $E$ let $Df(d)$ be the direction represented by the path $f(E)$. 

Consider a homotopy equivalence $f \from G \to G'$ between two marked graphs, taking vertices to vertices, and such that for each edge $E \subset G$ the restriction of $f$ to $E$ is a nontrivial path in $G'$. We shall define folds and generalized folds of $f$.

Consider two oriented edges $E_1,E_2$ with the same initial point such that the turn $\{E_1,E_2\}$ is nondegenerate and its image $\{Df(E_1),Df(E_2)\}$ is degenerate. Let $\alpha_1 \subset E_1$ and $\alpha_2 \subset E_2$ be the longest initial segments such that $f(\alpha_1)$ and $f(\alpha_2)$ are the same path in $G'$; since $f$ is a homotopy equivalence, $\alpha_1,\alpha_2$ intersect only at their initial point. In this situation there exists a composition of homotopy equivalences $G \xrightarrow{h} H \xrightarrow{h'} G'$ such that $h$ is a quotient map that identifies $\alpha_1$ with $\alpha_2$, and $f$ is equal to $h' \composed h$. The map $h$ is called the \emph{fold} determined by $f$ and the turn $\{E_1,E_2\}$. If $\alpha_1 \ne E_1$ and $\alpha_2 \ne E_2$ then this fold is \emph{partial}, and otherwise it is \emph{full}. Supposing that the fold is full, if $\alpha_1 = E_1$ and $\alpha_2 = E_2$ then the fold is \emph{improper}, otherwise it is \emph{proper}. Note that the fold is proper if and only if, for some $i \ne j \in \{1,2\}$, the path $f(E_j)$ is a proper initial segment of $E_i$, in which case we say that \emph{$E_i$ folds properly over $E_j$}. 

Stallings Theorem \cite{Stallings:folding} says that any homotopy equivalence $f \from G \to G'$ as above may be factored~as
$$G=G_0 \xrightarrow{f_1} G_1 \xrightarrow{f_2} G_2 \to \cdots \to G_{k-1} \xrightarrow{f_k} G_k = G'
$$
such that each map $f_i$ is the fold determined by the map $G_{i-1} \xrightarrow{f_k \composed\cdots\composed f_i} G'$ and some turn of $G_{i-1}$.

Continuing the notation above, suppose that $E_i$ folds properly over $E_j$. Consider an initial subpath $\mu$ of $E_i$, and a path $\sigma$ which has $E_j$ as an an initial subpath, such that $\sigma$ is disjoint from the interior of $E_i$, $f(\mu)$ and $f_\#(\sigma)$ are the same path in $G'$, and this path has endpoints at vertices. Then there is a composition of homotopy equivalences $G \xrightarrow{k} K \xrightarrow{k'} G'$ such that $k$ is a quotient map that identifies $\mu$ with $\sigma$, and $f$ is homotopic rel vertices to $k' \composed k$. The map $k$ is called the \emph{generalized fold} determined by $\mu$ and $\sigma$.

\paragraph{Subgroup systems.} We recall from \BookOne\ some properties of free factor systems. Also, to lay the ground for the ``vertex group systems'' introduced in Section~\ref{SectionVertexGroups}, we discuss a general notion of subgroup systems.

Define a \emph{subgroup system} in $F_n$ to be a finite subset of the set of conjugacy classes of nontrivial, finite rank subgroups of $F_n$. Note that a subgroup system $\A$ is completely determined by the set of subgroups $A \subgroup F_n$ such that $[A] \in \A$. With this in mind, we often abuse terminology by writing $A \in \A$ when we really mean $[A] \in \A$.

A finite set of subgroups $\{A_1,\ldots,A_k\}$ of $F_n$ with pairwise trivial intersection is a \emph{free factorization} of $F_n$, written $F_n = A_1 * \ldots * A_k$, if each nontrivial element of $F_n$ can be written uniquely as a product of finite sequence of nontrivial elements $a_1 \ldots a_k$ so that for each $i$ there exists $j(i) \in \{1,\ldots,k\}$ such that $a_i \in A_{j(i)}$, and $j(i) \ne j(i+1)$.

A \emph{free factor system} is a subgroup system of the form $\{[A_1],\ldots,[A_k]\}$ such that there is a free factorization of the form $F_n = A_1 * \cdots * A_k * B$; the subgroup $B$ may or may not be trivial. The action of $\Out(F_n)$ on the collection of conjugacy classes of subgroups induces an action on the collection of free factor systems. There is a partial order on free factor systems defined by $\A \sqsubset \A'$ if for each subgroup $A \in \A$ there exists a subgroup $A' \in \A'$ such that $A$ is a free factor of $A'$. 

Section~2 of \BookOne\ contains the following fact and definition, which is a consequence of Grushko's Theorem:

\begin{fact}\label{FactGrushko}
Every collection $\{\A_i\}$ of free factor systems has a well-defined \emph{meet} $\meet\{\A_i\}$, which is the unique maximal free factor system $\A$ such that $\A \sqsubset \A_i$ for each~$i$. Furthermore, a subgroup is in $\meet\{\A_i\}$ if and only if it can be written as an intersection of subgroups each of which is in some $\A_i$. In particular it follows that the relation $\sqsubset$ satisfies the descending chain condition.
\end{fact}

We often take the liberty of referring to a singleton free factor system $\{[A]\}$ as a \emph{free factor}, rather than the lengthier but more appropriate locution ``free factor conjugacy class''. This should not cause any confusion in a context where $\Out(F_n)$ is acting, since the action is always on conjugacy classes of subgroups rather than on subgroups themselves.

Given $\phi \in \Out(F_n)$ and a free factor $F$ of $F_n$, if $[F]$ is $\phi$-invariant then there exists $\Phi \in \Aut(F_n)$ representing $\phi$ such that $\Phi(F)=F$. The outer automorphism class of $\Phi \restrict F \in \Aut(F)$ is well-defined independent of the choice of $\Phi$, and we obtain an element of $\Out(F)$ denoted $\phi \restrict F$ called the \emph{restriction of $\phi$ to $F$}.

For any marked graph $G$ and any subgraph $H \subset G$, the fundamental groups of the noncontractible components of $H$ determine a free factor system that we denote~$[H]$. A subgraph of $G$ is called a \emph{core graph} if it has no valence~1 vertices. Every subgraph $H \subset G$ has a unique core subgraph $\core(H) \subset H$ which is a deformation retract of the union of the noncontractible components of $H$. Note that \hbox{$[\core(H)] = [H]$}.

A subgroup system $\A$ \emph{carries} conjugacy class $[g]$ in $F_n$ if there exists $[A] \in \A$ such that $g \in A$, and $\A$ \emph{carries} a line $\ell \in \B$ if any of the following equivalent conditions hold:
\begin{itemize}
\item For some (any) marked graph $G$, the realization of $\ell$ is the weak limit of a sequence of circuits carried by $\A$. 
\item There exists $[A] \in \A$ and a lift $\geod{\xi}{\eta} \in \wt\B$ of $\ell$ such that $\xi,\eta \in \bdy A$. 
\end{itemize}
When $\A$ is a free factor system the following additional condition is equivalent:
\begin{itemize}
\item For some (any) marked graph $G$ and a subgraph $H \subset G$ with $[H]=\A$, the realization of $\ell$ in $G$ is contained in $H$.
\end{itemize}



\noindent
Since $\bdy A$ is compact in $\bdy F_n$ when $A \subgroup F_n$ has finite rank, we have:

\begin{fact}
\label{FactLinesClosed}
For each subgroup system $\A$, the set of lines carried by $\A$ is a closed subset of~$\B$.
\qed\end{fact}

The \emph{free factor support} of a subset of $\B$ is defined by the following result from \BookOne:

\begin{fact}[\BookOne, Corollary 2.6.5]
\label{FactFFSupport}
For any subset $B \subset \B$ the meet of all free factor systems that carry $B$ is a free factor system that carries $B$ denoted $\A_\supp(B)$, called the \emph{free factor support} of $B$ and denoted $\A_\supp(B)$. If $B$ is a single line then $\A_\supp(B)$ is a single free factor.
\end{fact}

A set of lines $B \subset \B$ is said to be \emph{filling} if $\A_\supp(B) = [F_n]$.

\bigskip

We need the following fact about weak convergence:

\begin{fact}
\label{FactWeakLimitLines} 
For any free factor $F$,
\begin{enumerate}
\item For every sequence of lines $\ell_i$, if every weak limit of every subsequence of $\ell_i$ is carried by $[F]$ then $\ell_i$ is carried by $[F]$ for all sufficiently large $i$.
\item The weak accumulation set of every end not carried by $[F]$ contains a line not carried by $[F]$.
\end{enumerate}
\end{fact}

\begin{proof} If the conclusion of (1) fails then some subsequence $\ell_{i_n}$ is not carried by $[F]$. Choose a marked graph $G$ and subgraph $H$ representing $[F]$. The realization of $\ell_{i_n}$ in $G$ is not contained in $H$ and so, after passing to a subsequence, each $\ell_{i_n}$ contains some edge $E \subset G-H$. It follows that some weak limit of some subsequence is a line that contains $E$, and so is not contained in $H$ and is not carried by $[F]$. 

If the ray $r$ in $G$ represents an end not carried by $[F]$ then there is an edge $E \subset G-H$ that the ray $r$ crosses infinitely many times, and we can proceed similarly to prove~(2).
\end{proof}

\subsection{Relative train tracks and \cts}
\label{SectionRTT}

Following \cite{FeighnHandel:recognition}, our focus will be on the class of ``completely split improved relative train tracks'' or \cts. 
For quick reference later in the paper we shall state the basic definitions here in full, and we shall state numerous facts about \cts\ which are either cited from or quickly proved from results of  \cite{BestvinaHandel:tt},  \BookOne\ and \recognition, centered on the theorem that every rotationless outer automorphisms is represented by a \ct. We will also state a few lemmas whose proofs require a bit more care.

\paragraph{Principal automorphisms and rotationless outer automorphisms.}

The normal subgroup $\Inn(F_n)$ acts by conjugation on $\Aut(F_n)$, and the orbits of this action define an equivalence relation on $\Aut(F_n)$ called \emph{isogredience}, which is a refinement of the relation of being in the same outer automorphism class.

The group $\Aut(F_n)$ acts on $\bdy F_n$. Let $\wh\Phi \from \bdy F_n \to \bdy F_n$ denote the action of $\Phi \in \Aut(F_n)$, and $\Fix(\wh\Phi)$ its set of fixed points. Let $\Fix(\Phi)$ denote the subgroup of $F_n$ fixed by $\Phi$, a finite rank subgroup by \cite{Cooper:automorphisms}, and let $\bdy \Fix(\Phi) \subset \bdy F_n$ denote its boundary, either empty when $\Fix(\Phi)$ is trivial, 2 points when $\Fix(\Phi)$ has rank~1, or a Cantor set when $\Fix(\Phi)$ has rank~$\ge 2$. Let $\Fix_+(\wh\Phi)$ denote the set of attractors in $\Fix(\wh\Phi)$, a discrete subset consisting of points $\xi \in \Fix(\wh\Phi)$ such that for some neighborhood $U \subset \bdy F_n$ of $\xi$  we have $\wh\Phi(U) \subset U$ and the sequence $\wh\Phi^n(\eta)$ converges to $\xi$ for each $\eta \in U$. Let $\Fix_-(\wh\Phi) = \Fix_+(\wh\Phi^\inv)$ denote the set of repellers in $\Fix(\wh\Phi)$.

\begin{fact}[\cite{GJLL:index}, Proposition I.1]
\label{LemmaFixPhiFacts} For each $\Phi \in \Aut(F_n)$ the set $\Fix(\wh\Phi)$ is the union of $\bdy\Fix(\Phi)$, $\Fix_-(\wh\Phi)$, and $\Fix_+(\wh\Phi)$. If $\Fix(\Phi)$ does not have rank~1 then this is a disjoint union.
\qed
\end{fact}


Let $\Fix_N(\wh\Phi) = \Fix(\wh\Phi) - \Fix_-(\wh\Phi) = \bdy\Fix(\Phi) \union \Fix_+(\wh\Phi)$. Let $\Per(\wh\Phi) = \union_{k \ge 1} \Fix(\wh\Phi^k)$, and similarly for $\Per_+(\wh\Phi)$, $\Per_-(\wh\Phi)$, and $\Per_N(\wh\Phi)$.

We say that $\Phi \in \Aut(F_n)$ representing $\phi \in \Out(F_n)$ is a \emph{principal automorphism} if $\Fix_N(\wh\Phi)$ contains at least three points, or $\Fix_N(\wh\Phi)$ contains exactly two points which are neither the endpoints of some axis nor the endpoints of a lift of a generic leaf of an element of $\L(\phi)$. The set of principal automorphisms representing $\phi \in \Out(F_n)$ is denoted $P(\phi)$, and is invariant under isogredience.

We say that $\phi \in \Out(F_n)$ is \emph{(forward) rotationless} if $\Fix_N(\wh\Phi) = \Per_N(\wh\Phi)$ for all $\Phi \in P(\phi)$, and if for each $k \ge 1$ the map $\Phi \mapsto \Phi^k$ induces a bijection between $P(\phi)$ and $P(\phi^k)$.

We list here some properties of rotationless outer automorphisms $\phi \in \Out(F_n)$. 

\begin{fact} \label{FactRotationlessPower} (\recognition, Lemma 4.43)
There exists $K$ depending only on the rank $n$ such that for each $\phi \in \Out(F_n)$, $\phi^K$ is rotationless.
\end{fact}


\begin{fact} \label{FactPeriodicIsFixed} (\cite{FeighnHandel:recognition}, Lemma 3.30 and Corollary 3.31)
If $\phi \in \Out(F_n)$ is rotationless then
\begin{enumerate}
\item \label{ItemPeriodicClassFixed}
Every conjugacy class in $F_n$ which is $\phi$-periodic is fixed by $\phi$. 
\item \label{ItemFreeFactorFixed}
Every free factor system in $F_n$ which is $\phi$-periodic is fixed by $\phi$.
\item If $F$ is a $\phi$-invariant free factor then $\phi \restrict F$ is rotationless.
\item $\phi$ fixes every element of $\L(\phi)$.
\end{enumerate}
\end{fact}

\paragraph{Topological representatives and Nielsen paths.}
Given $\phi \in \Out(F_n)$ a \emph{topological representative} of $\phi$ is a map $f \from G \to G$ such that $G$ is a marked graph, $f$ is a homotopy equivalence, $f$ takes vertices to vertices and edges to paths, and $\bar\rho \composed f \composed \rho \from R_n \to R_n$ represents $\phi$, where $\rho$ is the marking of $G$ and $\bar\rho$ is a homotopy inverse. 

A nontrivial path $\gamma$ in $G$ is a \emph{periodic Nielsen path} if there exists $k$ such that $f^k_\#(\gamma)=\gamma$; the minimal such $k$ is the \emph{period}, and if $k=1$ then $\gamma$ is a \emph{Nielsen path}. A periodic Nielsen path is \emph{indivisible} if it cannot be written as a concatenation of two nontrivial periodic Nielsen paths. We will often abuse our notion of path equivalence when talking about Nielsen paths: many statements asserting the uniqueness of an indivisible Nielsen path $\rho$, such as Fact~\ref{FactEGNPUniqueness}, should really assert uniqueness of $\rho$ \emph{up to reversal of direction}; since $\rho$ determines its reversal $\bar\rho$ and vice versa, this abuse is easily detectable and harmless.

Given two subgraphs $H \subset K$ of a marked graph $G$ each with no isolated vertices, let $K \setminus H$ denote the subgraph which is the union of the edges of $K$ that are not in~$H$. 

\paragraph{Filtrations and strata.}
A \emph{filtration} of a marked graph $G$ is a strictly increasing sequence of subgraphs $G_0 \subset G_1 \subset \cdots \subset G_k = G$, each with no isolated vertices. The subgraph $H_k = G_k \setminus G_{k-1}$ is called the \emph{stratum of height $k$}. The \emph{height} of subset of $G$ is the minimum $k$ such that the subset is contained in $G_k$. The height of a map to $G$ is the height of the image of the map. The height of an end equals $s$ if every representative ray in $G$ has a subray of height $s$, equivalently the end is represented by a ray of height $s$ in $G$ that contains infinitely many edges of $H_s$. A \emph{connecting path} of a stratum $H_k$ is a nontrivial finite path $\gamma$ of height $< k$ whose endpoints are contained in $H_k$.

Given a topological representative $f \from G \to G$ of $\phi \in \Out(F_n)$, we say that $f$ \emph{respects} the filtration or that the filtration is \emph{$f$-invariant} if $f(G_k) \subset G_k$ for all $k$. If this is the case then we also say that the filtration is \emph{reduced} if for each free factor system $\A$ which is invariant under $\phi^i$ for some $i \ge 1$, if $[G_{r-1}] \sqsubset \A \sqsubset [G_r]$ then either $\A = [G_{r-1}]$ or $\A = [G_r]$.

Given an $f$-invariant filtration, for each stratum $H_k$ with edges $\{E_1,\ldots,E_m\}$, define the \emph{transition matrix} of $H_k$ to be the square matrix whose $i,j$ entry equals the number of times that the edge $E_j$ occurs in the path $f(E_i)$. If the transition matrix is irreducible --- meaning that for each $i,j$ there exists $p$ such that the $i,j$ entry of the $p^{\text{th}}$ power of the matrix is nonzero --- then we say that the stratum $H_k$ is \emph{irreducible}. If furthermore one can choose $p$ independently of $i,j$ then the stratum $H_k$ is \emph{aperiodic}. If $M_k$ is the zero matrix then we say that $H_k$ is a \emph{zero stratum}. 

\emph{Henceforth} we make the assumption that every stratum is either irreducible or a zero stratum --- if this is not already true then the filtration may be refined to make it true.

Let $H_k$ be an irreducible stratum with transition matrix $M$. The Perron-Frobenius theorem \cite{Hawkins:PerronFrobenius}
says that there exists a unique $\lambda \ge 1$ with the property that $M$ has a positive eigenvector with eigenvalue $\lambda$; we call this the \emph{Perron-Frobenius eigenvalue} of $M$. Furthermore, if $\lambda>1$ then a positive eigenvector for $\lambda$ is unique up to scalar multiplication; we call such a vector a \emph{Perron-Frobenius eigenvector}. If $\lambda>1$ then we say that the irreducible stratum $H_k$ is an \emph{exponentially growing} stratum or \eg\ stratum, whereas if $\lambda=1$ then $H_k$ is a \emph{nonexponentially growing} or \neg\ stratum.

\emph{Henceforth} we make the assumption that the edges of each \neg\ stratum $H_k$ may be numbered and oriented as $E_1,\ldots,E_N$ so that $f(E_i) = E_{i+1} u_i$ where $u_i \subset G_{k-1}$ and indices are taken modulo $N$. To see why this assumption can be made, by irreducibility the transition matrix $M$ is the permutation matrix of a permutation with a single cycle in its cycle decomposition, and so if this assumption is not already true then we can make it true by subdividing each edge of $H_k$ into two edges and either leaving $H_k$ as is or subdividing it into a pair of \neg\ strata. If each $u_i$ is the trivial path then we say that $H_k$ is a \emph{periodic stratum} and each edge $E \subset H_k$ is a \emph{periodic edge}; furthermore, if $N=1$ then $H_k$ is a \emph{fixed stratum} and $E$ is a \emph{fixed edge}. If each $u_i$ is a periodic Nielsen path then we say that $H_k$ is a \emph{linear stratum} and each edge $E \subset H_k$ is a \emph{linear edge}.

\paragraph{Directions and turns.}
If a filtration of $G$ is given, the \emph{height} of a direction $d$ is well-defined as the height of any sufficiently short path representing $d$, equivalently the height of the oriented edge representing $d$. A nondegenerate turn is said to have height $r$ if \emph{each} of its directions has height $r$.

Given a marked graph $G$ and a homotopy equivalence $f \from G \to G$ that takes edges to paths, a turn $\{d,d'\}$ in $G$ is said to be \emph{legal} with respect to the action of $f$ if $\{(Df)^k(d),(Df)^k(d')\}$ is nondegenerate for all $k \ge 0$. For any path $\gamma$ and any turn $\{E,E'\}$ at a vertex, if $\bar E E'$ or its inverse $\bar E' E$ is a subpath of $\gamma$ then we say that the turn $\{E,E'\}$ is \emph{taken} by $\gamma$. If in addition an $f$-invariant filtration is given, a path $\gamma$ is \emph{$r$-legal} if it has height $\le r$ and each turn of height $r$ taken by $\gamma$ is legal.

\begin{definition} \textbf{Relative train track map.} Given $\phi \in \Out(F_n)$ and a topological representative $f \from G \to G$ with invariant filtration $G_0 \subset G_1 \subset \cdots \subset G_k = G$, we say that $f$ is a \emph{relative train track} if for each \eg\ stratum $H_r$ the following hold:
\begin{description}
\item[RTT-(i)] $Df$ maps the set of directions of height $r$ to itself. In particular, each turn consisting of a direction of height $r$ and one of height $<r$ is legal.
\item[RTT-(ii)] For each connecting path $\gamma$ of $H_r$, $f_\#(\gamma)$ is a connecting path of $H_r$.
\item[RTT-(iii)] For each $r$-legal path $\alpha$, the path $f_\#(\alpha)$ is $r$-legal.
\end{description}
\end{definition}

\paragraph{Principal vertices.} Consider a relative train track map $f \from G \to G$ with filtration $\emptyset = G_0 \subset G_1 \subset \cdots \subset G_N$. Two periodic points $x \ne y \in G$ are \emph{Nielsen equivalent} if there exists a periodic Nielsen path with endpoints $x,y$. A periodic point $x$ is \emph{principal} if and only if \emph{neither} of the following hold:
\begin{itemize}
\item\label{ItemEGNonprincipal} $x$ is the unique periodic point in its Nielsen class, there are exactly two periodic directions based at $x$, and both of those directions are in the same \eg\ stratum.
\item\label{ItemBadCircle} $x$ is contained in a topological circle $C \subset G$ such that each point of $C$ is periodic and has exactly two periodic directions.
\end{itemize}



\paragraph{Splittings and complete splittings.} Let $f \from G \to G$ be a relative train track map with invariant filtration $\emptyset = G_0 \subset G_1 \subset \cdots \subset G_N = G$. A \emph{splitting} of a path or circuit $\sigma$ in $G$ is a decomposition of $\sigma$ into a concatenation of subpaths $\sigma_1 \ldots \sigma_k$, where $k \ge 1$ if $\sigma$ is a circuit and $k \ge 2$ if $\sigma$ is a path, such that for all $i \ge 1$ the (path or circuit) $f^i_\#(\sigma)$ decomposes as $f^i_\#(\sigma_1) \ldots f^i_\#(\sigma_k)$. The notation $\sigma = \sigma_1 \cdot \ldots \cdot \sigma_k$ means that this decomposition is a splitting of $\sigma$, and the paths $\sigma_1,\ldots,\sigma_k$ are called the \emph{terms} of the splitting. 

Certain paths will be atomic terms of splittings, including the following two types. 

Given a zero stratum $H_i$, a path $\tau$ in $H_i$, and an irreducible stratum $H_j$ with $j>i$, we say that $\tau$ is \emph{$j$-taken} if there exists an edge $E \in H_j$ and $k \ge 1$ such that $\tau$ is a maximal subpath in $H_i$ of the path $f^k_\#(E)$. We also say that $\tau$ is \emph{taken} if it is $j$-taken for some $j$. 

Given two linear edges $E_i, E_j$, a root-free closed Nielsen path $w$, and integers $m_i,m_j > 0$ such that $f_\#(E_i) = E_i w^{m_i}$ and $f_\#(E_j) = E_j w^{m_j}$, each path of the form $E_i w^p \overline E_j$ for $p \in \Z$ is called an \emph{exceptional path}. Note that the height of $w^p$ is less than the heights of $E_i$ and $E_j$.

\begin{definition} \textbf{Complete Splittings.}
\label{DefCompleteSplitting}
A splitting $\sigma = \sigma_1 \cdot \ldots \cdot \sigma_k$ is called a \emph{complete splitting} if each term $\sigma_i$ satisfies one of the following: $\sigma_i$ is an edge in an irreducible stratum; $\sigma_i$ is an indivisible Nielsen path; $\sigma_i$ is an exceptional path; there is a zero stratum $H_j$ such that $\sigma_i$ is a maximal subpath of $\sigma$ in $H_j$, and $\sigma_i$ is taken. A path is said to be \emph{completely split} if it has a complete splitting.
\end{definition}


\paragraph{Enveloping of zero strata.} Consider a relative train track map $f \from G \to G$ with filtration $\emptyset = G_0 \subset G_1 \subset \cdots \subset G_N=G$ and two strata $H_u,H_r$ with $1 \le u<r \le N$. Suppose that the following hold:
\begin{enumerate}
\item $H_u$ is irreducible.
\item $H_r$ is EG and each component of $G_r$ is noncontractible.
\item Each $H_i$ with $u<i<r$ is a zero stratum that is a component of $G_{r-1}$, and each vertex of $H_i$ has valence $\ge 2$ in $G_r$. 
\end{enumerate}
In this case we say that the zero strata $H_i$ with $u<i<r$ are \emph{enveloped} by the \eg\ stratum $H_r$, and we write $H^z_r = \union_{i=u+1}^r H_i$.

\paragraph{Nielsen paths of \eg\ height.} Consider a relative train track map $f \from G \to G$ with filtration $\emptyset = G_0 \subset G_1 \subset \cdots \subset G_K=G$, aperiodic \eg\ stratum $H_r$. Every indivisible height~$r$ periodic Nielsen path $\rho$ may be written uniquely as a concatenation $\rho = \alpha\beta$ so that the two directions $DF(\bar\alpha)$, $Df(\beta)$ have height $r$ and the turn $\{Df(\bar\alpha),Df(\beta)\}$ is degenerate (\cite{BestvinaHandel:tt} Lemma~5.11). 

Given an indivisible height~$r$ periodic Nielsen path $\rho = \alpha\beta$, the fold which is determined by $f$ and the turn $\{\bar\alpha,\beta\}$ is referred to as \emph{the fold determined by $f$ and $\rho$}. If this fold is proper then, by Lemma~2.22 of \recognition, there is a relative train track map $f' \from G' \to G'$ with filtration $\emptyset = G'_0 \subset G'_1 \subset \cdots \subset G'_K = G'$, an aperiodic \eg\ stratum $H'_r$, an indivisible height $r$ Nielsen path $\rho' = \alpha'\beta'$, and a filtration preserving homotopy equivalence $F \from G \to G'$, which are characterized by the following properties: letting $E_1,E_2$ be the initial edges of $\bar\alpha,\beta$, respectively, and assuming the notation chosen so that $E_1$ folds over $E_2$, we have:
\begin{itemize}
\item $\bar\alpha = E_1 b E_3 \ldots$, where $b$ is a (possibly trivial) subpath of $G_{r-1}$ and $E_3$ is an edge in $H_r$.
\item $E_2 = e''_2 e'_2$ where $f(e''_2) = f_\#(E_1 b)$
\item $F \from G \to G'$ is the generalized fold determined by $E_1 b$ and $e''_2$.
\item $G'_k = F(G_k)$ for $k=1,\ldots,K$, 
\item $\rho' = F_\#(\rho)$
\item $f' \composed F$ and $F \composed f$ are homotopic rel vertices.
\end{itemize}
We say that $F$ is the \emph{proper extended fold} determined by $f \from G \to G$, $H_r$, and $\rho$.
If the fold determined by $f'$ and $\rho'$ is proper then this process can be repeated a second time. Repeating this process $L$ times if possible, and obtaining a sequence $G=G^0 \xrightarrow{F = F_1} G'=G^1 \xrightarrow{F_2} \cdots \xrightarrow{F_L} G^L$, we say that $F_L \composed \cdots \composed F_1 \from G \to G^L$ is a \emph{composition of proper extended folds defined by iteratively folding $\rho$.} Notice that in this situation each of the restrictions $F_l \from G^{l-1}_{r-1} \to G^l_{r-1}$ is a homeomorphism, and so we identify each of these filtration elements with~$G_{r-1}$.


\bigskip

\begin{definition} \textbf{CT --- Completely split improved relative train track map.}
\label{DefCT}
Given a relative train track map $f \from G \to G$ and an $f$-invariant filtration 
$$\emptyset = G_0 \subset G_1 \subset \cdots \subset G_N = G
$$
we say that $f$ is a \ct\ if $f$ satisfies the following properties:
\begin{enumerate}
\item \textbf{(Rotationless)} Each principal vertex is fixed by $f$ and each periodic direction at a principal vertex is fixed by $Df$.
\item \textbf{(Completely Split)} For each edge $E$ in each irreducible stratum, the path $f(E)$ is completely split. For each taken connecting path $\sigma$ in each zero stratum, the path $f_\#(\sigma)$ is completely split.
\item \textbf{(Filtration)} The filtration $\emptyset = G_0 \subset G_1 \subset \cdots \subset G_N = G$ is reduced. For each~$i$ there exists $j \le i$ such that $G_i = \core(G_j)$.
\item \textbf{(Vertices)} The endpoints of all indivisible Nielsen paths are vertices. The terminal endpoint of each nonfixed \neg\ edge is principal.
\item \textbf{(Periodic Edges)} Each periodic edge is fixed and each endpoint of a fixed edge is principal. If $E = H_r$ is a fixed edge and $E$ is not a loop then $G_{r-1}$ is a core graph and both ends of $E$ are contained in $G_{r-1}$.
\item \label{ItemZeroStrata}
\textbf{(Zero Strata)} Each zero stratum $H_i$ is enveloped by some \eg\ stratum $H_r$, each edge in $H_i$ is $r$-taken, and each vertex in $H_i$ is contained in $H_r$ and has link contained in $H_i \union H_r$.
\item \textbf{(Linear Edges)} For each linear edge $E_i$ there exists a closed root-free Nielsen path $w_i$ such that $f(E_i) = E_i w_i^{d_i}$ for some $d_i \ne 0$. If $E_i,E_j$ are distinct linear edges, and if the closed path $w_i$ is freely homotopic to the closed path $w_j$ or its inverse, then $w_i=w_j$ and $d_i \ne d_j$.
\item \label{ItemNEGNielsenPaths}
\textbf{(\neg\ Nielsen Paths)} If $H_i = \{E_i\}$ is an \neg\ stratum and $\sigma$ is an indivisible Nielsen path of height $i$ then $E_i$ is linear and, with $w_i$ as in (Linear Edges), there exists $k \ne 0$ such that $\sigma = E_i w^k_i \overline E_i$.
\item \textbf{(\eg\ Nielsen Paths)} If $H_r$ is \eg, if $\rho$ is an indivisible Nielsen path of height~$r$, and if $G^0_r$ is the component of $G_r$ containing $H_r$, then there exists a map $f_r \from G^0_r \to \Gamma^1$ which is a composition of proper extended folds defined by iteratively folding $\rho$, and a map $f_{r-1} \from \Gamma^1 \to \Gamma^2$ which is a composition of folds involving only edges of height $\le r-1$, and a homeomorphism $\theta \from \Gamma^2 \to G^0_r$, such that $f \restrict G^0_r$ is homotopic rel vertices to $\theta \composed f_{r-1} \composed f_r$.
\end{enumerate}
\end{definition}


The following is the main technical result of \cite{FeighnHandel:recognition}. Given a relative train track map $f \from G \to G$ with filtration $\emptyset = G_0 \subset G_1 \subset G_K=G$, and given an increasing sequence $\C = (\F_1 \sqsubset \cdots \sqsubset \F_L)$ of free factor systems, we say that $f$ \emph{realizes} $\C$ if there exists a sequence $1 \le i_1 < \cdots < i_L \le K$ such that $\F_{l} = [G_{i_l}]$ for each $l=1,\ldots,L$.

\begin{theorem}\label{TheoremCTExistence}
For each rotationless $\phi \in \Out(F_n)$ and each increasing sequence $\C$ of $\phi$-invariant free factor systems, there exists a \ct\ $f \from G \to G$ that represents $\phi$ and realizes $\C$. \qed
\end{theorem}

\paragraph{Facts about \cts.} Let $f \from G \to G$ be a \ct\ with filtration $\emptyset = G_0 \subset G_1 \subset \cdots \subset G_k$, representing a rotationless $\phi \in \Out(F_n)$. 

\bigskip

\begin{fact}\label{FactEGAperiodic} If $H_r$ is an \eg\ stratum then $H_r$ is aperiodic and $G_r$ is a core graph.
\end{fact}

\begin{proof} By \recognition\ Lemma~3.19, $H_r$ contains a principal vertex $v$ whose link contains a principal direction in $H_r$. By (Rotationless), that direction is fixed. The transition matrix of $H_r$ therefore has a nonzero element on the diagonal, which implies that $H_r$ is aperiodic. By (Filtration) and \recognition\ Lemma~2.20~(2), $G_r$ is a core graph.
\end{proof}

\begin{fact}[\recognition, Remark 4.9]
\label{FactUsuallyPrincipal} A vertex of $G$ whose link contains edges in more than one irreducible stratum is principal. \qed
\end{fact}

The next fact is used repeatedly without reference:

\begin{fact}[\cite{FeighnHandel:recognition}, Lemma 4.6]
\label{FactComplSplitStable}
If $\sigma$ is a path in $G$ with endpoints at vertices and if $\sigma$ is completely split then $f_\#(\sigma)$ is completely split. \qed
\end{fact}


\begin{fact} \label{FactFinestSplitting}
If $\sigma$ is a finite path in $G$ with endpoints at vertices then:
\begin{enumerate}
\item \label{ItemCompleteCharacterized}
If $\sigma = \sigma_1 \ldots \sigma_m$ is a decomposition into subpaths, each of which is either a single edge in an irreducible stratum, an indivisible Nielsen path, an exceptional path, or a taken connecting path in a zero stratum, and if each turn $(\bar\sigma_i,\sigma_{i+1})$ is legal, then $\sigma = \sigma_1 \cdot \ldots \cdot \sigma_m$ is the unique complete splitting of $\sigma$.
\item \label{ItemRefined}
Any splitting of $\sigma$ at vertices is refined by the complete splitting of $\sigma$.
\end{enumerate}
\end{fact}

\begin{proof}
Item \pref{ItemCompleteCharacterized} is contained in Lemma~4.12 of \BookOne. Item \pref{ItemRefined} is proved by noticing that for any splitting of $\sigma$ at vertices, if each term of the given splitting is further subdivided into its complete splitting, the resulting subdivision of $\sigma$ satisfies the hypothesis of~\pref{ItemCompleteCharacterized}. 
\end{proof}



%

\begin{fact}[\cite{FeighnHandel:recognition}, Lemma 4.26]
\label{FactEvComplSplit}
If $\sigma$ is a path in $G$ with endpoints at vertices then $f^k_\#(\sigma)$ is completely split for all sufficiently large~$k \ge 1$. \qed
\end{fact}

\begin{fact}\label{FactPrincipalVertices}
For each periodic vertex $x \in G$, exactly one of the following holds:
\begin{enumerate}
\item $x$ is principal (and therefore fixed); 
\item $x$ is the unique periodic point in its Nielsen class, there are exactly two periodic directions based at $x$, and both of those directions are in the same \eg\ stratum.
\end{enumerate}
\end{fact}

\begin{proof} If $x$ is contained in a topological circle $C$ of periodic points then, applying (Periodic Edges), each edge of $C$ is fixed and (1) holds. If $x$ is contained in no such circle then, by definition of principal, either (1) or (2) holds. Items (1) and (2) are mutually exclusive by definition of principal.
\end{proof}

\textbf{Remark:} The above argument shows that circles as in the second item of the definition of principal vertices do not exist in \ct's.

\begin{fact}[\cite{FeighnHandel:recognition}, Lemma 4.22] 
\label{FactNEGEdgeImage}
Each nonfixed \neg\ stratum $H_i$ is a single edge $E_i$ satisfying $f(E_i) = E_i \cdot u_i$ for some nontrivial closed path $u_i$ of height $< i$ which forms a circuit such that the turn $\{u_i, \bar u_i\}$ is legal. \qed
\end{fact}



\begin{fact}
\label{FactNielsenBottommost} If $H_r$ is an \eg\ stratum then every indivisible periodic Nielsen path of height $r$ is fixed. If there exists an indivisible Nielsen path $\rho$ of height~$r$, and if $\rho = a_0 b_1 a_1 b_2 \cdots a_{k-1} b_k$ is the decomposition into maximal subpaths $a_i$ in $H_r$ and $b_i$ in $G_{r-1}$ then:
\begin{enumerate}
\item \label{ItemNoZeroStrata} No zero stratum is enveloped by $H_r$. 
\item Each $b_i$ is a Nielsen path.
\item \label{ItemBottommostEdges}
For each edge $E \subset H_r$ and each $k \ge 0$, the path $f^k_\#(E)$ splits into terms each of which is an edge of $H_r$ or one of the Nielsen paths $b_i$ in the decomposition of $\rho$. Furthermore, each term in the complete splitting of $f^k_\#(E)$ is an edge of $H_r$, a fixed edge, or an indivisible Nielsen path.
\end{enumerate}
\end{fact}

\begin{proof} The first sentence follows from \cite{FeighnHandel:recognition}, Lemma~3.28 and Remark~4.8. The rest is all contained in \cite{FeighnHandel:recognition}, Lemma 4.25, except for the following parts of item \pref{ItemBottommostEdges}. The first sentence is only stated in Lemma~4.25 for the case $k=1$; the general case follows by induction on $k$. To prove the second sentence, by combining the first sentence with Fact~\ref{FactFinestSplitting}~\pref{ItemRefined} it follows that each term in the complete splitting of $f^k_\#(E)$ is either an edge of $H_r$ or a term in the complete splitting of one of the Nielsen paths $b_i$ in the decomposition of $\rho$. However, each term in the complete splitting of a Nielsen path is a Nielsen path that does not split at vertices, and so is either a fixed edge or an indivisible Nielsen path.
\end{proof}


\begin{fact} \label{FactContrComp}
For each filtration element $G_r$ the following are equivalent: 
\begin{enumerate}
\item \label{ItemHasContrComp}
$G_r$ has a contractible component; 
\item \label{ItemIsZeroStrat}
$H_r$ is a zero stratum;
\item \label{ItemIsContrComp}
$H_r$ is a contractible component of $G_r$.
\end{enumerate}
\end{fact}

\begin{proof}
The equivalence of \pref{ItemHasContrComp} and~\pref{ItemIsZeroStrat} is Lemma~4.16 of \recognition. That \pref{ItemIsZeroStrat} implies \pref{ItemIsContrComp} is a consequence of (Zero Strata) which implies that $H_r$ is a component of $G_r$, and the fact that $f_\#$ is a bijection on circuits. That \pref{ItemIsContrComp} implies \pref{ItemHasContrComp} is obvious.
\end{proof}

\begin{fact}\label{FactEdgeToZeroConnector}
For any edge $E \subset G$, if $f_\#(E)$ is contained in a zero stratum~$H_t$ enveloped by the \eg\ stratum $H_s$ then $E$ is an edge in some other zero stratum $H_{t'} \ne H_t$ that is also enveloped by $H_s$.
\end{fact}

\begin{proof} Since the path $f_\#(E)$ is completely split and contained in the zero stratum $H_t$, its complete splitting must have just the single term $f_\#(E)$ which must be a taken connecting path of $H_s$. If the oriented edge $E$ is contained in an irreducible stratum $H_i$ then some term in the complete splitting of $f_\#(E)$ is an edge of $H_i$, contradiction. It follows that $E$ is contained in some zero stratum $H_{t'}$ enveloped by some \eg\ stratum~$H_{s'}$. 

By (Zero Strata) applied to $H_{t'}$ there exists an oriented edge $E_0 \subset H_{s'}$ having the same initial vertex as $E$. The paths $f_\#(E_0)$ and $f_\#(E)$ therefore have the same initial vertex, and by (Zero Strata) applied to $H_t$ the link of this vertex is contained in $H_{s} \union H_t$. By RTT-(i) the initial direction of $f_\#(E_0)$ is contained in $H_{s'}$, and so $s=s'$. 

If $H_t = H_{t'}$ then $f(H_t) \intersect H_t \ne \emptyset$, contradicting the definition of a zero stratum.
\end{proof}

\paragraph{Facts and lemmas about \eg\ strata.} As shown in \recognition\ Corollary 4.20, some facts about aperiodic \eg\ strata follow from the lone assumption of (EG Nielsen Paths) unaccompanied by other defining properties of a \ct. Of these we need only the following, which is part of Property~(eg)-i of \BookOne.

\begin{fact}\label{FactEGNPUniqueness}
If $f \from G \to G$ is a relative train track map and if $H_r$ is an aperiodic \eg\ stratum of $f$ satisfying (EG Nielsen Paths) then there exists at most one indivisible Nielsen path of height~$r$. \qed
\end{fact}

The following two lemmas use none of the defining properties of a \ct, deriving instead from \BookOne\ Section~4.2. The next lemma is essentially the same as the \eg\ case of the proof of \recognition\ Lemma~4.6.

\begin{lemma}
\label{LemmaEGPathSplitting}
Let $f \from G \to G$ be a relative train track and $H_r \subset G$ an aperiodic \eg\ stratum. For each $L$ there exists an integer $k = k_L \ge 0$ such that for each height $r$ path or circuit $\sigma$ whose endpoints, if any, are vertices, if $\sigma$ has length $\le L$ then $f^k_\#(\sigma)$ splits into terms each of which is an edge or indivisible periodic Nielsen path of height~$r$ or a path in~$G_{r-1}$.
\end{lemma}

\begin{proof} Since there are only finitely many circuits and paths with endpoints at vertices that have length $\le L$, it suffices to consider just a single $\sigma$ and find a $K$ such that $f^K_\#(\sigma)$ satisfies the conclusions. 

Let $P_r$ be the set of paths $\rho$ of height $r$ such that for each $k \ge 0$ the path $f^k_\#(\rho)$ begins and ends with an edge in $H_r$ and has exactly one $r$-illegal turn, and such that the number of edges in $f^k_\#(\rho)$ is bounded independently of $k$. By \BookOne\ Lemma 4.2.5, $P_r$ is a finite $f_\#$ invariant set, from which it follows that there exists $K_1$ such that for each $\rho \in P_r$ the path $f^{K_1}_\#(\rho)$ is a periodic Nielsen path, which therefore splits into indivisible periodic Nielsen paths of height $r$ and paths in $G_{r-1}$. 

Since $f_\#$ applied to an $r$-legal path is $r$-legal, and since a subpath of an $r$-legal path is $r$-legal, it follows that the number of maximal $r$-legal paths in $f^i_\#(\sigma)$ is a nonincreasing function of~$i$. The number of illegal turns of $f^i_\#(\sigma)$ in $H_r$ is therefore nonincreasing, and so is constant for sufficiently large $i$, say $i \ge K_2$. Applying Lemma~4.2.6 of \BookOne\ it follows that $f^{K_2}_\#(\sigma)$ splits into subpaths each of which is $r$-legal or is one of the paths in $P_r$. Letting $K=K_1+K_2$ it follows that $f^K_\#(\sigma)$ splits into $r$-legal paths and indivisible Nielsen paths of height $r$. Since an $r$-legal path splits into edges in $H_r$ and paths in $G_{r-1}$, this finishes the proof.
\end{proof}

We also need a form of Lemma~\ref{LemmaEGPathSplitting} whose conclusion has a stronger uniformity. This will be used in the proof of Proposition~\ref{PropVerySmallTree}.

\begin{lemma}
\label{LemmaEGUnifPathSplitting}
Let $f \from G \to G$ be a relative train track and $H_r \subset G$ an aperiodic \eg\ stratum. For each $M$ there exists an integer $d \ge 0$ such that for each height $r$ path or circuit $\sigma$ with endpoints, if any, at vertices, if $\sigma$ contains at most $M$ edges in the subgraph $H_r$ then $f^d_\#(\sigma)$ splits into terms each of which is an edge or indivisible periodic Nielsen path of height $r$ or a path in~$G_{r-1}$.
\end{lemma}

\begin{proof} In this proof, all paths and circuits are in $G$ and have endpoints, if any, at vertices.

If a path or circuit $\alpha$ of height $r$ has a splitting satisfying the conclusions of the lemma, the terms being edges and indivisible Nielsen paths of height $r$ and paths in $G_{r-1}$, then $f_\#(\alpha)$ also has a splitting satisfying the conclusions. We are therefore free to increase the exponent on $f_\#$ as needed in this proof, which we shall do without mention.

\subparagraph{Case 1: $\sigma$ is a height $r$ path.} We prove the lemma in this case by induction on~$M$. If $\sigma$ is $r$-legal, in particular if $M=1$, we can take $d=0$. Assume by induction that the exponent $D=d_{M-1}$ works when $\sigma$ has $\le M-1$ edges in $H_r$. 

Let $B \ge 0$ be a bounded cancellation constant for $f^D$. Let $L$ be a constant such that if $\beta$ is a path of length $> L$ then the path $f_\#(\beta)$ has length $\ge 2B+1$; the constant $L$ exists because the lift of $f$ to the universal cover of $G$ is a quasi-isometry. Given a path $\alpha \beta \gamma$, if $\beta$ has length $> L$ then not all of $f^D_\#(\beta)$ is cancelled when $f^D_\#(\alpha) f^D_\#(\beta) f^D_\#(\gamma)$ is tightened to $f^D_\#(\alpha \beta \gamma)$: at most $B$ initial edges of $f^D_\#(\beta)$ cancel with $f^D_\#(\alpha)$, and at most $B$ terminal edges cancel with $f^D_\#(\gamma)$. 

Let $d_M$ be the maximum of $D$ and the constant $K=k_{M+ML-L}$ from Lemma~\ref{LemmaEGPathSplitting}. Let $\sigma$ be a height $r$ path with exactly $M$ edges in $H_r$. 

We first reduce to the subcase that $\sigma$ begins and ends with edges in $H_r$. In the case of a general path $\sigma$, we can write $\sigma = \alpha \tau \beta$ where $\alpha,\beta$ are the longest initial and terminal segments in $G_{r-1}$. Knowing that the path $f^{d_M}_\#(\tau)$ satisfies the conclusions, it has a splitting $f^{d_M}_\#(\tau) = \alpha' \cdot \tau' \cdot \beta'$ where $\alpha',\beta'$ are its longest initial and terminal segments in $G_{r-1}$ and where $\tau'$ satisfies the conclusions, and so $f^{d_M}_\#(\sigma) = [f^{d_M}_\#(\alpha) \alpha'] \cdot \tau' \cdot [\beta' f^{d_M}_\#(\beta)]$ has a splitting that satisfies the conclusions, completing the reduction.


%

Suppose now that $\sigma$ begins and ends with edges in $H_r$. If $\sigma$ has length $\le M+ML-L$ then $h^K_\#(\sigma)$ satisfies the conclusions. If~$\sigma$ has length $>M+ML-L$ then, since $\sigma$ begins and ends with edges in $H_r$, we can write $\sigma = \alpha \beta \gamma$ where $\beta$ in $G_{r-1}$ has length $> L$ and $\alpha,\gamma$ each have between~$1$ and $M-1$ edges in $H_r$. The induction hypothesis implies that $f^D_\#(\alpha)$ and $f^D_\#(\gamma)$ each satisfy the conclusions, so there are splittings
\begin{align*}
f^D_\#(\alpha) &= \alpha' \cdot \beta_1 \\
f^D_\#(\gamma) &= \beta_2 \cdot \gamma'
\end{align*}
where $\beta_1$, $\beta_2$ are each maximal subpaths in $G_{r-1}$, the paths $\alpha',\gamma'$ satisfy the conclusions, and the terminal edge of $\alpha'$ and the initial edge of $\gamma'$ are both in~$H_r$. Our choice of $L$ guarantees that $\beta' =  [\beta_1 f^D_\#(\beta) \beta_2]$ is nontrivial, and so 
\begin{align*}
f^D_\#(\sigma) & = [f^D_\#(\alpha) f^D_\#(\beta) f^D_\#(\gamma)] \\
  &= \alpha' \beta' \gamma'
\end{align*}
This is a splitting of $f^D_\#(\sigma)$, because the turns $\{\bar\alpha',\beta'\}$, $\{\bar\beta',\gamma'\}$ are legal by RTT-(i), and $\alpha',\gamma'$ have splittings that satisfy the conclusions. It follows that $f^D_\#(\sigma)$ has a splitting that satisfies the conclusions.

\subparagraph{Case 2: $\sigma$ is a height $r$ circuit.} Suppose that $\sigma$ contains exactly $M$ edges in~$H_r$. Using the constant $d_M$ from Case~1, let $B'$ be a bounded cancellation constant for $f^{d_M}$. Let $L'$ be a constant such that if $\beta$ is a path of length $>L'$ then the path $f^{d_M}_\#(\beta)$ has length $\ge 2B'+1$. Let $d'_M$ be the maximum of $d_M$ and the constant $K'=k_{M(L'+1)}$ from Lemma~\ref{LemmaEGPathSplitting}. If $\sigma$ has length $\le M(L'+1)$ then $f^{K'}_\#(\sigma)$ satisfies the conclusions. If $\sigma$ has length $>M(L'+1)$ then $\sigma$ has a maximal subpath $\beta$ in $G_{r-1}$ of length $> L'$, with a complementary subpath $\tau$ that begins and ends in $H_r$ and has exactly $M$ edges in $H_r$. Applying Case~1, the path $f^{d_M}_\#(\tau)$ satisfies the conclusions, so there is a splitting 
$$f^{d_M}_\#(\tau) = \beta_1 \cdot \tau' \cdot \beta_2
$$
where $\beta_1,\beta_2$ are maximal subpaths in $G_{r-1}$, $\tau'$ satisfies the conclusions, and $\tau'$ begins and ends with edges of $H_r$. The choice of $L'$ guarantees that $\beta' = [\beta_2 f^{d_M}_\#(\beta) \beta_1]$ is nontrivial, and so $f^{d_M}_\#(\sigma) = \beta' \tau'$. This is a splitting of the circuit $f^{d_M}_\#(\sigma)$, because $\{\bar\beta',\tau'\}$ and $\{\bar\tau',\beta'\}$ are legal by RTT-(i) and $\tau'$ has a splitting that satisfies the conclusions. It follows that $f^{d_M}_\#(\sigma)$ has a splitting that satisfies the conclusions.

\bigskip

In conclusion, we have proved the lemma with $d = \max\{d_M,d'_M\}$.
\end{proof}

\subsection{Properties of Attracting Laminations}
\label{SectionLams} 

Attracting laminations were defined in Section~\ref{SectionTheBasics}. We review from Section~3 of \BookOne\ various properties of attracting laminations. Some proofs will use concepts of tiles from the same section.

\paragraph{Attracting laminations and \cts.} The following fact is compiled from results in \BookOne\ Section~3, from Definition~3.1.5 to Definition~3.1.12:

\begin{fact} \label{FactLamsAndStrata}
For any $\phi \in \Out(F_n)$ and any relative train track $f \from G \to G$ representing a positive power of $\phi$ such that each \eg\ stratum of $f$ is aperiodic (such as a \ct), the set $\L(\phi)$ is in one-to-one correspondence with the set of \eg\ strata $H_r \subset G$, where $\Lambda_r \in \L(\phi)$ corresponds to $H_r$ if and only if the realization in $G$ of each generic leaf of $\Lambda_r$ has height~$r$. 
\qed\end{fact}

The next two results characterize leaves and generic leaves of attracting laminations in terms of relative train tracks. Recall from Definition~3.1.7 of \BookOne\ that a path of the form $f^k_\#(E)$, where $k \ge 1$ and $E$ is an edge of $H_r$, is called a \emph{$k$-tile of height $r$}; we may drop $k$ and/or $r$ if they are clear from the context. 


\begin{fact} \label{FactAttractingLeaves}
With the notation of Fact~\ref{FactLamsAndStrata}, for each \eg\ stratum $H_r \subset G$ we have:
\begin{enumerate}
\item \label{ItemTileDecomp} 
For each $k$, each generic leaf of $\Lambda_r$ has a decomposition into subpaths each of which is a $k$-tile of height $r$ or is a path in $G_{r-1}$. 
\item \label{ItemTileExhausted} 
Each generic leaf of $\Lambda_r$ can be exhausted by tiles of height $r$, meaning that it can be written as an increasing union of subpaths each of which is a tile of height $r$.
\item \label{ItemTilePF}
There exists $p$ such that each for each $k \ge 0$, each $k+p$-tile of height $r$ contains each $k$ tile of height $r$.
\item \label{ItemLeafAsLimit}
The leaves of $\Lambda_r$ are characterized as the set of lines to which some (any) edge of $H_r$ is weakly attracted under iteration by $f$.
\item \label{ItemLeafComplSplit} Every generic leaf $\ell$ of $\Lambda_r$ has a complete splitting, which refines the splitting of $\ell$ into edges of $H_r$ and maximal subpaths in $G_{r-1}$.
\end{enumerate}
\end{fact}

\subparagraph{Remark.} Nongeneric leaves of $\Lambda_r$ are completely split too. In fact, to be a completely split path is a closed property in the weak topology on $\B(G)$. We omit the proof.

\begin{proof} 
Items~\pref{ItemTileDecomp} and~\pref{ItemTileExhausted} are contained in Lemma 3.1.10 of \BookOne. Item~\pref{ItemTilePF} follows for $k=0$ by choosing $p$ so that the $p^{\text{th}}$ power of the transition matrix for $H_r$ is positive, and it follows for general $k$ by induction. 

We prove item~\pref{ItemLeafAsLimit}. If $\ell$ is a generic leaf of $\Lambda_r$ then by \pref{ItemTileExhausted} any subpath of $\ell$ is a subpath of some $k$-tile, which by \pref{ItemTilePF} is a subpath of some tile $f^{k+p}_\#(E)$ for some (any) edge $E$ of $H_r$. It follows that $E$ is weakly attracted to $\ell$ by iteration of $\phi$, and so $E$ is also weakly attracted to every other leaf of $\Lambda_r$, each of which is a weak limit of $\ell$. Conversely, given an edge $E$ of $H_r$, suppose $E$ is weakly attracted to a line $\ell$. Choose a generic leaf $\beta$ of~$\Lambda_r$. Each subpath of $\ell$ is contained in $f^k_\#(E)$ for some $k$, which by \pref{ItemTilePF} is contained in any $r$-tile of the form $f^{k+p}_\#(E')$. By \pref{ItemTileDecomp}, some such tile $f^{k+p}_\#(E')$ is a subpath of $\beta$, proving that $\ell$ is in the weak closure of $\beta$, which is~$\Lambda_r$.

To prove item~\pref{ItemLeafComplSplit}, by combining (Completely Split) with Fact~\ref{FactComplSplitStable} it follows that each tile is completely split. By definition of relative train track, the decomposition of a tile of height $r$ into edges of $H_r$ and maximal subpaths in $G_{r-1}$ is a splitting which, by Fact~\ref{FactFinestSplitting}~\pref{ItemRefined}, is refined by the complete splitting of that tile. Combining this with item~\pref{ItemTileExhausted}, it follows for each generic leaf $\ell$ that each maximal subpath of $\ell$ in $G_{r-1}$ is completely split, and taking these terms together with the edges of $H_r$ we obtain a complete splitting of $\ell$.
\end{proof}

\begin{fact}\label{FactTwoEndsGeneric}
With the notation of Fact~\ref{FactLamsAndStrata}, for any \eg\ stratum $H_r \subset G$ with corresponding lamination $\Lambda_r \in \L(\phi)$, and for any leaf $\ell$ of $\Lambda$, the following are equivalent:
\begin{enumerate}
\item \label{ItemGenericLeaf}
$\ell$ is a generic leaf of $\Lambda_r$.
\item \label{ItemBirecurrentLeaf}
$\ell$ is birecurrent and has height $r$.
\item \label{ItemDoubleHeightLeaf}
Both ends of $\ell$ have height $r$.
\end{enumerate}
\end{fact}

\begin{proof} The implications \pref{ItemGenericLeaf} $\implies$ \pref{ItemBirecurrentLeaf} $\implies$ \pref{ItemDoubleHeightLeaf} are obvious, and \pref{ItemBirecurrentLeaf} $\implies$ \pref{ItemGenericLeaf} follows from \BookOne\ Lemma~3.1.15. 

We prove \pref{ItemDoubleHeightLeaf} $\implies$ \pref{ItemBirecurrentLeaf}. Let $\beta$ be a generic leaf of $\Lambda$. For each $k \ge 0$, applying Fact~\ref{FactAttractingLeaves}~\pref{ItemTileDecomp} to $\beta$ it follows that there exists $N_k > 0$ so that any subpath of $\beta$ that contains at least $N_k$ edges of $H_s$ contains a $k$-tile. Since every subpath of $\ell$ occurs as a subpath of $\beta$, every subpath of $\ell$ that contains at least $N_k$ edges of $H_s$ contains a $k$-tile. Since both ends of $\ell$ have height $s$ it follows that every initial or terminal ray of $\ell$ contains a $k$-tile for all~$k$. By Fact~\ref{FactAttractingLeaves}~\pref{ItemTileExhausted}, each subpath of $\beta$, and hence each subpath $\sigma$ of $\ell$, is contained in a subpath which is a tile, and hence $\sigma$ occurs as a subpath of every initial or terminal ray  of $\ell$. This proves that $\ell$ is birecurrent.
\end{proof}

\begin{remark} \label{basin is open} Suppose that $\fG$ is a \rtt\ and that $\Lambda^+ \in \L(\phi)$. Lemma~4.2.2 and Corollary~4.2.4 of \BookOne\ imply that being weakly attracted to $\Lambda^+$ under iteration by $f_\#$ is an open condition on paths in $\B(G)$. 
\end{remark}

\paragraph{Pairing of attracting laminations.} Consider any $\phi \in \Out(F_n)$. For any $\Lambda \in \L(\phi)$ and any generic leaf $\ell \in \Lambda$, their free factor supports $\A_\supp(\Lambda)$ and $\A_\supp(\ell)$ are equal and consist of a single free factor; see Definition~3.2.3 of \BookOne. 

By Lemma~3.2.4 of \BookOne, there is a bijection between $\L(\phi)$ and $\L(\phi^\inv)$ such that $\Lambda^+ \in \L(\phi)$ and $\Lambda^- \in \L(\phi^\inv)$ correspond under the bijection if and only if their free factor supports $\A_\supp(\Lambda^-)$, $\A_\supp(\Lambda^+)$ are equal. We use the notation $\L^\pm(\phi)$ to denote the set of ordered pairs $\Lambda^\pm = (\Lambda^+,\Lambda^-) \in \L(\phi) \cross \L(\phi^\inv)$ that correspond under this bijection, and we refer to the elements of $\L^\pm(\phi)$ as the \emph{dual lamination pairs} of~$\phi$.

\begin{fact}
\label{FactDualDifferent}
For each $\phi \in \Out(F_n)$ and each $\Lambda^\pm \in \L^\pm(\phi)$, $\Lambda^-$ is not a subset of $\Lambda^+$, equivalently no generic leaf of $\Lambda^-$ is a leaf of $\Lambda^+$.
\end{fact}

\begin{proof} The two conclusions are equivalent because $\Lambda^-$ is the closure of any of its generic leaves and $\Lambda^+$ is closed. 

Arguing by contradiction suppose that the generic leaf $\ell$ of $\Lambda^-$ is also a leaf of~$\Lambda^+$. Applying Theorem~\ref{TheoremCTExistence} choose a \ct\ $f \from G \to G$ representing a positive power of $\phi$ with \eg\ stratum $H_r$ corresponding to $\Lambda^+$ such that $[G_r] = \A_\supp(\Lambda^\pm) = \A_\supp(\ell)$, and so the realization of $\ell $ in $G_r$ has height $r$. By Fact~\ref{FactTwoEndsGeneric} the line $\ell$ is birecurrent, since it is generic leaf of $\Lambda^-$. Having shown that $\ell$ is a birecurrent height $r$ leaf of $\Lambda^+$, it is therefore a generic leaf of $\Lambda^+$ by Fact~\ref{FactTwoEndsGeneric}. It follows that $\Lambda^- = \Lambda^+$.

For any $\theta \in \Out(F_n)$ and any $\Lambda \in \L(\theta)$ let $\Stab(\Lambda)$ be the subgroup of all $\psi \in \Out(F_n)$ such that $\psi(\Lambda) = \Lambda$, so in particular $\theta,\theta^\inv \in \Stab(\Lambda)$. By \BookOne\ Definition~3.3.2 and Proposition~3.3.3 there is a natural homomorphism $\mu_\Lambda \from \Stab(\Lambda) \to \reals_+$ such that $\mu(\psi) > 1$ if and only if $\Lambda \in \L(\psi)$. Since $\Lambda^- = \Lambda^+$, by naturality we have equality of homomorphisms $\mu_{\Lambda^+} = \mu_{\Lambda^-}$. But $\mu_{\Lambda^+}(\phi) > 1$ and $\mu_{\Lambda^+}(\phi^\inv)=\mu_{\Lambda^-}(\phi^\inv) > 1$ so $\mu_{\Lambda^+}(\Id) = \mu_{\Lambda^+}(\phi) \cdot \mu_{\Lambda^+}(\phi^\inv) > 1$ implying that $\Lambda^+ \in \L(\Id)$, contradicting that $\L(\Id) = \emptyset$.
\end{proof}

\paragraph{Inclusion of attracting laminations.} The set $\L(\phi)$ of attracting laminations for $\phi \in \Out(F_n)$ is partially ordered by inclusion. Define an attracting lamination to be \emph{topmost} if it is not properly included in any other attracting lamination, and \emph{bottom-most} if it does not properly include any other attracting lamination. 

The following lemma shows how inclusion of attracting laminations can be detected on the level of \cts.

\begin{proposition}
\label{PropInclusion}
For any $\phi \in \Out(F_n)$, any \ct\ $f \from G \to G$ representing a rotationless positive power of $\phi$, and any two EG strata $H_r, H_s \subset G$ with corresponding attracting laminations $\Lambda_r, \Lambda_s \in \L(\phi)$, we have $\Lambda_r \subset \Lambda_s$ if and only if for some (every) edge $E$ of $H_s$ and some $k \ge 1$, some term in the complete splitting of $f^k_\#(E)$ is an edge $E'$ of $H_r$.
\end{proposition}

\begin{proof} The hypothesis of the ``if'' direction implies that for any line $\ell$ in $G$, if $E'$ is weakly attracted to $\ell$ by iteration of $f_\#$ then $E$ is also weakly attracted to $\ell$. By Fact~\ref{FactAttractingLeaves}~\pref{ItemLeafAsLimit} it follows that $\Lambda_r \subset \Lambda_s$.

Conversely, suppose that for each edge $E$ of $H_s$ and each $k \ge 1$, no term in the complete splitting of the height $s$ tile $f^k_\#(E)$ is an edge of $H_r$. We claim that if an edge $E'$ of $H_r$ is in $f^k_\#(E)$ then $E'$ is in a subpath $\rho$ of $f^k_\#(E)$ that is an indivisible Nielsen path of height $r$. To see why, let $\sigma$ be a term in the complete splitting of $f^k_\#(E)$ that contains $E'$, so $\sigma$ must be an exceptional path or an indivisible Nielsen path, and $\sigma$ has height $t$ with $r \le t < s$. If $\sigma$ is an exceptional path then $H_t$ is \neg\ and $r<t$ and so $E'$ is contained in a subpath of $\sigma$ which is an indivisible Nielsen path. Let $\rho$ be a subpath of $f^k_\#(E)$ of minimal height $s$ that contains $E'$ and is an indivisible Nielsen path, and arguing by contradiction suppose that $s>r$. If $H_s$ is \eg\ then, by Fact~\ref{FactNielsenBottommost}, $E'$ is contained in a subpath of $\rho$ which is an indivisible Nielsen path of height less than the height of $\rho$, contradicting minimality. If $H_s$ is \neg\ then the same contradiction is obtained from (\neg\ Nielsen Paths). This proves the claim.

Since each generic leaf $\ell_s$ of $\Lambda_s$ is exhausted by height $s$ tiles, it follows from the claim that the only way an edge of $H_r$ can occur in $\ell_s$ is in an indivisible height $r$ Nielsen path which is a subpath of $\ell_s$. However, if $\ell_r$ is a generic leaf of $\Lambda_r$ then, by Fact~\ref{FactAttractingLeaves}~\pref{ItemTileExhausted}, $\ell_r$ is exhausted by $r$-tiles, and since an $r$-tile is $r$-legal, $\ell_r$ cannot contain any indivisible height $r$ Nielsen path. Since $\Lambda_s$ is the weak closure of $\ell_s$, it follows $\ell_r$ is not a leaf of~$\Lambda_s$.
\end{proof}

\begin{corollary}\label{CorBottommost}
For any $\phi \in \Out(F_n)$, any $\Lambda \in \L(\phi)$, and any \ct\ $f \from G \to G$ representing a rotationless positive power of $\phi$, if $H_r$ is the \eg\ stratum corresponding to $\Lambda$, and if there exists an indivisible Nielsen path $\rho_r$ of height $r$, then $\Lambda$ is a bottommost element of $\L(\phi)$. 
\end{corollary}

\begin{proof} For each edge $E \subset H_r$ and each $k \ge 0$, by Fact~\ref{FactNielsenBottommost}~\pref{ItemBottommostEdges} each term in the complete splitting of $f^k_\#(E)$ is an edge of $H_r$, a fixed edge, or an indivisible Nielsen path. The result follows from Proposition~\ref{PropInclusion}.
\end{proof}



It was proved in \BookOne\ Corollary~6.0.1 that a lamination in $\L(\phi)$ is topmost if and only if its dual lamination in $\L(\phi^\inv)$ is topmost. We extend this result to show that the duality relation respects inclusion:

\begin{lemma}\label{containmentSymmetry} If $\Lambda^\pm_i $ and $\Lambda^\pm_j$ are dual lamination pairs for $\phi \in \Out(F_n)$ then $ \Lambda^+_i \subset \Lambda_j^+$ if and only $\Lambda^-_i \subset \Lambda^-_j$. 
\end{lemma}

\begin{proof} Passing to a power we may assume $\phi$ is rotationless. Let 
\begin{align*}
[F^i] &= \A_\supp(\Lambda^+_i) = \A_\supp(\Lambda^-_i) \\
[F^j] &= \A_\supp(\Lambda^+_j) = \A_\supp(\Lambda^-_j)
\end{align*} 
If $[F^j]$ does not carry $[F^i]$ then $\Lambda^+_j \not \supset \Lambda_i^+$ and $\Lambda^-_j \not \supset \Lambda_i^-$. We may therefore assume that $[F^j]$ carries $[F^i]$. Restricting $\phi$ to $F^j$ we may assume that $F^j = F_n$, which implies that $\Lambda^+_j$ is topmost (in $\L(\phi)$) and $\Lambda^-_j$ is topmost (in $\L(\phi^\inv)$). By Corollary~6.0.11 of \BookOne\ $\Lambda^+_i$ is topmost if and only if $\Lambda^-_i$ is topmost. In this topmost case, $\Lambda^+_j \not \supset \Lambda_i^+$ and $\Lambda^-_j \not \supset \Lambda_i^-$. We may therefore assume that neither $\Lambda^+_i$ nor $\Lambda^-_i$ is topmost, in which case it suffices to show that $\Lambda^+_j \subset \Lambda_i^+$ and that $\Lambda^-_j \subset \Lambda_i^-$; the two cases are similar, so it suffices to assume that $\Lambda^+_j \not \supset \Lambda_i^+$ and argue to a contradiction. Since $\Lambda^+_i$ is not topmost and is not contained in $\Lambda_j^+$, there is a topmost lamination $\Lambda^+_k \in \L(\phi)$ such that (A)~$\Lambda^+_i \subset \Lambda_k^+$ and (B)~$\Lambda_k^+\not \subset \Lambda_j^+$. Choose a \ct\ $f \from G \to G$ representing $\phi$ and let $H_r$ be the \eg\ stratum corresponding to $\Lambda^+_k$. Corollary~\ref{CorBottommost} together with (A) implies that there is no indivisible Nielsen path of height $r$, and (B) implies that $\Lambda_j^+$ is not weakly attracted to $\Lambda_k^+$, that is, no generic leaf of $\Lambda_j^+$ is weakly attracted to a generic leaf of $\Lambda_k^+$. The Weak Attraction Theorem 6.0.1 coupled with Remark 6.0.2 of \BookOne\ applied to the topmost lamination $\Lambda_k^+$ then implies that $\Lambda_j^+$ is carried by a proper free factor of $F_n$, which contradicts our assumption that $F^j = F_n$.
\end{proof}

\subsection{Geometric strata and their laminations}
\label{SectionGeometric}
In this section we review results about geometric strata and their laminations contained in \BookOne, and we give some slight improvements on those results. Proposition~\ref{PropGeomLams} below, which is implicit in the proof of the geometric case of Proposition 6.0.8 of \BookOne, relates the lamination pair of a geometric stratum to the stable and unstable laminations of the associated pseudo-Anosov surface homeomorphism.

\paragraph{Geodesic laminations and pseudo-Anosov mapping classes.} First we review pseudo-Anosov homeomorphisms and their stable and unstable geodesic laminations, collecting what we need in Fact~\ref{FactPsAnWeakAttr}, whose proof will depend on the Nielsen-Thurston theory. Our concern is \emph{solely} with surfaces having nonempty boundary and we will state results only in that case. We shall take the liberty to state definitions and results in a manner that illuminates the connection with geometric strata and laminations in free groups. The main reference for this section is \cite{CassonBleiler},  with various citations also made to \cite{FLP} and \cite{Miller:Nielsen}. 

Let $S$ be a compact hyperbolic surface with totally geodesic boundary~$\bdy S \ne \emptyset$. Identify the universal cover $\wt S$ with the convex hull of a Cantor set embedded in $\bdy_\infinity \hyp^2$ and denoted $\bdy_\infinity \wt S$ (identified with the Gromov boundary of $\wt S$). The deck transformation action of $\pi_1 S$ on $\wt S$ extends to a properly discontinuous, free, isometric action on $\hyp^2$ that is cocompact on $\wt S$ but not on $\hyp^2$. Let $\wt\B(\wt S) = (\bdy\wt S \cross \bdy \wt S - \Delta) / (\Z/2)$ be the \emph{space of geodesics} in $\wt S$, equipped with the \emph{weak topology} induced from $\bdy\wt S$. The action of $\pi_1 S$ on $\wt\B(\wt S)$ has compact quotient $\B(S)$, identified with the space of locally geodesic maps $\reals \to S$ modulo isometric reparameterization. There exists $\delta>0$ such that the topology on $\B(S)$ has a basis with one element for each $\gamma \in \B(S)$ and each finite subarc $\alpha$ of $\gamma$, consisting of all $\gamma' \in \B(S)$ for which there exists a finite subarc $\alpha'$ of $\gamma'$ such that $\alpha$ and $\alpha'$ are \emph{$\delta$-parallel}, meaning that there is a map $[a,b] \cross [c,d] \to Y$ whose restriction to $[a,b] \cross c$ parameterizes~$\alpha$, whose restriction to $[a,b] \cross d$ parameterizes~$\alpha'$, and whose restriction to each vertical fiber $s \cross [c,d]$ has length $<\delta$. A geodesic in $S$ is \emph{simple} if it has no transverse self-crossings, so it can be either an embedding or a covering of a simple closed geodesic.

A \emph{geodesic lamination} is a subset $\Lambda \subset \B(S)$ consisting of a disjoint family of simple geodesics contained in $S - \bdy S$ that are called the \emph{leaves} of $\Lambda$, such that $\Lambda$ is closed in $\B(S)$, equivalently the union of the leaves of $\Lambda$ is a closed subset of $S$. Its total lift $\wt \Lambda \subset \wt \B(\wt S)$ is a $\pi_1 S$-invariant, closed, disjoint union of geodesics in $\hyp^2$. The closures of the components of $\wt S - \wt \Lambda$ are called the \emph{principal regions} of $\wt \Lambda$, and their images downstairs in $S$ are called the principal regions of $\Lambda$. There are only finitely many \emph{boundary leaves} of $\Lambda$, which means leaves on the boundary of some principal region. A principal region $R$ is an \emph{ideal polygon} if any of its connected lifts to $\wt S$ is the convex hull in $\hyp^2$ of a finite subset of $\bdy_\infinity \wt S$, in which case the restricted covering map $\wt R \to R$ is an injective immersion. A principal region $R$ is a \emph{peripheral ideal crown} if $R$ has one compact boundary geodesic $c \subset \bdy S$, and for any connected lift $\wt R \subset \wt S$ with lifted boundary geodesic $\wt c \subset \bdy\wt S$, its stabilizer $\Stab(\wt R)$ in $\pi_1 S$ is equal to the infinite cyclic group $\Stab(\ti c)$, and $\wt R$ is the convex hull in $\hyp^2$ of the union of the two points $\bdy_\infinity \ti c$ with a discrete, $\Stab(\ti c)$-invariant, cofinite subset of $\bdy_\infinity \wt S - \bdy_\infinity c$; in this case the induced map $\wt R / \Stab(\ti c) \to S$ is an injective immersion. We say $\Lambda \subset S$ \emph{fills $S$} if no leaf is compact and every principal region is an ideal polygon or a peripheral ideal crown; if this is so then every leaf of $\Lambda$ is dense in $\Lambda$, and every closed geodesic in the interior of $S$ intersects $\Lambda$ nontrivially and transversely. 

The mapping class group $\MCG(S)$ is the quotient of the group of homeomorphisms $\Homeo(S)$ by the normal subgroup of homeomorphisms isotopic to the identity (isotopies need not be stationary on $\bdy S$). For any $\psi \in \MCG(S)$ and a representative $\Psi \in \Homeo(S)$, any lift $\wt\Psi \from \wt S \to \wt S$ extends continuously to a homeomorphism $\wh\Psi \from \bdy\wt S \to \bdy\wt S$ which respects the decomposition of $\bdy\wt S$ into $\pi_1 S$ orbits. This induces in turn a homeomorphism of $\wt\B(\wt S)$ that respects $\pi_1 S$ orbits and so descends to a homeomorphism $\psi \from \B(S) \to \B(S)$ depending only on $\psi$. This action respects simplicity and disjointness of geodesics and so induces an action on the set of geodesic laminations. The map $\wh\Psi$ extends continuously to all of $\bdy_\infinity\hyp^2$ in a way which depends on the choice of $\Psi$ but respects $\pi_1 S$ orbits --- letting $p \from \bdy_\infinity \hyp^2 \to \bdy_\infinity\wt S \union \bdy\wt S$ be the extension of the identity map on $\bdy_\infinity \wt S$ which projects each point of $\bdy_\infinity \hyp^2 \setminus  \bdy_\infinity \wt S$ orthogonally to $\bdy \wt S$, define $\wh\Psi(x)$ so that $p(\wh\Psi(\xi)) = \wt\Psi(p(\xi))$.

A \emph{proper geodesic} in $S$ is any of the following: a bi-infinite geodesic, including the cover of a closed geodesic such as a component of $\bdy S$; a finite geodesic arc which hits $\bdy S$ orthogonally at its two endpoints; a geodesic ray which hits $\bdy S$ orthogonally at its one endpoint. The set of proper geodesics, equipped with the weak topology, is a compact space denoted $\wh\B(S)$; it contains $\B(S)$ as a closed subset. The action of $\MCG(S)$ on $\B(S)$ extends to a homeomorphic action on $\wh\B(S)$: if $\gamma \in \wh\B(S)$ is a finite geodesic arc or geodesic ray then $\psi(\gamma)$ is the unique finite arc or geodesic ray in $\wh\B(S)$, respectively, such that for some (any) $\Psi\in\Homeo(S)$ representing $\psi$, there is a \emph{proper homotopy} from $\Psi(\gamma)$ to $\psi(\gamma)$, meaning a homotopy that keeps endpoints in $\bdy S$ and that moves all points along rectifiable paths of uniformly bounded length.

Thurston's classification theorem \cite{CassonBleiler,FLP}, says that for any $\psi \in \MCG(S)$, either $\psi$ has finite order, or there exists a geodesic lamination with compact leaves that is invariant under $\psi$, or $\psi$ is a pseudo-Anosov mapping class. The following fact contains everything we will need to know about pseudo-Anosov mapping classes.

\begin{fact}\label{FactPsAnWeakAttr}
For any pseudo-Anosov $\psi \in \MCG(S)$ there exists a pair of geodesic laminations $\Lambda^s \ne \Lambda^u$ that fill $S$ and are invariant under $\psi$ such that for any proper geodesic $\gamma \in \wh\B(S)$, exactly one of the following holds:
\begin{enumerate}
\item \label{ItemPsAnWeakLeaf}
$\gamma$ is a leaf of $\Lambda^s$.
\item \label{ItemPsAnWeakPrinc}
$\gamma$ is a bi-infinite geodesic or proper ray that is contained in a principal region of $\Lambda^s$; this includes the case that $\gamma$ covers a component of $\bdy S$.
\item \label{ItemPsAnWeakAttr}
$\psi^i(\gamma)$ converges weakly to each leaf of $\Lambda^u$ as $i \to +\infinity$.
\end{enumerate}
\end{fact}

This fact is a straightforward consequence of standard results in the Nielsen-Thurston theory. We sketch a proof for completeness, referring to Miller's paper \cite{Miller:Nielsen} for Nielsen-Thurston theory; see sections 2--8 of that paper for the closed case, and section~9 for the bounded case. See also Chapter~11 of \cite{FLP} for the measured foliation theory.

\begin{proof} First we describe source-sink dynamics on the circle at infinity of $\hyp^2$.

Choose $\Psi \in \Homeo(S)$ representing $\psi$ which preserves both $\Lambda^s$ and $\Lambda^u$, and so $\Psi$ preserves the Cantor set $C = \Lambda^s \intersect \Lambda^u$; see Theorem~9 of \cite{Miller:Nielsen}. The set of $\Psi$-periodic points in $C$ is dense, with two on each boundary leaf of $\Lambda^s$ and at most one on every other leaf; see \cite{FLP} Expos\'e~10 for the analogous statement in terms of measured foliations, which translates into the statement here using the standard correspondence between geodesic laminations and measured foliations \cite{CassonBleiler,Miller:Nielsen}. Choose a periodic $x \in C$ not contained in a boundary leaf. Let $\ell^s \subset \Lambda^s$, $\ell^u \subset \Lambda^u$ be the leaves containing $x$. Let $\Psi^a$ be a positive power that fixes $x$ and preserves the orientations of $\ell^s$ and $\ell^u$. Choose a lift $\ti x \in \wt S$ of $x$ and a lift $\wt\Psi^a \from \wt S \to \wt S$ of $\Psi^a$ that fixes~$\ti x$. Lift $\ell^s$, $\ell^u$ to leaves $\ti\ell^s \subset \wt\Lambda^s$, $\ti\ell^u \subset \wt \Lambda^u$ passing through~$\ti x$ and invariant under $\wt\Psi^a$, so $\bdy\ti\ell^s, \bdy\ti\ell^u \subset \bdy_\infinity \wt S$ are fixed by $\wh\Psi^a$. 

By Nielsen-Thurston theory \cite{Miller:Nielsen}, the homeomorphism $\wh\Psi^a$ acts on $\bdy\hyp^2$ with source-sink dynamics: each of the four components of $\bdy\hyp^2 - \left(\bdy \ell^s \union  \bdy \ell^u\right)$ has one endpoint in $\bdy\ell^s$ and the other in $\bdy\ell^u$, and every orbit converges to the $\bdy\ell^s$ endpoint in negative time and to the $\bdy\ell^u$ endpoint in positive time.

Let $\gamma$ be a proper geodesic in $S$. Assuming that \pref{ItemPsAnWeakLeaf}---\pref{ItemPsAnWeakPrinc} do not hold, and using that $\Lambda^s$ fills, it follows that $\gamma$ has a transverse intersection point with $\Lambda^s$. Since $\ell^s$ is dense in $\Lambda^s$, it follows that $\gamma$ and $\ell^s$ intersect transversely, so there is a lift $\ti\gamma$ that intersects $\ti\ell^s$ transversely. For each $j \ge 0$ let $\ti\gamma_j$ be the lift of $\psi^{aj}(\gamma)$ that is obtained from $\wt\Psi^{aj}(\ti\gamma)$ by a proper homotopy.

Let $\bdy\ti\gamma = \{p,q\} \subset \bdy_\infinity \wt S \union \bdy_\infinity \wt S \subset \hyp^2 \union \bdy\hyp^2$. Define $I(p), I(q) \subset \bdy\hyp^2$, each a nonempty but possibly degenerate compact subinterval of $\bdy\hyp^2$, as follows. If $p \in \bdy\hyp^2$ then $I(p) = \{p\}$. If $p \in c$, a component of $\bdy\wt S$, then $I(p)$ is the closure of the component of $\bdy\hyp^2 - \bdy c$ that is disjoint from $\bdy_\infinity \wt S$. The set $I(q)$ is defined similarly. Notice that $I(p)$ and $I(q)$ are contained in the closures of opposite components $J(p), J(q)$ of $\bdy\hyp^2 - \bdy \ti\ell^s$, because $\ti\gamma$ crosses $\ti\ell^s$ transversely. Moreover, $I(p)$, $I(q)$ are actually contained in $J(p)$, $J(q)$ themselves, not just their closures, because the endpoints of $I(p)$, $I(q)$ are endpoints of components of $\bdy\wt S$, which cannot be endpoints of $\ti\ell^s$ because the principal region of $\Lambda^s$ incident to each component of $\bdy S$ is a peripheral ideal crown.

By source-sink dynamics, the intervals $J(p)$ and $J(q)$ are invariant by $\wh\Psi^a$, and denoting $\eta_p = \bdy\ell^u \intersect J(p)$ and $\eta_q = \bdy\ell^u \intersect J(q)$, the forward orbit of each point of $J(p)$ acted on by $\wh\Psi^a$ converges to $\eta_p$, and the forward orbit of each point of $J(q)$ converges to $\eta_q$. This convergence is uniform on each the subsets $I(p) \subset J(p)$, $I(q) \subset J(q)$, because these subsets are compact. It follows that if $(\rho_j)$ is a sequence of bi-infinite geodesics in $\hyp^2$ such that $\rho_j$ has one endpoint in $\wh\Psi^{aj}(J(p))$ and the other in $\wh\Psi^{aj}(J(q))$ then $\rho_j$ converges, in the Hausdorff topology on closed sets, to the geodesic $\wt\ell^u$. Letting $\rho_j$ be the bi-infinite extension of $\ti\gamma_j$, it follows that for each $\delta>0$ and $L>0$ there exists $J > 0$ such that if $j>J$ then $\ti\gamma_j$ has a subsegment of length $\ge L$ which is $\delta$ parallel to a subsegment of $\wt\ell^u$, from which it follows that the sequence $\psi^{aj}(\gamma)$, $j \ge 0$, converges weakly to $\ell^u$. Since each leaf of $\Lambda^u$ is dense, it follows that the sequence $\psi^{aj}(\gamma)$ converges weakly to each leaf of $\Lambda^u$. Since $\Lambda^u$ is invariant under~$\psi$, it follows that the sequence $\psi^i(\gamma)$, $i \ge 0$, converges weakly to each leaf of $\Lambda^u$.
\end{proof}





\begin{definition}\label{DefGeometricStratum} \textbf{Geometric Strata.}
Suppose that $\fG$ is a \ct\ representing a rotationless $\phi \in \Out(F_n)$ and that $H_r \subset G$ is an \eg\ stratum. Following Definition~5.1.4 and Section~5.3 of \BookOne\ we say that $H_r$ is a \emph{geometric stratum} if the map $f \restrict G_r$ has a \emph{geometric model}, consisting of a compact $2$-complex $Y$, an embedding $G_r \inject Y$, a deformation retraction $d \from Y \to G_r$, and a homotopy equivalence $h \from Y \to Y$, with the following properties:
\begin{enumerate}
\item \label{ItemY}
There exists $m \ge 0$, a compact, connected, hyperbolic surface $S$ with totally geodesic boundary components $\bdy S = \bigcup_{0 \le i \le m} \bdy_i S$, a collection of annuli $\A = \bigcup_{1 \le i \le m} \, \A_i$, and a collection of circuits $\alpha_i$ in $G_{r-1}$, $1 \le i \le m$, such that $Y$ is obtained from the disjoint union of $S$, $\A_1,\ldots,\A_m$ and $G_{r-1}$ by attaching one component of $\bdy \A_i$ to $\bdy_i S$ and the other to $\alpha_i$. (See below under the heading ``A metric on $Y$'' for a further restriction on the hyperbolic structure of $S$.)
\item $h \from Y \to Y$ represents $\phi$ in the sense that $f \composed d$ and $d \composed h$ are homotopic. 
\item $h \restrict G_{r-1} = f \restrict G_{r-1}$ 
\item \label{ItemStratumPsAn}
The map $\Psi = h \restrict S \from S \to S$ is a homeomorphism whose mapping class $\psi \in \MCG(S)$ is pseudo-Anosov.
\end{enumerate} 
Let $\Lambda^s, \Lambda^u \subset S$ denote the stable and unstable geodesic laminations of $\psi$. Let $D \from S \to G$ be the composition $S \subset Y \xrightarrow{d} G_r \subset G$, inducing a monomorphism $D_* \from \pi_1 S \to \pi_1 G \approx F_n$ that is well-defined up to inner automorphism of $F_n$. Let $[\bdy_i S]$ denote the conjugacy class in $F_n$ of $D_*(\bdy_i S)$, and let $[\pi_1(S)]$ denote the conjugacy class of the subgroup $D_*(\pi_1(S))$. 

Items \pref{ItemY}---\pref{ItemStratumPsAn} constitute Definition~5.1.4 of \BookOne, but for our definition we incorporate one more item which is verified in the proof of Proposition 5.3.1, where the items \pref{ItemY}---\pref{ItemStratumPsAn} are also verified; see specifically the bottom of page 573, the diagram at the top of page 574, and the proof beginning on page~578.
\begin{enumeratecontinue}
\item \label{ItemRhoMappingCylinder} 
With respect to a parameterization $\rho \from [0,1] \to G_r$, there exists a homeomorphism $k \from C_\rho \to Y$ where $C_\rho$ is the mapping cyclinder
$$C_\rho = \bigl[([0,1] \cross [0,1] )\disjunion G_r\bigr] \biggm/ (t,0) \sim \rho(t)
$$
such that $k$ takes $(0 \cross [0,1]) \union ([0,1] \cross 1) \union (1 \cross [0,1])$ to $\bdy_0 S$, $k$ restricts to the identity on $G_r$, and there exists a deformation retraction $d' \from [0,1] \cross [0,1] \to [0,1] \cross 0$ such that $k d' = d k$.
\end{enumeratecontinue}
By Fact~\ref{FactEGStrata}~\pref{ItemEachTwice}, the base point of $\rho_r$ is not contained in $G_{r-1}$, and combining with \pref{ItemRhoMappingCylinder} we obtain the following as well:
\begin{enumeratecontinue}
\item \label{ItemInteriorBasePoint}
$G_r \cap \partial_0 S$ is a single point $v_S$, it is the base point of the closed Nielsen path $\rho_r$, $d(v_S)=v_S$, and there is a parameterization $\rho_S$ of $\bdy_0 S$ based at $v_S$ such that $d\composed \rho_S= \rho_r$.
\end{enumeratecontinue} 
Item~\pref{ItemInteriorBasePoint} will be used in Section~\ref{SectionAsubNA}, in the definition of the geometric model for $\A_\na\Lambda$.
\end{definition}


\begin{fact}\label{FactEGStrata} If $f \from G \to G$ is a \ct\ representing a rotationless $\phi \in \Out(F_n)$, and if $H_r \subset G$ is an \eg\ stratum, then:
\begin{enumerate}
\item \label{ItemFirstLast}
There is at most one indivisible Nielsen path $\rho_r$ of height $r$ (up to orientation reversal). The first and last edges of $\rho_r$ are distinct edges of $H_r$.
\item \label{ItemSomeOnce}
If $H_r$ is not geometric and if $\rho_r$ is an indivisible Nielsen path of height $r$ then $\rho_r$ is not closed, $\rho_r$ crosses some edge of $H_r$ exactly once, and at least one endpoint of $\rho_r$ is not in $G_{r-1}$.
\item \label{ItemEachTwice}
If $H_r$ is geometric then there is an indivisible Nielsen path $\rho_r$ of height $r$, and furthermore: $\rho_r$ is a closed path with basepoint disjoint from $G_{r-1}$; the circuit determined by $\rho_r$ is freely homotopic to $d(\bdy_0 S)$ (following the notation of Definition~\ref{DefGeometricStratum}); and $\rho_r$ crosses each edge of $H_r$ exactly twice. 
\end{enumerate}
\end{fact}


\begin{proof} Corollary~4.20 of \recognition\ says that, under our hypotheses, conclusions eg-(i), (ii), and (iii) of Theorem~5.1.5 of \BookOne\ hold. Item~\pref{ItemFirstLast} is just eg-(i). Item~\pref{ItemSomeOnce} is just eg-(ii), together with an application of Lemma~5.1.7 of \BookOne\ to conclude that $\rho_r$ has at least one endpoint that is not in~$G_{r-1}$; the hypotheses of Lemma~5.1.7 are checked by applying Item~\pref{ItemFirstLast}, (Filtration), Fact~\ref{FactEGAperiodic}, and the conclusion of eg-(ii) that $\rho_r$ crosses some edge of $H_r$ exactly once.

Item~\pref{ItemEachTwice} is just eg-(iii), together with the following argument to conclude that $\rho_r$ crosses each edge of $H_r$ exactly twice. Lemma 4.18 of \recognition\ says that, as a consequence of (\eg\ Nielsen paths), the conclusions of \cite{BestvinaHandel:tt} Theorem~5.15 hold. These conclusions imply that either $\rho_r$ crosses each edge of $H_r$ exactly twice in which case $\rho_r$ is closed, or $\rho_r$ crosses some edge of $H_r$ exactly once in which case $\rho_r$ is not closed (this implication is explained in the proof of \recognition\ Corollary~4.20).
%


\end{proof}

The following result is based on Section~6.0.8 of \BookOne. The proof will repeat material from that section which will be useful in what follows. Recall that a lamination is \emph{minimal} if every leaf is dense.


\begin{proposition} 
\label{PropGeomLams} Suppose that $\phi,\phi^\inv \in \Out(F_n)$ are both rotationless and consider a dual lamination pair $\Lambda^\pm \in \L^\pm(\phi)$. Let $f \from G \to G$ be a \ct\ representing $\phi \in \Out(F_n)$, let $H_r \subset G$ be the \eg\ stratum associated to $\Lambda^+$. If $H_r$ is a geometric stratum then, adopting the notation above, we have:
\begin{enumerate}
\item \label{ItemStableIsMinus}
$d_\#(\Lambda^u) = \Lambda^+$, and $d_\#(\Lambda^s) = \Lambda^-$.
\item \label{ItemLambdasMinimal}
$\Lambda^+$ and $\Lambda^-$ are minimal.
\item \label{ItemAllGeneric}
All leaves of $\Lambda^+$ and $\Lambda^-$ are generic.
\end{enumerate}
\end{proposition}

By applying \pref{ItemLambdasMinimal} we immediately obtain another proof, independent of Corollary~\ref{CorBottommost}, that $\Lambda^+$ is bottommost.

\begin{proof} We adapt from the proof of Proposition~6.0.8 of \BookOne\ a metric on the 2-complex $Y$ and a class of ``quasilines'' in $Y$. Start with a geodesic metric on the graph $G_{r-1}$. This induces a geodesic metric on one boundary circle of each $\A_i$, which extends to a Euclidean product metric on the annulus $\A_i$, which induces a geodesic metric on each circle $\bdy_i S$, which extends to a hyperbolic structure on $S$ with totally geodesic boundary; henceforth we may assume that this hyperbolic structure is the one referred to in Fact~\ref{FactEGStrata}. The homotopy equivalences $h \from Y \to Y$, $d \from Y \to G_r$, $f \from G_r \to G_r$ have quasi-isometric lifts $\ti h$, $\ti d$, $\ti f$ to universal covers with continuous homeomorphic extensions $\hat h$, $\hat d$, $\hat f$ to Gromov boundaries, such that $\hat d \composed \hat h = \hat f \composed \hat d$. 

Define an \emph{$S$-proper quasiline in $Y$} to be a continuous map $\reals \to Y$ which intersects $\bdy S$ transversely and is a concatenation of pieces that alternate between proper geodesics in $S$ and geodesics in $G_{r-1} \union A$. Every $S$-proper quasiline is disjoint from~$\bdy_0 S$. Any pair $\xi \ne \eta \in \bdy \wt Y \approx \bdy \wt G_r$ is connected either by a geodesic which is a lift of a component of $\bdy S$ or by a unique quasigeodesic, with uniform quasigeodesic constants, which is a lift of an $S$-proper quasiline: start with a geodesic $\gamma$ in $Y$ which lifts to a geodesic connecting $\xi,\eta$, and if no choice of $\gamma$ covers a component of $\bdy S$, homotope $\gamma$ properly staying transverse to $\bdy S$ so as to replace each maximal subsegment of $\gamma$ in $S$ by the corresponding proper geodesic in $S$, and then properly homotope so as to replace each remaining subsegment of $\gamma$ by the corresponding geodesic in $G_{r-1} \union A$. 

The set of $S$-proper quasilines has the \emph{weak topology}: it is the quotient topology inherited from the usual topology on $(\bdy_\infinity \wt Y \cross \bdy_\infinity \wt Y - \Delta) / \Z/2$; also, there is a constant $\delta>0$ and a basis element for each $S$-proper quasiline $\gamma$ and each finite subsegment $\alpha$, consisting of all $S$-proper quasilines $\gamma'$ having a finite subsegment $\alpha'$ such that $\alpha$ and $\alpha'$ are $\delta$-parallel in $Y$. The map $h \from Y \to Y$ induces a self-homeomorphism $\rho \mapsto h_\#(\rho)$ of the set of $S$-proper quasilines in $Y$, and the map $d \from Y \to G_r$ induces a homeomorphism $\rho \mapsto d_\#(\rho) \in \B(G_r)$ from the set of $S$-proper quasilines in $Y$ to the set of lines in $G_r$ that are not bi-infinite iterates of any of $\rho_r, \alpha_1,\ldots,\alpha_m$. The inverse of the map $\rho \mapsto d_\#(\rho)$ is denoted $\sigma \mapsto \sigma^*$. In each case these maps are well-defined by requiring that the endpoints of lifts to the universal cover correspond under the continuous extension to infinity of the lift of the map. Note that $d_\#$ preserves height in the following sense: for each $S$-proper quasiline $\gamma$, $d_\#(\gamma)$ has height $r$ in $G_r$ if and only if at least one piece of $\gamma$ is a proper geodesic in $S$.

\bigskip

Item~\pref{ItemLambdasMinimal} is an immediate consequence of~\pref{ItemStableIsMinus}, continuity of $d_\#$, and minimality of $\Lambda^s,\Lambda^u$. Item~\pref{ItemAllGeneric} is an immediate consequence of~\pref{ItemLambdasMinimal}.

\bigskip

We turn to the proof of~\pref{ItemStableIsMinus}. By Fact~\ref{FactPsAnWeakAttr}, $\Lambda^u$ is a minimal $h_\#$-invariant lamination and it has a neighborhood in the weak topology on $S$-proper quasilines such that each point in this neighborhood is weakly attracted to $\Lambda^u$. Using continuity of $d_\#$ and its inverse ${}^*$ (see also the proof of Proposition~6.0.8 of \BookOne), it follows that for any geodesic line $\sigma$ in $G_r$ which is not the bi-infinite iterate of any of $\rho_r, \alpha_1, \ldots, \alpha_m$, and any geodesic line $\lambda$ in the interior of~$S$, $\sigma^*$ is weakly attracted to $\lambda$ under iteration of $h_\#$ if and only if $\sigma$ is weakly attracted to $d_\#(\lambda)$ under iteration of $f_\#$. It then follows that $d_\#(\Lambda^u)$ is a minimal closed subset of $\B(G)$, and it has an attracting neighborhood for the action of $f_\#$. Also, $d_\#(\Lambda^u)$ is not carried by a free factor of rank one, for if it were then the leaves of $\Lambda^u$ would be closed geodesics in $S$, which they are not. By Definition~3.1.5 of \BookOne, $d_\#(\Lambda^u)$ is an element of $\L(\phi)$. Since the leaves of $d_\#(\Lambda^u)$ have height $r$ in $G_r$, 
it follows that $d_\#(\Lambda^u) = \Lambda^+$, which is the first part of~\pref{ItemStableIsMinus}.

To prove the second part of \pref{ItemStableIsMinus}, by Fact~\ref{FactPsAnWeakAttr} the leaves of $ \Lambda^s$ are not weakly attracted to $ \Lambda^u$ under iteration by $h_\#$, so the leaves of the minimal lamination $d_\#(\Lambda^s)$ are not weakly attracted to $\Lambda^+$ under iteration by~$f$. Since these leaves have height~$r$, by applying the Weak Attraction Theorem~6.0.1 and Remark~6.0.2 of \BookOne\ to the restriction of $\phi$ to $\A_\supp(G_r)$, it follows that $d_\#(\Lambda^s) = \Lambda^-$.

\end{proof}

The following proposition shows that geometricity is an invariant of a dual lamination pair, well-defined independent of the choice of \ct\ and corresponding \eg\ stratum. Because of this proposition, for any $\phi \in \Out(F_n)$ and any $\Lambda^\pm \in \L^\pm(\phi)$, we may define $\Lambda^+$, $\Lambda^-$, or the pair $\Lambda^\pm$ to be \emph{geometric} if condition \pref{ItemBoundaryExists} of the proposition holds.

\begin{proposition}\label{PropGeomEquiv}
For any $\phi \in \Out(F_n)$, any $\Lambda^\pm \in \L^\pm(\phi)$, and any \ct\ $f \from G \to G$ representing a positive power of $\phi$, if $H_r \subset G$ is the \eg\ stratum associated to $\Lambda$ then the following are equivalent:
\begin{enumerate}
\item \label{ItemStratumIsGeom}
$H_r$ is a geometric stratum.
\item \label{ItemBoundaryExists}
There exists a finite set of conjugacy classes $\{[c_1],\ldots,[c_k]\}$ such that the free factor support of this set equals the free factor support of $\Lambda^+$, and none of $[c_1],\ldots,[c_k]$ are weakly attracted to $\Lambda^+$.
\end{enumerate}
Furthermore, if \pref{ItemStratumIsGeom}, \pref{ItemBoundaryExists} hold then for any \ct\ $f' \from G' \to G'$ representing a positive or negative power of~$\phi$ the following holds:
\begin{enumeratecontinue}
\item \label{ItemGeomInv}
If $H'_s \subset G'$ is the \eg\ stratum associated to $\Lambda^+$ or $\Lambda^-$ then $H'_s$ is geometric. 
\item \label{ItemSameNielsen}
If $[G_r] = [G'_s]$ and $[G_{r-1}] = [G'_{s-1}]$, and if $\rho_r$, $\rho'_s$ are the closed indivisible Nielsen paths in $G$, $G'$ of height $r,s$, respectively, then the conjugacy classes $[\rho_r]$, $[\rho'_s]$ are the same up to reversal.
\end{enumeratecontinue}
\end{proposition}

\begin{proof} Item \pref{ItemGeomInv} is proved as follows. The property stated in \pref{ItemBoundaryExists} is independent of the choice of \ct\ $f \from G \to G$ representing a positive power of~$\phi$, so~\pref{ItemStratumIsGeom} holds independent of $f$, proving \pref{ItemGeomInv} for positive powers. For negative powers, let $[F]$ be the free factor support of $\Lambda^+$ and $\Lambda^-$. Assuming \pref{ItemBoundaryExists} holds, each of $[c_i]$ is carried by $[F]$. Applying \BookOne\ Corollary~6.0.10, $c_i$ is weakly attracted to $\Lambda^+$ by iteration of $\phi \restrict F$ if and only if it is weakly attracted to $\Lambda^-$ by iteration of $\phi^\inv \restrict F$. The property stated in~\pref{ItemBoundaryExists} is therefore true for $\phi^\inv$ and $\Lambda^-$, and is independent of the choice of \ct\ representing negative powers of $\phi$, so \pref{ItemStratumIsGeom} holds for all such \cts, proving \pref{ItemGeomInv} for negative powers. 

To prove item~\pref{ItemSameNielsen}, by \BookOne\ Theorem~6.0.1 and Corollary~6.0.10, the conjugacy classes represented by $\rho_r$, $\rho'_s$ have the same characterization and are therefore equal (up to reversal): it is the unique conjugacy class that is carried by $[G_r]=[G'_s]$ but not by $[G_{r-1}]=[G'_{s-1}]$, and that is not weakly attracted to $\Lambda^+$ by iterating $\phi$ nor to $\Lambda^-$ by iterating $\phi^\inv$.

To prove \pref{ItemBoundaryExists}$\implies$\pref{ItemStratumIsGeom}, suppose that $H_r$ is not geometric. Consider a finite set of conjugacy classes $\{[c_1],\ldots,[c_k]\}$, none of which is weakly attracted to $\Lambda$. Applying the Weak Attraction Theorem~6.0.1 of \BookOne\ to the restriction of~$f$ to the component of $G_r$ containing $H_r$, it follows that if each of $\{[c_1],\ldots,[c_k]\}$ is supported by $[G_{r}]$ then each is supported by $[G_{r-1}]$, and so the free factor support of this set is not equal to the free factor support of $\Lambda$.

To prove \pref{ItemStratumIsGeom}$\implies$\pref{ItemBoundaryExists}, suppose that $H_r$ is geometric. We adopt the notation of Definition~\ref{DefGeometricStratum}. Consider the following free factor systems: $\A_\supp(\Lambda)$; $\A_\supp(\pi_1 S) = $ the free factor support of the set of lines carried by the subgroup system $[\pi_1(S)]$; and $\A_\supp(\bdy S) = $ the free factor support of $\{[bdy_0 S], [\bdy_1 S], \ldots, [\bdy_m S]\}$. We prove that $\A_\supp(\Lambda) \sqsubset  \A_\supp(\pi_1 S) \sqsubset \A_\supp(\bdy S) \sqsubset \A_\supp(\Lambda)$, and therefore these free factor systems are all equal. The inclusion $\A_\supp(\Lambda) \sqsubset  \A_\supp(\pi_1 S)$ is obvious.

The inclusion $\A_\supp(\pi_1 S) \sqsubset \A_\supp(\bdy S)$ is proved by a surgery argument following Stallings. Let $K$ be a marked graph with a subgraph $L$ such that $[L] = \A_\supp(\bdy S)$. Let $M$ be the set of midpoints of edges of $K-L$. Choose a map $m \from S \to K$ such that $m(\bdy S) \subset L$ and such that $m_* \from \pi_1 S \to \pi_1 K \approx F_n$ is conjugate in $F_n$ to $D_*$; also, choose a perturbation rel $\bdy S$ so that $m$ is transverse to $M$, so $m^\inv(M)$ is a pairwise disjoint collection of embedded circles in the interior of~$S$. Make these choices so that $m^\inv(M)$ has the fewest components. If $m^\inv(M) = \emptyset$ then $m(S) \subset K-M$, and so $m$ is homotopic to a map $S \mapsto L$, proving that $\A_\supp(\pi_1 S) \sqsubset [L]=\A_\supp(\bdy S)$. If $m^\inv(M) \ne \emptyset$ then, since $m_*$ is injective, each component of $m^\inv(M)$ is homotopically trivial and so bounds a closed disc embedded in $S$. Let $D$ be an innermost such disc. Let $D'$ be a slightly larger disc whose interior contains~$D$. Homotope $m$, by a homotopy supported on $D'$, to a map $m'$ such that $m'(D')$ is disjoint from $M$, and so $m'{}^\inv(M)$ has one fewer component than $m^\inv(M)$, contradicting minimality.

To prove $\A_\supp(\bdy S) \sqsubset \A_\supp(\Lambda)$, first we claim that for any free factor $F \subgroup F_n$ and $g \in F_n - F$, the subsets $\bdy F$ and $\bdy(g F g^\inv)$ are disjoint in~$\bdy F_n$. To see why, from Fact~\ref{FactGrushko} it follows that $F \intersect gFg^\inv$ is trivial. Let $T$ be a free, minimal, simplicial $F_n$ tree, and let $S \subset T$ be the minimal $F$-invariant subtree, so $gS$ is the minimal $gFg^\inv$ invariant subtree. The intersection $S \intersect gS$ is a finite tree, for if the number of vertices in $S \intersect gS$ exceeds the square of the number of vertex orbits of $F$ acting on $S$ then there exist two vertices $v,w \in S \intersect gS$ contained in the same $F$ orbit and the same $gFg^\inv$ orbit, producing a nontrivial element of $F \intersect gFg^\inv$. The claim follows.

Let $c$ be a component of $\bdy S$, let $R$ be the peripheral crown of $\Lambda$ containing $c$, and let $\wt R$ be a connected lift of $R$ to $\wt S$. As with the boundary of any finitely generated subgroup of $F_n$, the Cantor set $\bdy(\pi_1 S) \approx \bdy_\infinity \wt S$ embeds in~$\bdy F_n$. Let $\ti c$ be the lift of $c$ in $\bdy\wt R$. Let $\ell_i$, $i \in \Z$, be the remaining components of $\bdy\wt R$ listed in order, so $\ell_i$ and $\ell_{i+1}$ have a common ideal endpoint $\xi_i \in \bdy_\infinity \wt S \subset \bdy F_n$. Let $[F] = \A_\supp(\Lambda)$. Since $\ell_i$ is a leaf of $\Lambda$, there is a free factor $F_i$ conjugate to $F$ such that $\bdy \ell_i \subset \bdy F_i$. Since $\bdy F_i \intersect \bdy F_{i+1}$ contains the point $\xi_i$, it follows by the claim that $F_i =F_{i+1}$ for all $i$, so they are all equal to some free factor $F'$ conjugate to $F$. Since $\bdy \wt c$ is in the closure of the set $\{\xi_i \suchthat i \in \Z\}$ which is a subset of the closed set $\bdy_\infinity F'$, it follows that $\bdy \wt c \subset \bdy F'$, and so $[c] \in [F] = \A_\supp(\Lambda)$. Since $c$ is an arbitrary component of $\bdy S$, we have proved $\A_\supp(\bdy S) \subset \A_\supp(\Lambda)$.
\end{proof}

\section{Singular lines} 
\label{SectionSingular}

In this section, we study  ``singular lines'' of outer automorphisms, by using results from \recognition\ about principal automorphisms and principal lifts of \cts.

We consider singular lines only for a special subclass of outer automorphisms. For the remainder of this section we fix this subclass as follows:
\begin{description}
\item[Standing Assumption:] Let $\phi \in \Out(F_n)$ be rotationless. Suppose that there are no $\phi$-periodic conjugacy classes in~$F_n$; equivalently, by Fact~\ref{FactPeriodicIsFixed}~\pref{ItemPeriodicClassFixed}, no conjugacy class in~$F_n$ is fixed by $\phi$.
\end{description}
Under this assumption, a few extra facts are true. First, for any $\Phi \in \Aut(F_n)$ representing $\phi$, the subgroup $\Fix(\Phi) \subgroup F_n$ is trivial, so its boundary $\bdy\Fix(\phi) \subset \bdy F_n$ is empty. Applying Fact~\ref{LemmaFixPhiFacts} it follows that:
\begin{itemize}
\item For each $\Phi \in \Aut(F_n)$ representing $\phi$, $\Fix_N(\wh\Phi) = \Fix_+(\wh\Phi)$.
\end{itemize}
In the context of the Standing Assumption we mostly stick with the notation $\Fix_+(\wh\Phi)$.
Next, applying (Linear Edges), we have
\begin{itemize}
\item For any \ct\ representing $\phi$, no \neg\ stratum is a linear edge.
\end{itemize}

\begin{definition} \label{def:singular line} A line $\gamma$ is a \emph{singular line} for $\phi$ if there exists $\Phi \in P(\phi)$ and a lift $\ti \gamma$ of $\gamma$ whose endpoints are contained in $\Fix_+(\wh\Phi) \subset \bdy F_n$. The set of singular lines of $\phi$ is denoted~$S_\phi$. 
\end{definition}

If $\phi$ is fully irreducible, which is analogous to a pseudo-Anosov surface homeomorphism $f \from \Sigma \to \Sigma$, a singular line of $\phi$ is analogous to a line in $\Sigma$ which is a concatenation at a singularity of a pair of unstable separatrices of $f$; it is also analogous to a geodesic which is contained in one of the principal regions of the unstable geodesic lamination of $f$ with respect to any hyperbolic structure on~$\Sigma$. This analogy carries over to the broader case where $\phi$ satisfies the standing assumption: we shall prove that the set of singular lines of $\phi$ is finite and fills the free group $F_n$, and we shall characterize the singular lines very explicitly in terms of any \ct\ representing $\phi$.

Our results about singular lines are applied in Proposition~\ref{prop:WA2} and in Section~\ref{SectionLooking}, where the focus is on outer automorphisms that satisfy the stronger property of having a ``universally attracting lamination''. In this section, though, we focus on the standing assumption, since that is all we need in order to develop the theory of singular lines.


\begin{lemma} \label{finite singular set} If $\phi \in \Out(F_n)$ is rotationless and there are no $\phi$-periodic conjugacy classes then:
\begin{enumerate}
\item \label{ItemFixfinite}
$\Fix_+(\wh\Phi)$ is finite for all $\Phi \in \Aut(F_n)$ representing $\phi$. 
\item \label{ItemS_phiFills} 
$S_\phi$ fills
\item \label{ItemS_phiFinite}
$S_\phi$ is finite
\item \label{ItemFFFinite}
The set of $\phi$-invariant free factors is finite. 
\end{enumerate}
\end{lemma}

\begin{proof} In the compact space $\bdy F_n$, the set $\Fix(\wh\Phi)$ is closed and $\Fix_+(\wh\Phi)$, $\Fix_-(\wh\Phi)$ are discrete, so $\Fix_+(\wh\Phi) = \Fix(\wh\Phi) - \Fix_-(\wh\Phi)$ is closed and discrete and therefore finite, proving~\pref{ItemFixfinite}. Item~\pref{ItemS_phiFills} follows from the proof of Lemma~3.30 of \cite{FeighnHandel:recognition} which shows that $F_n$ is filled by the set of lines $\gamma$ for which there exists a lift $\ti\gamma$ and $\Phi \in P(\phi)$ such that $\bdy\ti\gamma \subset \Fix_+(\wh\Phi)$ (see the remark after the proof of Lemma~\ref{LemmaSingularLine} for a more concrete proof that $S_\phi$ fills and is finite).

For each $\Phi \in P(\phi)$, there are only finitely many singular lines $\gamma$ having lifts~$\ti\gamma$ whose endpoints are contained in the finite set $\Fix_+(\wh\Phi)$. If $\Phi,\Phi' \in \Aut(F_n)$ are isogredient then, choosing an inner automorphism $i_g$ such that $\Phi' = i_g \Phi i_g^\inv$, it follows that $\Fix_+(\wh\Phi') = \hat g(\Fix_+(\wh\Phi))$, and so the set of lines $\gamma$ which have a lift with endpoints in $\Fix_+(\wh\Phi)$ equals the set of lines which have a lift with endpoints in $\Fix_+(\wh\Phi')$. The number of isogredience classes of principal automorphisms representing $\phi$ is finite by \cite[Remark 3.9]{FeighnHandel:recognition}, and it follows that $S_\phi$ is finite, proving~\pref{ItemS_phiFinite}.

If $[F]$ is a $\phi$-invariant free factor then under the identification of $\bdy F$ with a subset of $\bdy F_n$ we have $S_{\phi \mid F} \subset S_\phi$. By \pref{ItemS_phiFills}, $[F]$ is the free factor support of the subset $S_{\phi \mid F}$. By \pref{ItemS_phiFinite}, $S_\phi$ is finite, so there are only finitely many possibilities for the subset $S_{\phi \mid F}$ and hence only finitely many possibilities for $[F]$.
\end{proof}

We next describe the singular lines of $\phi$ very explicitly in terms of any \ct\ $f \from G \to G$ representing $\phi$. This description is based on Lemma~4.37 of \cite{FeighnHandel:recognition}. For quoting that lemma we rely on the fact that no edge of $G$ is linear, which is a consequence of (Linear Edges) and the standing assumption that $\phi$ has no periodic conjugacy classes. 

Recall the natural bijection between lifts $\ti f \from \wt G \to \wt G$ and automorphisms $\Phi$ representing $\phi$, where $\ti f$ corresponds to $\Phi$ if and only if the extension $\hat f$ to $\bdy\wt G \approx \bdy F_n$ equals the extension $\wh\Phi$ to $\bdy F_n$. We say that $\ti f$ is a \emph{principal lift} if the corresponding $\Phi$ is a principal automorphism. 

\begin{fact}[\recognition\ Lemma 3.21, Corollary 3.22, Corollary 3.27]
\label{FactPrincipalLift}
For any \ct\ $f \from G \to G$, a lift $\ti f \from \wt G \to \wt G$ is principal if and only if there is a principal vertex $v \in G$ and a lift $\ti v \in \wt G$ such that $\ti f(\ti v) = \ti v$. \qed
\end{fact}

Define a \emph{principal direction} in $G$ to be (the direction of) an oriented, non-fixed edge $E$ whose initial vertex $v$ is principal and whose initial direction is fixed. A \emph{singular ray} is a ray that is obtained by iterating a principal direction, as defined precisely in item~\pref{ItemSingRay} of the following:

\vfill\break

\begin{fact} \label{FactPrincipalRay}
For each principal direction $E$ in $G$ the following hold:
\begin{enumerate}
\item \label{ItemSingRay}
The increasing union of the strictly nested sequence $E \subset f(E) \subset f^2_\#(E) \subset f^3_\#(E) \ldots$ is a completely split ray $R$ that we call a \emph{singular ray}. 
\item \label{ItemSingRayToAttr}
If $\wt E \subset \wt G$ is a lift of $E$ and $\ti f \from \wt G \to \wt G$ is the principal lift of $f$ fixing the initial endpoint of $\ti E$ then the increasing union $\wt E \subset \ti f(\wt E) \subset\ti f^2_\#(\wt E) \subset \ldots$ is a ray $\wt R$ that projects to $R$ and terminates at a point $P \in \Fix_+(\wh\Phi)$, where $\Phi \in P(\phi)$ corresponds to $\ti f$. 
\end{enumerate}
Conversely, for each $\Phi \in P(\phi)$ and each $P \in \Fix_+(\wh\Phi)$:
\begin{enumeratecontinue}
\item \label{ItemAttrToSingRay}
There exists a singular ray in $G$ and a lift of that ray to $\wt G$ that terminates at~$P$. 
\end{enumeratecontinue}
From this we obtain a bijection between the set of singular rays in~$G$ and the set of principal directions in~$G$, and a surjection from this set to the set of $F_n$-orbits of $\union_{\Phi \in P(\phi)} \Fix_+(\wh\Phi)$.
\end{fact}

\begin{proof} To prove \pref{ItemSingRay}, since the edge $E$ is not fixed, it follows by induction that each containment $f^k_\#(E) \subset f^{k+1}_\#(E)$ is strict, and so the union $R$ is indeed a ray. We claim that for each $k$ the sequence of terms in the complete splitting of $f^k_\#(E)$ is an initial segment of the sequence of terms in the complete splitting of $f^{k+1}_\#(E)$; from this claim it follows that $R$ has a complete splitting, whose sequence of terms is the union over~$k$ of the sequence of terms of the complete splitting of $f^k_\#(E)$. To prove the claim for $k=0$, let $H_r$ be the stratum containing $E$. If $H_r$ is \neg\ then the claim follows from Fact~\ref{FactNEGEdgeImage}. If $H_r$ is \eg\ then the claim follows by combining Fact~\ref{FactFinestSplitting}~\pref{ItemRefined} with the observation that, by definition of a relative train track map, the decomposition of $f(E)$ into edges of height $r$ and maximal subpaths of height $< r$ is a splitting. To prove the claim for $k>0$, use induction combined with Fact~\ref{FactFinestSplitting}~\pref{ItemRefined}. 

Item \pref{ItemSingRayToAttr} is contained in \cite{FeighnHandel:recognition} Lemma~4.37~(1), and item~\pref{ItemAttrToSingRay} is contained in \cite{FeighnHandel:recognition} Lemma~4.37~(2). 
\end{proof}

\begin{lemma}\label{LemmaSingularLine}
Suppose that $\phi \in \Out(F_n)$ is rotationless and has no periodic conjugacy classes, and $f \from G \to G$ is a \ct\ representing $\phi$.  A line $\ell$ in $G$ is the realization of a singular line for $\phi$ if and only if we can write $\ell = \overline R \alpha R'$ for some singular rays $R \ne R'$ and some path $\alpha$ which is either trivial or a Nielsen path. 
\end{lemma}

\begin{proof} Let $\ell = \overline R  \alpha  R'$ be as described in the lemma. Choose a lift $\ti g \from \wt G \to \wt G$ of $g$ and a lift $\ti\ell = \wt{\overline R}  \ti \alpha  \wt{R'}\subset \wt G$ of $\ti\ell$ so that $\ti g$ fixes the endpoints of $\ti\alpha$. Letting $\Phi \in P(\phi)$ be the principal automorphism corresponding to $\ti g$, it follows that $\wh\Phi$ fixes the infinite ends of both $\wt{\overline R}$ and $\wt{R'}$. By Fact~\ref{FactPrincipalRay}, these ends are in $\Fix_+(\wh\Phi)$, and so $\ell$ is a singular line.

Conversely, let $\ell$ be a singular line, and choose $\Phi \in P(\phi)$ and a lift $\ti\ell$ whose endpoints are contained in $\Fix_+(\wh\Phi)$.
By Fact~\ref{FactPrincipalRay}, we may choose singular rays $R,R'$ with initial edges $E,E'$ and lifts $\wt R, \wt R'$ which converge to the two endpoints of $\ti\ell$ respectively. If $\ti g \from \wt G \to \wt G$ is the lift of $g$ corresponding to $\Phi$ it follows that $\ti g$ fixes the initial points $\ti v, \ti v'$ of $\wt R, \wt R'$, respectively. Let $\ti\alpha$ be the path in $\wt G$ connecting $\ti v$ to~$\ti v'$. The projection of $\ti\alpha$ to $G$ is a Nielsen path $\alpha$ for $g$. Notice that the concatenation $\overline R \alpha R'$ need not be locally injective at the endpoints of $\alpha$. The path $\alpha$ has a unique decomposition into indivisible Nielsen paths and fixed edges, written $\alpha = \rho_1  \ldots  \rho_k$. 

We claim that if $R,R'$ are chosen so that $k$ is minimal then the concatenation $\overline R \alpha R'$ is indeed locally injective at the endpoints of $\alpha$. From this claim it follows that $\ell = \overline R \alpha R'$. Furthermore, since there are no closed Nielsen paths, the initial directions of $R$ and $R'$ are distinct, and so $R$ and $R'$ are distinct, completing the proof of the lemma.

To prove the claim, first note that if $k=0$ then $R,R'$ have distinct initial directions, by \recognition\ Lemma 3.21~(3), and so $\overline R R'$ is locally injective. Suppose that $k \ge 1$ and that $\overline R \alpha R'$ is not locally injective at the endpoints of $\alpha$; by reversing orientation if necessary we may assume that $R$ and $\alpha$ have the same initial edge $E$. From the definition of a singular ray it immediately follows that $E$ is not a fixed edge. Noting that $E$ is the initial edge of $\rho_1$, if $E$ were an \neg\ edge then, by (NEG Nielsen Paths) applied to $\rho_1$, the edge $E$ would be linear, contradicting that $G$ has no linear edges. It follows that $E$ is an edge of some \eg\ stratum $H_r \subset G$, and that $\rho_1$ is the unique indivisible height $r$ Nielsen path, with a decomposition into $H_r$-legal paths as $\rho_1 = \sigma \bar \tau$.

There is an initial segment of $E$ which maps to $\sigma$ under some iterate of $f_\#$, and so $\sigma$ is an initial segment of $R$. Since $R$ is $r$-legal and $\rho_1$ has an $r$-illegal turn at the endpoint of $\sigma$, it follows that the path $\sigma$ is the longest common initial segment of $R$ and $\rho_1$. Furthermore the ray $R'' = [\bar\rho_1 R]$ is a singular ray asymptotic to $R$, and $\tau$ is the longest common initial segment of $\bar\rho_1$ and $R''$. We may therefore replace the choice of $R$ by $R''$, obtaining a concatenation $\overline R'' (\rho_2  \ldots  \rho_k) R$ which contradicts minimality of $k$.
\end{proof}

\paragraph{Remark.} Lemma~\ref{LemmaSingularLine} and its proof gives a more concrete demonstration that $S_\phi$ is finite, because there are only finitely many singular rays and Nielsen paths. It also gives a more concrete proof that $S_\phi$ fills: if not, choose a \ct\ $f \from G \to G$ in which the free factor support of $S_\phi$ is realized by some proper filtration element $G_r$. There is a stratum of height~$>r$, and by (Zero Strata) there is an irreducible stratum $H_s$ of height $s>r$. By \recognition\ Lemma~3.19, there is an oriented edge $E \subset H_s$ whose initial direction is principal. By definition of principal vertex, there is another oriented edge $E' \subset G$ whose initial direction is principal, such that the initial vertices of $E$ and $E'$ are connected by a path $\alpha$ which is either trivial or a Nielsen path. By the proof of Lemma~\ref{LemmaSingularLine}, if $E'$ is chosen to minimize the number of terms in the complete splitting of $\alpha$, then the principal rays associated to $E$ and $E'$ may be concatenated with $\alpha$ in the middle, producing a singular line of height $\ge s$, a contradiction.

\bigskip

The next lemma will be used in Step~2 of Section~\ref{SectionLooking}.

\begin{lemma} 
\label{FactSingularRayDense}
Suppose that $\phi \in \Out(F_n)$ is rotationless and has no periodic conjugacy classes. For any $\Lambda \in \L(\phi)$, some leaf $\ell$ of $\Lambda$ is a singular line, at least one of whose ends is dense in $\Lambda$.
\end{lemma}

\begin{proof} Choose a \ct\ $g \from G \to G$ representing $\phi$. 

Note that $g$ has only finitely many indivisible Nielsen paths $\rho$, indeed only finitely many of height $r$ for each stratum $H_r \subset G$: by combining (Zero Strata), (NEG Nielsen Paths), and the standing assumption, it follows that $H_r$ is an \eg\ stratum, in which case there is at most one possibile $\rho$ by Fact~\ref{FactEGNPUniqueness}. Note also that amongst the set of paths $\alpha = \alpha_1 \ldots \alpha_m$ which are concatenations of fixed edges and indivisible Nielsen paths, there is an upper bound $M$ to the length $m$ of the concatenation, because if $m$ were sufficiently large then some term in this sum would be repeated with the same orientation, and so we would obtain a circuit which is a Nielsen path, violating the standing assumption.

Let $H_r$ be the \eg\ stratum corresponding to $\Lambda$. By \recognition\ Lemma~3.19, there exists an oriented edge $E \subset H_r$ whose initial direction is principal, and therefore fixed. Choose a generic leaf $\ell$ of $\Lambda$. Since $\ell$ crosses the edge $E$ it follows, by Fact~\ref{FactAttractingLeaves}~\pref{ItemLeafComplSplit}, that $\ell$ has a splitting of the form $\ell = \overline R' \cdot \bar E \cdot R$. Let $\alpha$ be the longest initial segment of the ray $R$ consisting of fixed edges and indivisible Nielsen paths in the complete splitting of $R$, and so we have $\ell = \overline R' \cdot \bar E \cdot \alpha \cdot R''$. Choose $\ell$ to maximize the number $m \ge 0$ of fixed edges and indivisible Nielsen paths in the complete splitting of $\alpha$; this is possible because $m \le M$ for any such $\alpha$. Choose $k$ so that the initial direction of $g^k_\#(R'')$ is fixed, and consider the generic leaf
$$g^k_\#(\ell) = g^k_\#(\overline R') \cdot g^k_\#(\bar E) \cdot \alpha \cdot g^k_\#(R'')
$$
By maximality of $m$, the initial term in the complete splitting of $g^k_\#(R'')$ is not a fixed edge or indivisible Nielsen path, and it is not an exceptional path by our standing assumption, so the only possibility is that it is a nonfixed edge $E'$ in an irreducible stratum. Consider the splitting
$$g^k_\#(\ell) = g^k_\#(R') \cdot g^k_\#(\bar E) \cdot \alpha \cdot E' \cdot R'''
$$
and note that the initial direction of $E'$ is fixed, by choice of $k$.

As we continue to iterate, we obtain a sequence of paths
$$\gamma_i = g^{k+i}_\#(\bar E) \cdot \alpha \cdot g^i_\#(E')
$$
where each $\gamma_i$ is a path in a generic leaf, and the sequence is strictly nested in the sense that $\gamma_i$ is a subpath of the interior of $\gamma_{i+1}$. It follows that the union of the paths $\gamma_i$ is a line $\ell' = \overline R_E \cdot \alpha \cdot R_{E'}$ where $R_E$ and $R_{E'}$ are the singular rays associated to the principal directions $E,E'$, and so by Lemma~\ref{LemmaSingularLine} the line $\ell'$ is singular. Furthermore, $\ell'$ is a weak limit of the generic leaves $g^k_\#(\ell)$ and so $\ell'$ is a leaf of $\Lambda$. Furthermore, by construction the singular ray $R_E$ contains a copy of each of the tiles $g^j_\#(E)$ for each $j \ge 0$, and so the weak closure of $R_E$ equals $\Lambda$, that is, $R_E$ is dense in $\Lambda$.

\end{proof}

The next lemma will be used in Step~2 of Section~\ref{SectionLooking}.

\begin{lemma}
\label{LemmaSingularSingleton}
Suppose that $\phi \in \Out(F_n)$ is rotationless and has no periodic conjugacy classes, and suppose that $\Lambda \in \L(\phi)$ has the property that every other $\Lambda' \in \L(\phi)$ contains $\Lambda$ as a subset. Then for any nonempty subset $L \subset S_\phi$, the free factor support $\A_\supp L$ is a singleton.
\end{lemma}

\begin{proof} Consider $[F] \in \A_\supp(L)$ and any $b \in L$ supported by $[F]$. Applying Lemma~\ref{LemmaSingularLine}, write $b$ in the form $\overline R_1 \alpha R_2$ where $R_1,R_2$ are singular rays and $\alpha$ is either trivial or a Nielsen path. 

We claim that the weak closure of any singular ray $R$ contains some element of $\Lambda' \in \L(\phi)$. Applying this claim, the weak closure of $b$ contains the weak closure of $R_1$, which contains some $\Lambda' \in \L(\phi)$, and so contains $\Lambda$. Since $[F]$ is weakly closed it follows that $\Lambda$ is supported by $[F]$. But this is true for any $[F] \in \A_\supp(L)$, and $\Lambda$ can be supported by at most one element of a free factor system, so $\A_\supp(L)$ must be a singleton. 

To prove the claim, write $R$ as the increasing union of $E \subset f(E) \subset f^2_\#(E) \subset \cdots$ for some nonfixed edge $E$ whose initial vertex $v$ is principal and whose initial direction is fixed. Let $H_r$ be the stratum containing $E$. If $H_r$ is an \eg\ stratum then clearly $R$ weakly accumulates on the lamination in $\L(\phi)$ associated to $H_r$ and we are done. Suppose that $H_r = E$ is \neg; we proceed by induction on $r$. Since $E$ is not fixed we have $r>1$ and $f(E) = E \cdot u$ for some nontrivial completely split circuit $u$ in $G_{r-1}$. If some term of the complete splitting of $u$ is an edge in an \eg\ stratum $H_s$, then $R$ weakly accumulates on the lamination in $\L(\phi)$ associated to $H_s$ and we are done. It remains to consider the case that no term in the complete splitting of $u$ is an edge in an \eg\ stratum. By (Zero Strata) it follows that no term of $u$ is a maximal subpath in a zero stratum. Since there are no linear strata, it follows that no term of $u$ is an exceptional path. The only remaining possibilities for terms of $u$ are indivisible Nielsen paths and edges of \neg\ strata. Since $E$ is not linear, $u$ is not fixed, and so some term of $u$ must be a nonfixed \neg\ edge $E'$ comprising some \neg\ stratum $H_s$ with $s<r$. Let $R'$ be the principal ray associated to $E'$. By induction, $R'$ weakly accumulates on some lamination $\Lambda' \in \L(\phi)$, and so $R$ weakly accumulates on $\Lambda'$ and we are done.
\end{proof}

\section{Vertex groups and vertex group systems}
\label{SectionVertexGroups}

\paragraph{Vertex groups.} Consider a proper, nontrivial subgroup $A \subgroup F_n$. We say that $A$ is a \emph{vertex group} if there exists a very small $F_n$ tree $T$ such that $A$ is a maximal elliptic subgroup for the action of $F_n$ on $T$ --- to be precise, there exists a point $x \in T$ such that 
\begin{description}
\item[(1) $A$ is elliptic:] $A = \Stab(x)$;
\item[(2) $A$ is maximal:] For each $y \in T$, if $A \subgroup \Stab(y)$ then $A = \Stab(y)$.
\end{description}
In this context we also say that the vertex group $A$ is \emph{realized at $x$ in $T$}. Each free factor is a vertex group, as one sees by constructing the appropriate simplicial tree. 

Note that if $A$ satisfies (1) but not (2) then there exists $q \ne p$ such that $A \subgroup \Stab(q)$, which implies that $A \subgroup \Stab \left( \overline{pq} \right)$, that is, $A$ is an edge stabilizer. In such a case it follows that $\Stab(q)$ is maximal, because $T$ is very small. This also shows that every subgroup that satisfies (1) is contained in one that satisfies (1) and~(2).

\paragraph{The descending chain condition.} Grushko's Theorem implies that the collection of free factors satisfies the descending chain condition: given free factors $A < A'$, we have $\rank(A) \le \rank(A')$ with equality if and only if $A=A'$. We need an analogous result about vertex groups, the proof of which was suggested to us by Mark Feighn.

\begin{proposition}\label{PropVDCC}
There is a bound depending only on the rank $n$ of the length $L$ of any strictly descending chain of vertex groups $A_1 > A_2 > \cdots > A_L$.
\end{proposition}

\begin{proof} We may restrict our attention to vertex groups which are not infinite cyclic, because every infinite cyclic vertex group $A$ is a maximal infinite cyclic subgroup of~$F_n$, being maximal in its commensurability class. 

We need the following result. The ``only if'' portion of the second sentence is \cite{Paulin:AutExt}, Theorem~5.2, and the rest of the result is an immediate consequence of the index inequality which is the main result of \cite{GaboriauLevitt:rank}.

\begin{lemma}
\label{LemmaVertexControl} For any $m \ge 2$, any free group $F$ of rank $m$, any small $F$-tree $T$, and any vertex $x \in T$, the group $\Stab(x)$ has rank $\le m$. If $T$ is very small then $\Stab(x)$ has maximal rank $m$ if and only if $T$ is simplicial and the quotient graph of groups $T/F$ is a wedge of circles each edge of which has an infinite cyclic edge group. \hfill\qed
\end{lemma}

One consequence of Lemma~\ref{LemmaVertexControl} is that for vertex groups $A_1 < A_0 < F_n$ we have $\rank(A_1) \le \rank(A_0)$. To see why, let $A_1$ be realized at $p$ in the very small $F_n$-tree~$T$. Let $S \subset T$ be a minimal $A_0$-subtree of $T$, and so $S$ is either a point or a very small $A_0$-tree. If $p \not\in S$ then $A_1$ stabilizes the shortest arc from $p$ to a point of $S$ and so $\rank(A_1) \le 1 \le \rank(A_0)$. If $\{p\} = S$ then $A_0 \subgroup \Stab(S) = \Stab(p) = A_1$ and so $A_1 = A_0$ and $\rank(A_1) = \rank(A_0)$. If $p \in S$ and $\{p\} \ne S$ then $S$ is a very small $A_0$-tree and so $\rank(A_1) \le \rank(A_0)$ by Lemma~\ref{LemmaVertexControl}.

The argument of the preceding paragraph also shows that if $A_0 > A_1$ is a proper containment of vertex groups with $\rank(A_1) \ge 2$, if $T$ is a very small $F_n$-tree, and if $S \subset T$ is a minimal $A_0$-subtree, then $S$ is not a point but instead is a very small $A_0$ tree, and every point of $T$ with stabilizer $A_1$ is in $S$.


For the rest of the proof consider a chain of proper containments of vertex groups $A_0 > A_1 > A_2$ with $\rank(A_0) \ge \rank(A_1) \ge \rank(A_2) \ge 2$. Let $\nc{A_i}$ denote the normal closure of $A_i$ in the group~$A_0$. Applying the functor $H_1(\,\cdot\,)$ (with $\Z$-coefficients) to the sequence of group epimorphisms $A_0 \surjection A_0 / \nc{A_2} \surjection A_0 / \nc{A_1}$ we obtain 
$$n \ge \beta_1(A_0) \ge \beta_1\bigl(A_0 / \nc{A_2}\bigr) \ge \beta_1\bigl(A_0 / \nc{A_1}\bigr)
$$
where $\beta_1(\cdot)$ is the first Betti number. To prove the proposition it therefore suffices, by induction, to assume that $\rank(A_0) = \rank(A_1) = \rank(A_2)$ and to prove that
$$\beta_1\bigl(A_0 / \nc{A_2} \bigr) > \beta_1\bigl(A_0 / \nc{A_1} \bigr)
$$

For $i=1,2$ let $A_i$ be realized at $p_i$ in the very small $F_n$-tree $T_i$. Let $S_i \subset T_i$ be a minimal $A_0$ subtree, $i=1,2$, so $S_i$ is a very small $A_0$ tree and $p_i \in S_i$. Applying Corollary~\ref{LemmaVertexControl} with $F=A_0$, the tree $S_i$ is simplicial and the graph of groups $S_i / A_0$ is a wedge of circles whose vertex group corresponds to $A_i$ and has rank equal to $\rank(A_0)$, and whose edge groups are all infinite cyclic. In what follows a \emph{vertex of $S_i$} is a point whose valence in $S_i$ is $\ge 3$ (its valence in $T_i$ is not relevant). The vertex set of $S_i$ is therefore equal to the orbit $A_0 \cdot p_i$.

Define the \emph{connecting homomorphism} $A_0 \mapsto \pi_1(S_i / A_0)$, which takes $\gamma \in A_0$ to the element of the fundamental group represented by projecting the path $\overline{p_i \gamma(p_i)}$ to a loop in $S_i / A_0$. This connecting homomorphism is clearly surjective. Since each vertex stabilizer of $S_i$ is conjugate to $A_i$, the kernel of the connecting homomorphism is the normal subgroup generated by the vertex stabilizers of $S_i$, which is just $\nc{A_i}$. We therefore have an induced group isomorphism $A_0 / \nc{A_i} \approx \pi_1(S_i / A_0)$ which induces an abelian group isomorphism $H_1(A_0 / \nc{A_i}) \approx H_1(\pi_1(S_i / A_0))$ and an equation $\beta_1\bigl(A_0 / \nc{A_i} \bigr) = \beta_1(S_i / A_0)$. It therefore suffices to prove that $\beta_1(S_2 / A_0) > \beta_1(S_1 / A_0)$.


Since $A_2 \subgroup A_1$, there is an $A_0$-equivariant map $\phi \from S_2 \to S_1$ such that $\phi(p_2)=p_1$: for each $f,g \in A_0$, if $f p_2$, $g p_2$ are connected by an edge then $\phi$ takes this edge to the path (possibly degenerate) connecting vertices $f p_1$, $g p_1$. Passing to the quotient we obtain a map $\Phi \from S_2 / A_0 \to S_1 / A_0$. The induced homology homomorphism $\Phi_* \from H_1(S_2 / A_0) \to H_1(S_1 / A_0)$ is surjective because it is identified with the inclusion induced homomorphism $H_1(A_0 / \nc{A_2}) \to H_1(A_0 / \nc{A_1})$. It therefore suffices to produce a loop in the graph $S_2 / A_0$ which is homologically nontrivial and whose $\Phi$-image is homologically trivial in $S_1 / A_0$.



\paragraph{Step 1: We may assume no edge of $S_2$ has degenerate image in $S_1$.} Otherwise, the projection of this edge to $S_2 / A_0$ is a loop representing a basis element of homology, whose image in $S_1 / A_0$ is a constant loop, and we are done.


\paragraph{Step 2: Two vertices of $S_2$ are identified in $S_1$.} To see why, the vertex set of $S_i$ equals the orbit $A_0 \cdot p_i$, and $\phi$ restricts to a map $A_0 \cdot p_2 \mapsto A_1 \cdot p_1$ which is clearly surjective. If this map were also injective it would follow that $A_1 = \Stab(p_1) = \Stab(\phi(p_1)) = \Stab(p_2) = A_2$, contradicting that $A_1$ is proper in $A_2$. 

\bigskip

Choose vertices $q,q' \in A_0 \cdot p_2$ such that $\phi(q) = \phi(q')$ and such that the edge path $\overline{qq'}$ has the minimal number of edges among all such choices. Write this edge path as $\overline{qq'} = E_1 E_2 \cdots E_I$. By Step~1, none of its edges has degenerate image, so its image $\phi\left(\overline{qq'}\right)$ is a nondegenerate finite subtree of $S_1$ (after Step~3 we will know that its image is a subarc of $S_1$). Notice that by choice of $q,q'$, no interior vertex of $\overline{qq'}$ is identified with any other vertex of $\overline{qq'}$.

\paragraph{Step 3: The edges of $\overline{qq'}$ all have the same stabilizer.} We start the proof with some notation. Let $Q = \phi(q) = \phi(q')$, and let $q=q_0,\ldots,q_M=q'$ be the points of the finite set $\phi^\inv(Q) \intersect \overline{qq'}$ written in order along $\overline{qq'}$ (the fact that this set is finite is a consequence of Step~1). By choice of $q,q'$, only the points $q_0,q_M$ are vertices: for $m=1,\ldots,M-1$ the point $q_m$ is interior to some edge denoted $E_{i_m}$. Also denote $E_{i_0} = E_0$, the edge whose initial point is $q_0$, and $E_{i_M} = E_I$, the edge whose terminal point is $q_M$.

Consider two edges in the edge path $E_1 \cdots E_I$, say $E_i,E_j$ with $i<j$. Observe first that if $\phi(E_i) \intersect \phi(E_j)$ contains an edge $E'$ of $S_1$ then $\Stab(E_i)=\Stab(E')=\Stab(E_j)$. Observe second that if $\Stab(E_i) = \Stab(E_j)$ then all of the edges in the subpath $E_i E_{i+1} \cdots E_j$ have the same stabilizer. Both of these observations follow from the fact that every edge stabilizer in $S_2$ or $S_1$ is a maximal infinite cyclic subgroup of $F_n$. 

For each $m=1,\ldots,M$ the set $\phi \left( \overline{q_{m-1} q_m} \right)$ is contained in the closure of a single component of $\phi\left(\overline{qq'}\right)-Q$, and in particular the initial and terminal directions of this path at the point $Q$ are equal. From the observations above it therefore follows that each of the edges in the subpath $E_{i_{m-1}} \cdots E_{i_m}$ have the same stabilizer. Chaining these together for $m=1,\ldots,M$ we obtain the desired conclusion, that all the edges $E_1,\ldots,E_I$ have the same stabilizer.

\bigskip

Let $Z$ denote the common stabilizer of the edges in $\overline{qq'}$.

\paragraph{Step 4: No two edges of $\overline{qq'}$ are in the same $A_0$-orbit.} Otherwise there is an edge $E_i \subset \overline{qq'}$ and $\gamma \in F_n$ such that $E_i \ne E_j = \gamma(E_i) \subset \overline{qq'}$. Since $Z$ is maximal infinite cyclic and $\gamma \not\in Z = \Stab(E_i)$ it follows that $\gamma^k \not\in Z$ for all $k \ne 0$. But $\gamma Z \gamma^\inv = \gamma \, \Stab(E_i) \, \gamma^\inv = \Stab(E_j) = Z$, and so the group $\<\gamma,Z\>$ is a semidirect product of two infinite cyclic groups, and therefore has a subgroup of index $\le 2$ isomorphic to the free abelian group of rank~2, which is impossible in a free group.

\bigskip

We now complete the proof of Proposition~\ref{PropVDCC} by observing that the projection of the path $\overline{qq'}$ to $S_2 / A_0$ is a loop which, by Step 4, crosses each edge of graph $S_2 / A_0$ at most once, and is therefore homologically nontrivial. But since $\phi(\overline{qq'})$ is a homotopically trivial loop in the tree $S_1$, the projection of $\overline{qq'}$ to $S_2 / A_0$ has homologically trivial image in $S_1 / A_0$ under the map $\Phi$, and we are done.
\end{proof}

\paragraph{Vertex group systems.} Given a very small $F_n$-tree $T$, the \emph{vertex group system of $T$}, denoted $\A_T$, is the subgroup system consisting of the conjugacy classes of all vertex groups that are realized in~$T$. We note that the cardinality of $\A_T$ is bounded above by a finite constant depending only on the rank~$n$; see \cite{GJLL:index}. 

A subgroup system $\A$ is said to be a \emph{vertex group system} if there exists a very small $F_n$-tree $T$ such that $\A = \A_T$. We also say that $\A$ is \emph{realized in $T$}. Each free factor system is a vertex group system, and can be realized in the appropriate simplicial $F_n$-tree.

\begin{lemma}
\label{LemmaVSElliptics}
A vertex group system $\A$ is uniquely determined by the set $\C$ of conjugacy classes of nontrivial elements of $F_n$ that are carried by $\A$. More precisely, for any subgroup $A \subgroup F_n$, we have $[A] \in \A$ if and only if the following hold:
\begin{itemize}
\item[(1)] the $F_n$-conjugacy class of each nontrivial element of $A$ is an element of $\C$; 
\item[(2)] $A$ is maximal with respect to inclusion among all subgroups satisfying (1).
\end{itemize}
\end{lemma}

\begin{proof} Choose a very small $F_n$-tree $T$ such that $\A = \A_T$. The lemma is an immediate consequence of two elementary results in the theory of trees: $\C$ consists precisely of the conjugacy classes of nontrivial elements of $F_n$ that are elliptic on~$T$; and for each subgroup $A \subgroup F_n$ whose elements are all elliptic on~$T$, there exists $p \in T$ such that $A \subgroup \Stab(p)$.
\end{proof}

\section{The Weak Attraction Theorem Revisited}
\label{SectionWeakAttraction}
Given $\phi \in \Out(F_n)$ and $\Lambda \in \L(\phi)$, a line $\ell \in \B$ is weakly attracted to some generic leaf of $\Lambda$ under iteration of $\phi$ if and only if $\ell$ is weakly attracted to every leaf in $\Lambda$. In this case we say that $\ell$ is \emph{weakly attracted to $\Lambda$}. We also use this terminology when iterating a \ct\ $f \from G \to G$ representing $\phi$ on any path $\ell \in \wh\B(G)$

The Weak Attraction Theorem, Theorem~6.0.1 of \BookOne, answers the question: 
\begin{itemize}
\item Which birecurrent lines $\gamma$ in $F_n$ are weakly attracted to a given topmost $\Lambda \in \L(\phi)$?
\end{itemize}
In this section we extend this result by dropping the hypotheses that $\Lambda$ be topmost and that $\gamma$ be birecurrent. Our main results in this direction are Propositions~\ref{prop:WA1} and~\ref{prop:WA2}, stated and proved in Section~\ref{SectionWAResults}. 

In Section~\ref{SectionAsubNA} we develop some preliminary results that will be needed before we can formulate Propositions~\ref{prop:WA1} and~\ref{prop:WA2}. In particular, we prove that there is a vertex group system $\A_{na}(\Lambda)$ called the ``nonattracting vertex group system'', depending implicitly on $\phi$ as well as on $\Lambda$, with the property that a conjugacy class in $F_n$ is not weakly attracted to $\Lambda$ if and only if it is carried by $\A_{na}(\Lambda)$; see Corollary~\ref{NA is well defined} and Proposition~\ref{PropVerySmallTree}. Furthermore, in the case that $\Lambda$ is not geometric, then we prove that $\A_{na}(\Lambda)$ is a free factor system; see Proposition~\ref{PropVerySmallTree}. 

\subsection{The nonattracting subgroup system $\A_{na}(\Lambda)$}
\label{SectionAsubNA}
 
The Weak Attraction Theorem answers the question posed above in terms of an improved relative train track representative $g \from G \to G$, a ``nonattracting subgraph'' $Z \subset G$, a (possibly trivial) Nielsen path $\hat \rho_r$, and an associated set of paths denoted $\<Z,\hat \rho_r\>$. The construction and properties of $Z$ and $\<Z,\hat \rho_r\>$ are given in \cite[Proposition~6.0.4]{BFH:TitsOne}. 

In Definition~\ref{defn:Z} and the lemmas that follow, we generalize $Z$ and $\<Z,\hat \rho_r\>$ beyond the topmost setting, and we develop their properties in the context of a subgroup system $\A_\na(\Lambda)$, all expressed in terms of a \ct\ $g \from G \to G$ representing $\phi$. Eventually, once we get to Corollary~\ref{CorPMna}, we will show that $\A_\na(\Lambda)$ is independent of the choice of $g$.

\begin{defns} \label{defn:Z} \textbf{The graph $Z$, the path $\hat \rho_r$, the path set $\<Z,\hat \rho_r\>$, the subgroup system $\A_\na(\Lambda)$, and its geometric model.} \quad

Suppose that $\fG$ is a \ct\ representing $\phi \in \Out(F_n)$ and that $\Lambda \in \L(\phi)$ corresponds to the \eg\ stratum $H_r \subset G$. We shall define the \emph{nonattracting subgraph} $Z$ of $G$, a path $\hat\rho_r$, and a graph immersion $h \from K \to G$ which can be thought of as the union of the inclusion map $Z \inject G$ and the path $\hat\rho_r$. We then define the subgroup system $\A_\na(\Lambda)$ using the induced fundamental group injection on each component of~$K$. We also define a groupoid of paths $\<Z,\hat\rho_r\>$, consisting of all concatenations whose terms are edges of $Z$ and copies of the path $\hat\rho_r$ or its inverse.
\end{defns}

\begin{defn*} \textbf{The graph $Z$.} The \emph{nonattracting subgraph} $Z$ of $G$ is defined as follows. Each stratum $H$ is contained either in $Z$ or in $G \setminus Z$. If $H_i$ is an irreducible stratum then $H_i \subset G \setminus Z$ if some (every) edge of $H_i$ is weakly attracted to $\Lambda$, which holds if and only if for some (every) edge $E_i$ of $H_i$ there exists $k \ge 0$ so that some term in the complete splitting of $f^k_\#(E_i)$ is an edge in~$H_r$. If $H_i$ is a zero stratum enveloped by an \eg\ stratum $H_s$ then $H_i \subset Z$ if and only if $H_s \subset Z$. 
\end{defn*}

\begin{remark} \label{RemarkLamContains} If $H_i$ is an \eg\ stratum then by applying Proposition~\ref{PropInclusion} we conclude that $H_i$ is excluded from $Z$ if and only if the lamination $\Lambda_i \in \L(\phi)$ associated to $H_i$ contains the lamination $\Lambda$.
\end{remark}

\begin{remark} \label{geometric strata are bottommost} If $H_i \ne H_r$ is an \eg\ stratum, and if there exists an \iNp\ of height $i$ then by Fact~\ref{FactNielsenBottommost} it follows that $H_i^z = H_i \subset Z$, because for each edge $E \subset H_i$ and each $k \ge 1$, the path $f^k_\#(E)$ splits into edges of $H_i$ and Nielsen paths of height $< i$.
\end{remark}

\begin{remark} \label{edges are taken} Suppose that $H_i$ is a zero stratum enveloped by the \eg\ stratum $H_s$ and that $H_i \subset Z$. Applying the definition of $Z$ to $H_i$ it follows that $H_s \subset Z$. Applying the definition of $Z$ to $H_s$ it follows that no $s$-taken connecting path in $H_i$ is weakly attracted to $\Lambda$. Applying (Zero Strata) it follows that no edge in $Z$ is weakly attracted to $\Lambda$.
\end{remark}

\begin{defn*} \textbf{The path $\hat \rho_r$.} If there is an \iNp\ $\rho_r$ of height $r$ then it is unique and we define $\hat \rho_r = \rho_r$. If there is no \iNp\ of height $r$, then by convention we choose a vertex of $H_r$, and define $\hat \rho_r$ to be the trivial path at that vertex. 
\end{defn*}

\begin{defn*} \textbf{The path set $\<Z,\hat \rho_r\>$.} Let $\<Z,\hat \rho_r\>$ denote those elements of $\wh\B(G)$ --- lines, rays, circuits, and finite paths in $G$ --- that decompose into a concatenation of subpaths, each of which is either an edge in $Z$, the path $\hat \rho_r$ or its inverse $\hat \rho_r^{-1}$.
\end{defn*}

\begin{defn*} \textbf{The subgroup system $\A_\na(\Lambda)$.} If $\hat\rho_r$ is the trivial path, define $K = Z$ and $h :K \to G$ to be the inclusion. Otherwise, define $K$ to be the graph obtained from the disjoint union of $Z$ and an edge $E_\rho$ representing the domain of the map $\rho$, with identifications as follows. If the initial [resp.\ terminal] endpoint of $\rho$ is contained in $Z$ then identify the initial [resp.\ terminal] endpoint of $E_\rho$ with the initial [resp.\ terminal] endpoint of $\rho$. If $\rho_r$ forms a loop and the basepoint of this loop is not contained in $Z$ then identify the endpoints of $E_r$ (these points are already identified if the basepoint is contained in $Z$). Define $h : K \to G$ to be the inclusion\ on $Z$ and to be the map $\rho$ on $E_\rho$. By Fact~\ref{FactEGStrata} item~\pref{ItemFirstLast} the first and last edges of $\rho_r$ are distinct edges in $H_r$, and since no edge of $H_r$ is in $Z$ it follows that the map $h$ is an immersion. The restriction of $h$ to each component of $K$ therefore induces an injection on the level of fundamental groups.

Define $\A_\na(\Lambda)$, the \emph{nonattracting subgroup system}, to be the subgroup system determined by the images of the fundamental group injections induced by the immersion $h \from K \to G$, over all noncontractible components of $K$.

\paragraph{Definition. The geometric model for $\A_\na(\Lambda)$.} In the special case that the stratum $H_r$ is geometric, we adopt the notation of the geometric model for $f \restrict G_R$ as given in Definition~\ref{DefGeometricStratum}, and we extend this to a geometric model for $\A_\na(\Lambda)$ as follows. Let $X$ be the 2-complex obtained from the disjoint union of $Y$ and $G$ by identifying the copy of $G_r$ in $Y$ with the copy of $G_r$ in $G$. The union of the deformation retraction $d \from Y \to G_r$ with the identity map on $G$ induces a map $X \mapsto G$, a deformation retractions with which we mark the 2-complex $X$, and which by extension is denoted $d \from X \to G$. The induced isomorphism $d_* \from \pi_1 X \approx F_n$ is well-defined up to inner automorphism of $F_n$. By Definition~\ref{DefGeometricStratum}~\pref{ItemInteriorBasePoint}, we may identify the 1-complex $K$ with the subcomplex $Z \union \bdy_0 S \subset X$ in such a way that the immersion $K \to G$ is identified with the restriction of the map $d$. Letting $K_1,\ldots,K_m$ be the noncontractible components of $K$, the subgroup system $\A_\na(\Lambda)$ is therefore identified with the subgroup system $\{[d_* (\pi_1 K_1)],\ldots,[d_* (\pi_1 K_m)]\}$.

\end{defn*}

\begin{remark} We show in Corollary~\ref{CorPMna} that $\A_\na(\Lambda)$ depends only on $\phi$ and on $\Lambda \in \L(\phi)$. We often suppress $\phi$ from the notation, although in contexts where more than one element of $\Out(F_n)$ is under discussion we will sometimes affix $\phi$ to the notation, by writing $\A_\na(\Lambda;\phi)$. 
\end{remark}

Item~\pref{Item:Groupoid} in the next lemma states that $\<Z,\hat \rho_r\>$ is a groupoid, by which we mean that the tightened concatenation of any two paths in $\<Z,\hat \rho_r\>$ is also a path in $\<Z,\hat \rho_r\>$ as long as that concatenation is defined. For example, the concatenation of two distinct rays in $\<Z,\hat \rho_r\>$ with the same base point tightens to a line in $\<Z,\hat \rho_r\>$. 


\begin{lemma} \label{ZP is closed} Assuming the notation of Definitions~\ref{defn:Z} 
\begin{enumerate}
\item \label{Item:BEQuivZP}
The map $h$ induces a bijection between $\wh\B(K)$ and $\<Z,\hat \rho_r\>$.
\item \label{Item:Groupoid}
$\<Z,\hat \rho_r\>$ is a groupoid. 
\item \label{item:ZP=NA} The set of lines carried by $\<Z,\hat \rho_r\>$ is the same as the set of lines carried by $\A_{na}(\Lambda)$. 
\item\label{Item:circuits} The set of circuits carried by $\<Z,\hat \rho_r\>$ is the same as the set of circuits carried by $\A_{na}(\Lambda)$. 
\item \label{item:closed}The set of lines carried by $\<Z,\hat \rho_r\>$ is closed in the weak topology.
\end{enumerate}
\end{lemma}

\begin{proof} We make use of three evident properties of the immersion $h:K \to G$. The first is that every path in $K$ with endpoints, if any, at vertices is mapped by $h$ to an element of $\<Z,\hat \rho_r\>$. The second is that $h$ induces a bijection between the vertex sets of $K$ and $Z \cup \partial \hat \rho_r$. The last is that $\hat \rho_r$ and each edge in $Z$ lift via $h$ to a unique path in $K$ with endpoints at vertices. Together these imply~\pref{Item:BEQuivZP}. Items~\pref{Item:Groupoid} and~\pref{item:ZP=NA} are immediate consequences. Item~\pref{Item:circuits} follows from \pref{item:ZP=NA} using the natural bijection between periodic lines and circuits. Item~\pref{item:closed} follows from~\pref{item:ZP=NA} and Fact~\ref{FactLinesClosed}.
\end{proof}

The following lemma is based on Proposition~6.0.4 and Corollary~6.0.7 of \BookOne.


\begin{lemma} \label{defining Z} Let $f \from G \to G$ be a \ct\ representing a rotationless $\phi \in \Out(F_n)$, let $\Lambda \in \L(\phi)$, and assume the notation of Definitions~\ref{defn:Z}. We have: 
\begin{enumerate}
\item\label{ItemZPEdgesInv}
If $E$ is an edge of $Z$ then 
$f_\#(E) \in \langle Z,\hat \rho_r \rangle $.
\item\label{ItemZPPathsInv}
$\<Z,\hat \rho_r\>$ is $f_\#$-invariant.
\item\label{ItemZPAnyPaths}
If $\sigma \in\<Z,\hat \rho_r\>$ then $\sigma$ is not weakly attracted to $\Lambda$.
\item\label{ItemZPFinitePaths}
For any finite path $\sigma$ in $G$ with endpoints at fixed vertices, the converse to \pref{ItemZPAnyPaths} holds: if $\sigma$ is not weakly attracted to $\Lambda$ then $\sigma \in \<Z,\hat \rho_r\>$.
\item \label{ItemBijection} $f_\#$ restricts to bijections of the the following sets: lines in $\<Z,\hat \rho_r\>$; finite paths in $\<Z,\hat \rho_r\>$ whose endpoints are fixed by $f$; and circuits in $\<Z, \hat\rho_r\>$.
\end{enumerate}
\end{lemma} 

\begin{proof} $\<Z,\hat \rho_r\>$ contains each fixed or linear edge by construction. Given an indivisible Nielsen path $\rho_i$ of height $i$, we prove by induction on $i$ that $\rho_i$ is in $\<Z,\hat\rho_r\>$. If $H_i$ is \neg\ this follows from (\neg\ Nielsen Paths) and the induction hypothesis. If $H_i$ is \eg\ then, applying Fact~\ref{FactNielsenBottommost}~\pref{ItemBottommostEdges} and the induction hypothesis, it follows that each edge of $H_i$ is in $Z$; together with another application of the induction hypothesis it follows that $\rho_i$ is in $\<Z,\hat\rho_r\>$.

Since all indivisible Nielsen paths and all fixed edges are contained in $\<Z,\hat\rho_r\>$, it follows that all Nielsen path are contained in $\<Z,\hat\rho_r\>$, which immediately implies that $\<Z,\hat \rho_r\>$ contains all exceptional paths. 

Suppose that $\tau = \tau_1 \cdot \ldots \cdot \tau_m$ is a complete splitting of a finite path that is not contained in a zero stratum. Each $\tau_i$ is either an edge in an irreducible stratum, a taken connecting path in a zero stratum, or, by the previous paragraph, a term which is not weakly attracted to $\Lambda$. If $\tau_i$ is a taken connecting path in a zero stratum $H_t$ that is enveloped by an \eg\ stratum $H_s$ then, by definition of complete splitting, $\tau_i$ is a maximal subpath of $\tau$ in $H_t$; since $\tau \not\subset H_t$ it follows that $m \ge 2$, and by applying (Zero Strata) it follows that at least one other term $\tau_j$ is an edge in $H_s$. In conjunction with Remark~\ref{edges are taken}, this proves that $\tau$ is contained in $\<Z,\hat \rho_r\>$ if and only if each $\tau_i$ that is an edge in an irreducible stratum is contained in $Z$ if and only if $\tau$ is not weakly attracted to $\Lambda$. 
 
We apply this in two ways. First, this proves item \pref{ItemZPFinitePaths} in the case that $\sigma$ is completely split. Second, applying this to $\tau = f_\#(E)$ where $E$ is an edge in $Z$, item \pref{ItemZPEdgesInv} follows in the case that $f_\#(E)$ is not contained in any zero stratum. Consider the remaining case that $\tau=f_\#(E)$ is contained in a zero stratum $H_t$ enveloped by the \eg\ stratum $H_s$. By definition of complete splitting, $\tau=\tau_1$ is a taken connecting path. By Fact~\ref{FactEdgeToZeroConnector} the edge $E$ is contained in some zero stratum $H_{t'}$ enveloped by the same \eg\ stratum $H_s$. Since $E \subset Z$, it follows that $H_s \subset Z$, and so $H^z_s \subset Z$, and so $\tau \subset Z$, proving \pref{ItemZPEdgesInv}.

Item \pref{ItemZPPathsInv} follows from item \pref{ItemZPEdgesInv}, the fact that $f_\#(\hat \rho_r) = \hat \rho_r$ and the fact that $\<Z,\hat \rho_r\>$ is a groupoid. 

 Every generic leaf of $\Lambda$ contains subpaths in $H_r$ that are not subpaths of $\hat \rho_r$ or $\hat \rho_r^{-1}$ and hence not subpaths in any element of $\<Z,\hat \rho_r\>$. Item \pref{ItemZPAnyPaths} therefore follows from item \pref{ItemZPPathsInv}. 
 
To prove \pref{ItemBijection}, for lines and finite paths the implication $\pref{ItemZPPathsInv} \Rightarrow \pref{ItemBijection}$ follows from Corollary~6.0.7 of \BookOne. For circuits, use the natural bijection between circuits and periodic lines, noting that this bijection preserves membership in $\<Z,\hat\rho_r\>$.
 
 It remains to prove \pref{ItemZPFinitePaths}. By \pref{ItemBijection}, there is no loss of generality in replacing $\sigma$ with $f^k_\#(\sigma)$ for any $k \ge 1$. By Fact~\ref{FactEvComplSplit} this reduces \pref{ItemZPFinitePaths} to the case that $\sigma$ is completely split which we have already proved.
\end{proof} 

The following corollary shows that the set of conjugacy classes carried by $\A_{\na}(\Lambda)$ depends only on $\phi$ and $\Lambda$ and not on $\fG$. This is strengthened in Corollary~\ref{CorPMna} which shows that subgroup system $\A_{\na}(\Lambda)$ itself depends only on $\phi$ and~$\Lambda$. 

\begin{corollary}\label{NA is well defined} Let $f \from G \to G$ be a \ct\ representing a rotationless $\phi \in \Out(F_n)$, let $\Lambda \in \L(\phi)$, and assume the notation of Definitions~\ref{defn:Z}. Then a conjugacy class $[a]$ in $F_n$ is not weakly attracted to $\Lambda$ if and only if it is carried by $\A_{\na}(\Lambda)$. \end{corollary}

\begin{proof} By Lemma~\ref{ZP is closed}\pref{item:ZP=NA}, it suffices to show that a circuit in $G$ is not weakly attracted to $\Lambda$ under iteration by $f_\#$ if and only if it is carried by $\<Z,\hat\rho_r\>$. Both the set of circuits in $\<Z,\hat\rho_r\>$ and the set of circuits that are not weakly attracted to $\Lambda$ are $f_\#$-invariant. We may therefore replace $\sigma$ with any $f^k_\#(\sigma)$ and hence may assume that $\sigma$ is completely split. After taking a further iterate, we may assume that some coarsening of the complete splitting of $\sigma$ is a splitting into subpaths whose endpoints are fixed by~$f$. Lemma~\ref{defining Z}(\ref{ItemZPFinitePaths}) completes the proof.
\end{proof}

\begin{corollary}\label{NA is independent of plusminus} Assume that $\phi, \phi^\inv \in \Out(F_n)$ are rotationless, and let $\Lambda^\pm \in \L^\pm(\phi)$ be a dual lamination pair. For each nontrivial conjugacy class $[a]$ in $F_n$, $[a]$ is weakly attracted to $\Lambda^+$ under iteration by $\phi$ if and only if $[a]$ is weakly attracted to $\Lambda^-$ under iteration by~$\phi^{-1}$.
\end{corollary}

\begin{proof} Fixing $[a]$, by replacing $\phi$ with $\phi^\inv$ it suffices to prove the ``if'' direction. Applying Theorem~\ref{TheoremCTExistence}, choose a \ct\ $ \fG$ representing $\phi$ which realizes the free factor system $\A_\supp(\Lambda^\pm)$. It follows that if $H_r \subset G$ is the \eg\ stratum corresponding to $\Lambda^+$ then $[G_r] = \A_\supp(\Lambda^\pm)$. We adopt the notation of Definitions~\ref{defn:Z}. 

Suppose that $[a]$ is not weakly attracted to $\Lambda^+$ under iteration by $\phi$. Then the same is true for all $\phi^{-k}([a])$ and so $\phi^{-k}([a]) \in \<Z,\hat\rho_r\>$ for all $k \ge 0$ by Corollary~\ref{NA is well defined}. 

Suppose in addition that $[a]$ is weakly attracted to $\Lambda^-$ under iteration by $\phi^{-1}$. Corollary~\ref{ZP is closed}~\pref{item:closed} implies that a generic line $\gamma$ of $\Lambda^-$ is contained in $\<Z,\hat\rho_r\>$. However, since $\A_\supp(\gamma) = \A_\supp(\Lambda^\pm) = [G_r]$, it follows that $\gamma$ has height $r$. If $\hat\rho_r$ is trivial then $\gamma$ is a concatenation of edges of $Z$ none of which has height $r$, a contradiction. If $\hat\rho_r = \rho_r$ is nontrivial then every maximal subpath of $\gamma$ in $H_r$ is a concatenation of copies of $\rho_r$. By Fact~\ref{FactEGStrata} items~\pref{ItemSomeOnce} and~\pref{ItemEachTwice}, at least one endpoint of $\rho_r$ is disjoint from $G_{r-1}$. If $\rho_r$ is not closed then we obtain an immediate contradiction. If $\rho_r$ is closed then $\gamma$ is a bi-infinite iterate of $\rho_r$, but this contradicts \BookOne\ Lemma~3.1.16 which says that no generic leaf of $\Lambda^-$ is periodic. 

We conclude that $[a]$ is not weakly attracted to $\Lambda^-$ under iteration by $\phi^{-1}$. 
\end{proof}

\begin{corollary}
\label{CorollaryCTFullIrr}
For each rotationless $\phi \in \Out(F_n)$ and each $\Lambda \in \L(\phi)$, if the subgroup system $\A_\na(\Lambda)$ is trivial and the free factor system $\A_\supp(\Lambda)$ is not proper then $\phi$ is fully irreducible.
\end{corollary}

\begin{proof} If $\phi$ is not fully irreducible then by Fact~\ref{FactPeriodicIsFixed}~\pref{ItemFreeFactorFixed} there is a proper, nontrivial, invariant free factor $F$ whose conjugacy class $[F]$ is fixed by $\phi$. Let $f \from G \to G$ be a \ct\ representing $\phi$ with a filtration element $G_r$ such that $[G_r] = [F]$. Since $\A_\supp(\Lambda)$ is not proper, the stratum corresponding to $\Lambda^+$ is the top stratum $H_s$, so $r < s$. Since $\A_\na(\Lambda)$ is trivial it follows, by Corollary~\ref{NA is well defined} and Lemma~\ref{defining Z}~\pref{ItemZPAnyPaths}, that each component of $Z$ is contractible. Since $G_r \subset G_{s-1} \subset Z$ it follows that each component of $G_r$ is contractible. But this contradicts that $[G_r]=[F]$ is nontrivial.
\end{proof}

\subsection{Nonattracting subgroup systems are vertex group systems}
\label{SectionNASysIsVertexSys}

In the next proposition we study the structure of the subgroup system $\A_\na(\Lambda)$ for any attracting lamination of any outer automorphism


\begin{proposition} \label{PropVerySmallTree} 
For any $\phi \in \Out(F_n)$ and any $\Lambda \in \L(\phi)$ the following hold:
\begin{enumerate}
\item \label{ItemA_naNP}
$\A_\na(\Lambda)$ is a vertex group system.
\item \label{ItemA_naNoNP}
If $\Lambda$ is not geometric then $\A_\na(\Lambda)$ is a free factor system.
\end{enumerate}
\end{proposition}

\begin{proof} Passing to a power we may assume that $\phi$ is rotationless. Choose a \ct\ $f \from G \to G$ representing $\phi$ with \eg\ stratum $H_r$ corresponding to $\Lambda$, and adopt the notation of Definition~\ref{defn:Z}. By applying Fact~\ref{FactEGStrata} and Proposition~\ref{PropGeomEquiv}, when $\Lambda$ is not geometric then $\hat\rho_r$ is either trivial or a nonclosed Nielsen path, and when $\Lambda$ is geometric then $\hat\rho_r$ is a closed Nielsen path. We consider these three cases of $\hat\rho_r$ separately.

\smallskip

\textbf{Case 1: $\hat\rho_r$ is trivial.} In this case $K=Z$, and $\A_\na(\Lambda)$ is the free factor system associated to the subgraph $Z \subset G$.

\smallskip

\textbf{Case 2: $\hat\rho_r = \rho_r$ is a nonclosed Nielsen path.} We prove that $\A_\na(\Lambda)$ is a free factor system following an argument of \BookOne\ Lemma~5.1.7. By Fact~\ref{FactEGStrata} item~\pref{ItemSomeOnce} there is an edge $E \subset H_r$ that is crossed exactly once by $\rho_r$. We may decompose $\rho_r$ into a concatenation of subpaths $\rho_r = \sigma E \tau$ where $\sigma,\tau$ are paths in $G - E$. Let $\wh G$ be the graph obtained from $G - E$ by attaching an edge $J$, letting the initial and terminal endpoints of $J$ be equal to the initial and terminal endpoints of~$\rho_r$, respectively. The identity map on $G - E$ extends to a map $h \from \wh G \to G$ that takes the edge $J$ to the path $\rho_r$, and to a homotopy inverse $\bar h \from G \to \wh G$ that takes the edge $E$ to the path~$\bar\sigma J \bar\tau$. We may therefore view $\wh G$ as a marked graph, pulling the marking on $G$ back via $h$. Notice that $K$ may be identified with the subgraph $Z \union J \subset \wh G$, in such a way that the map $h \from \wh G \to G$ is an extension of the map $h \from K \to G$ as originally defined. It follows that $\A_\na(\Lambda)$ is the free factor system associated to the subgraph $Z \union J$.

\smallskip

\textbf{Case 3: $\hat\rho_r = \rho_r$ is a closed Nielsen path.} We prove that $\A_\na(\Lambda)$ is a vertex group system, first doing a special case.

%



\subparagraph{Case 3a: $G = Z \union H_r$.} Using an iterative process, we contruct a very small $F_n$ tree $T_\infinity$ in which $\A_\na(\Lambda)$ is realized. We do not iterate the map $f$, instead we iterate a relative train track map $h \from G \to G$ closely related to $f$ in which the dynamics off of the stratum $H_r$ have been simplified as much as possible.

In this proof, since we will be reordering strata of $G$, we will write the subgraph $H_r$ just as $H$, and the path $\rho_r$ as $\rho$.

Define $h \from G \to G$ by $h \restrict Z = \Id$ and $h \restrict H = f \restrict H$. The map is well defined and continuous because $f$ fixes each point of $Z \intersect H$. To see that $h$ is a homotopy equivalence, first define $h' \from G \to G$ by $h' \restrict G_{r} = f \restrict G_r$ and $h' \restrict (G \setminus G_r) = \Id$; by Fact~\ref{FactContrComp} the subgraph $G_r$ has no contractible components, implying that $h' \restrict G_r$ is a homotopy equivalence, and so its extension by the identity is a homotopy equivalence. Next factor $h'$ as $G \xrightarrow{h''} G \xrightarrow{h} G$ where $h'' \restrict G_{r-1} = h' \restrict G_{r-1} = f \restrict G_{r-1}$ and $h'' \restrict (G \setminus G_{r-1}) = \Id$; by Facts~\ref{FactNielsenBottommost} and~\ref{FactContrComp} the subgraph $G_{r-1}$ has no contractible components, implying that $h'' \restrict G_{r-1}$ is a homotopy equivalence, and so its extension by the identity is a homotopy equivalence. Since $h'$ and $h''$ are both homotopy equivalences, so is $h$.

Note that $h$ respects the filtration of $G$. Since the restrictions of $f$ and $h$ to each edge of $H$ are equal, it follows that $f$ and $h$ have the same transition matrix $M$ for their actions on $H$, and so $H$ is an \eg-aperiodic stratum for $h$. It follows easily that $h$ is a relative train track map. Also, the restrictions of $Df$ and $Dh$ to the set of directions in $H$ are equal, implying that $f$ and $h$ have the same legal turns in the stratum $H$. Also, $\rho$ is a Nielsen path for $h$ as well as for $f$, because the restrictions of $h$ and $f$ to each edge of $H$ in $\rho$ are equal, and by Fact~\ref{FactNielsenBottommost} each maximal subpath $\beta$ of $\rho$ in $G_{r-1}$ is a Nielsen path for $f$ and so $f_\#(\beta)$ and $h(\beta)$ are both equal to $\beta$. Furthermore, for each edge $E$ of $H$ and each $i \ge 0$, the paths $f^i_\#(E)$ and $h^i(E)$ are equal and each of its maximal subpaths in $G \setminus H$ is a Nielsen path for $f$ in $G_{r-1}$; for $i=1$ this follows by the same argument as above using Fact~\ref{FactNielsenBottommost}, and the general case follows easily by induction.


Since $h$ fixes each edge of $Z = G \setminus H$, we may subdivide and reorder the strata to obtain a new filtration respected by $h$ in which each edge of $Z$ is a stratum and so that $H$ is the top stratum. With respect to this new filtration, $h$ is still a relative train track map with aperiodic \eg\ stratum $H$. There are only finitely many indivisible Nielsen paths for $h$ which intersect the interior of $H$, by \BookOne\ Lemma~4.2.5, and so we may subdivide $G$ so that each end endpoint of each such Nielsen path is a vertex of $G$.

We claim that $\rho$ is the unique $h$-periodic indivisible Nielsen path which intersects the interior of $H$. Suppose that $\rho'$ is another such path. Applying \BookOne\ Remark 5.1.2 using the new filtration, we can write $\rho' = \alpha' \beta'$ where $\alpha',\beta'$ are each $H$-legal paths. Choose $k$ so that $\rho' = h^k_\#(\rho')$. We need the fact that if $\bar\alpha'(i)$ and $\beta(i)$ are the initial segments of $\bar\alpha'$ and $\beta'$ that are identified by $f^i_\#$ then the paths $\alpha'(i) \beta'(i)$ are an increasing sequence of subpaths of $\rho$ whose union is the interior of $\rho$; this fact is noted in \recognition\ Lemma~2.11, by quoting from the statement and proof of \BH\ Lemma~5.11. Letting $E$ be the initial edge of $\alpha'$, it follows that there exists $l>0$ such that $\alpha'$ is an initial segment of $h^{kl}(E) = f^{kl}_\#(E)$ whose last edge is contained in~$H$, and so $\alpha'$ is a concatenation of edges of $H$ and maximal subpaths in $G \setminus H$ each of which is a Nielsen path for $f$ in $G_{r-1}$. It follows that $h^i(\alpha') = f^i_\#(\alpha')$ for all $i \ge 0$, and the same argument shows that $h^i(\beta') = f^i_\#(\alpha')$. We therefore have $f^k_\#(\rho') = [f^k_\#(\alpha') f^k_\#(\beta')] = [h^k(\alpha') h^k(\beta')] = h^k_\#(\rho') = \rho'$. Since $\rho$ is the unique $f$-periodic Nielsen path in $G_{r}$ which intersects the interior of $H$, it follows that $\rho' = \rho$, proving the claim.

\bigskip

Applying the claim together with Lemma~\ref{LemmaEGPathSplitting}, and using that $H$ is the top stratum of the new filtration, for each circuit $\sigma$ in $G$ there exists $K \ge 0$ such that $h^K_\#(\sigma)$ splits into terms each of which is an edge or a copy of $\rho$ or $\bar\rho$. 

\bigskip

We next construct a tree $T_0$ on which to iterate the outer automorphism $\psi \in \Out(F_n)$ which is represented by $h$. Let $\lambda>1$ be the Perron-Frobenius eigenvalue for $M$ and let $\vec v$ be a positive vector satisfying $\vec v M = \lambda \vec v$. 

Choose $\epsilon > 0$ and assign lengths $L_\epsilon(E)$ to the edges $E \subset G$ as follows. If $E \subset Z$ then $L_\epsilon(E) = \epsilon$. If $E \subset H$ then $L_\epsilon(E) = v_E$ where $v_E$ is the coordinate of $\vec v$ determined by $E$. This defines a function $L_\epsilon$ assigning positive lengths to circuits and finite edge paths in $G$ by $ L_\epsilon( E_1E_2 \ldots E_m) = \sum_{i=1}^m L_\epsilon(E_i)$. The restriction of $L_\epsilon$ to circuits in $G$ defines a positive length function $\wh L_\epsilon$ on conjugacy classes in $F_n$, and $\wh L_\epsilon$ corresponds to a very small $F_n$-tree $T_\epsilon$ which represents the same point in $\CV_n$ as the marked graph $G$ equipped with the path metric given by the values of $L_\epsilon$. 

Next we let $\epsilon$ go to zero. Assign formal lengths $L_0(E)$ to edges in $G$ by $L_0(E) = 0$ for $E \subset Z$ and $L_0(E) = v_E$ for $E \subset H$. For any circuit or finite edge path define $ L_0( E_1E_2 \ldots E_m) = \sum_{i=1}^m L_0(E_i)$. Thus $L_0(h_\#(E)) = \lambda L_0(E)$ for every edge $E$; this is obvious if $E \subset Z$ and follows from our choice of $\vec v$ if $E \subset H$. The restriction of $L_0$ to circuits in $G$ defines a non-negative length function $\wh L_0$ on conjugacy classes in $F_n$, and $\wh L_0$ is the limit, as $\epsilon \to 0$, of the functions $\wh L_\epsilon$. It follows that $\wh L_0$ is the translation length function for some very small $F_n$-tree $T_0$ representing a point in $\overline\CV_n$. 

Define a sequence of very small $F_n$-trees $T_k = T_{0}\psi^k/\lambda^k$. Let $\wh L_k$ be the translation length function on conjugacy classes in $F_n$ corresponding to $T_k$. 

\bigskip

We now let $k \to \infinity$ and prove that the sequence in $\overline\CV_n$ represented by the $F_n$-trees $T_k$ converges to a point in $\overline\CV_n$ represented by a very small $F_n$-tree tree $T_\infinity$ in which $\A_\na(\Lambda)$ is realized.

Consider a circuit $\sigma \subset G$ representing a conjugacy class $[a]$. We have 
$$
\wh L_k([a]) = \wh L_0(\psi^k[a])/\lambda^k = L_0(h^k_\#(\sigma))/\lambda^k
$$
Choose $K$ such that\ $h^K_\#(\sigma)$ splits into terms $\sigma^K_i$ each of which is a single edge or $\rho$ or~$\bar\rho$. Partition these terms into two sets $A$ and $B$, with $A$ containing the terms that are $\rho$ or $\bar\rho$, and $B$ containing all the other terms. The cardinality $\abs{A}$ of the set $A$ is independent of $k\ge K$. Each term $\sigma^K_i \in B$ satisfies $L_0(h_\#(\sigma^K_i)) = \lambda L_0(\sigma^K_i)$. It follows that for all $ k\ge K$
$$ 
\wh L_k([a]) = \sum L_0(h^{k-K}_\#(\sigma^K_i))/\lambda^k \\
 = \abs{A} L_0(\rho)/\lambda^{k} + \sum_{\sigma^K_i \in B} L_0( \sigma^K_i)/\lambda^{K} 
$$
and hence that
$$
\wh L_\infty([a]) := \lim_{k \to \infty}\wh L_k([a]) = \sum_{\sigma^K_i \in B} L_0( \sigma^K_i)/\lambda^{K} 
$$
This proves that the sequence $T_k$ converges to some point $T_{\infty} \in \overline\CV_n$ with corresponding translation length function $\wh L_\infty$. Moreover, the following properties of $\sigma$ are equivalent:
\begin{description}
\item[(i)] $\sigma \in \<Z,\rho\>$
\item[(ii)] $h^K_\#(\sigma) \in \<Z,\rho\>$.
\item[(iii)] No term $\sigma^K_i$ of $h^K_\#(\sigma)$ is an edge in $H$.
\item[(iv)] $L_\infinity([a]) = 0$, that is, the conjugacy class $[a]$ is elliptic for the action of $F_n$ on~$T_\infinity$.
\end{description}
The only possibly unclear equivalence in this list is (i)$\iff$(ii). We use the fact that any circuit in $\<Z,\rho\>$ is clearly fixed by $h_\#$, and therefore by all positive powers $h^k_\#$. This immediately proves (i)$\implies$(ii). Conversely, suppose $h^K_\#(\sigma) \in \<Z,\rho\>$. It follows that $h^{K+1}_\#(\sigma) = h_\#(h^K_\#(\sigma)) = h^K_\#(\sigma)$, and since $h^K_\#$ is a bijection on circuits it follows that $h_\#(\sigma)=\sigma$, implying that $\sigma  = h^K_\#(\sigma) \in \<Z,\rho\>$.

\bigskip
 
We prove that $\A_\na(\Lambda)$ is realized in $T_\infinity$, that is, $\A_\na(\Lambda) = \A_{T_\infinity}$. The equivalence of (i) and (iv), coupled with Lemma~\ref{ZP is closed}~\pref{Item:circuits}, shows that $\A_\na(\Lambda)$ and $\A_{T_\infinity}$ carry the same conjugacy classes. By applying Lemma~\ref{LemmaVSElliptics}, to prove that $\A_\na(\Lambda) = \A_{T_\infinity}$ it remains to show that for each $[A_i] \in \A_\na(\Lambda)$, the subgroup $A_i$ is a maximal elliptic subgroup for the action of $F_n$ on~$T_\infinity$, and we do this in two subcases. Let $v$ denote the base point of $\rho$.

\smallskip

\textbf{Subcase (1):} Suppose that $[A_i]$ is represented by the subgroup $\pi_1(Z_i,w_i)$ of $\pi_1(G,w_i)$, where $Z_i$ is a component of $Z$ such that $v \not\in Z_i$, and $w_i$ is a basepoint in $Z_i$.  Each element of $\pi_1(Z_i,w_i)$ is elliptic in $T_{\infty}$. If $\pi_1(Z_i,w_i)$ is not a maximal elliptic subgroup then there exists $c \in \pi_1(G,w_i)$ that is not in $\pi_1(Z_i,w_i)$ such that every element of $\<\pi_1(Z_i,w_i), c\>$ is elliptic in $T_{\infty}$. Choose $b \in \pi_1(Z_i,w_i)$ such that the circuit $\gamma$ representing $[b c]$ crosses edges in $Z_i$ but is not contained in $Z_i$. Then $\gamma$ is not contained in $\<Z, \rho\>$ in contradiction to the fact that $bc$ is elliptic in $T_\infinity$. 

\smallskip
 
\textbf{Subcase (2):} Suppose that subcase (1) does not hold. Let $Z_i$ be the component of $Z$ that contains $v$ if there is one, or $Z_i = \{v\}$ otherwise. Let $x \in \pi_1(G,v)$ be the element represented by $\rho$. It follows that $[A_i]$ is represented by the subgroup $\<\pi_1(Z_i, v),x\>$ of $\pi_1(G,v)$. As in the previous case, if $c \in \pi_1(G,v)$ is not contained in $\<\pi_1(Z_i, v),x\>$ then there exists $b \in \<\pi_1(Z_i, v),x\>$ such that the circuit representing $[bc]$ crosses $\rho$ but is not contained in $\<Z_i, \rho\>$, and hence is not contained in $\<Z,\rho\>$; if $\<\pi_1(Z_i,v),x,c\>$ is elliptic we again get a contradiction.

\smallskip

This completes Case~3a.

\subparagraph{Case 3b:} We now consider Case 3 in general. To allay some confusion we affix $\phi$ to our notation, writing $\A_\na(\Lambda;\phi)$ for $\A_\na(\Lambda)$.

First note that there exists a single component of $Z \union H_r$ containing the subgraph~$H_r$. This follows from the observation that for any edge $E$ of $H_r$ and any sufficiently high exponent $n$, the connected set $f^n_\#(E) \subset G_r = G_{r-1} \union H_r \subset Z \union H_r$ contains every edge of $H_r$.

List the noncontractible components of $Z \union H_r$ as $C_0,C_1,\ldots,C_K$, letting $C_0$ be the component containing $H_r$. We regard the free group $\pi_1(C_i)$ as a subgroup of $F_n$ with conjugacy class $[\pi_1(C_i)]$ in $F_n$. With this notation we have a free factor system in $F_n$ given by $[Z \union H_r] = \{[\pi_1(C_0)],[\pi_1(C_1)],\ldots,[\pi_1(C_K)]\}$. 

The outer automorphism $\phi_0  = \phi \restrict \pi_1(C_0) \in \Out(\pi_1(C_0))$ is represented by the \ct\ $f \restrict C_0$, which has $H_r$ as an \eg\ geometric stratum with Nielsen path $\rho_r$. As a lamination in the subgraph $C_0 \subset G$, the element of $\L(\phi_0)$ associated to $H_r$ is equal to the lamination $\Lambda$. Also, the nonattracting subgraph of $\Lambda$ with respect to $f \restrict C_0$ is equal to $Z \intersect C_0$. Since $C_0 = (Z \intersect C_0) \union H_r$, it follows that $f \restrict C_0$ falls under the hypothesis of Case~3a, and therefore the nonattracting subgroup system of $\Lambda$ with respect to the action of $\phi_0$, which we shall denote $\A_\na(\Lambda;\phi_0)$, is a vertex group system in $\pi_1(C_0)$. Let $T_0$ denote a very small $\pi_1(C_0)$-tree in which $\A_\na(\Lambda;\phi_0)$ is realized.

In general, consider a free factor $F$ of $F_n$ and a collection of pairwise nonconjugate subgroups $A_1,\ldots,A_I$ in $F$ with conjugacy classes in $F$ defining a subgroup system in $F$ denoted $\{[A_1]_F,\ldots,[A_I]_F\}$. No two of the subgroups $A_i,A_j$, $i \ne j$ are conjugate in~$F_n$, because conjugates of the free factor $F$ by distinct elements of $F_n$ have trivial intersection. We get a subgroup system in $F$ denoted $\{[A_1],\ldots,[A_I]\}$ called the \emph{extension from $F$ to $F_n$ of $\{[A_1]_F,\ldots,[A_I]_F\}$}, and we get a bijection between these two sets.

Let $\A_0$ be the extension of $\A_\na(\Lambda;\phi_0)$ from $\pi_1(C_0)$ to $F_n$. By applying the construction of the nonattracting subgroup system to this situation it follows that $\A_\na(\Lambda;\phi) = \A_0 \union \{[\pi_1(C_1)],\ldots,[\pi_1(C_K)]\}$, which we must show is a vertex group system in $F_n$. We do this by constructing a very small $F_n$ tree $T$ in which this subgroup system is realized.

Let $\wh T$ be a simplicial $F_n$-tree with trivial edge stabilizers in which the free factor system $[Z \union H_r]$ is realized. The tree $\wh T$ has $K+1$ vertex orbits with nontrivial stabilizers, represented by vertices $V_0,V_1,\ldots,V_K$ so that the stabilizer of $V_k$ is the subgroup $\pi_1(C_k) \subgroup F_n$. For concreteness, we may take the tree $\wh T$ to be the quotient of the universal cover of $G$, obtained by collapsing to a point each component of the total lift of each of the subgraphs $C_0,C_1,\ldots,C_K$. 

The desired $F_n$-tree $T$ will be constructed by blowing up each of the vertices in the orbit $F_n \cdot V_0$, replacing each vertex in this orbit by a copy of $T_0$. Each edge of $\wh T$ that was attached to a vertex in $F_n \cdot V_0$ will be reattached to a point in the appropriate copy of $T_0$, and these reattachments will be done in an equivariant manner.

There are finitely many orbits of oriented edges in $\wh T$ whose terminal endpoints are in the set $F_n \cdot V_0$, and $F_n$ acts freely on each such orbit. For the choice of $\wh T$ above, there is one such orbit for each oriented edge of $G$ whose terminal endpoint is in $C_0$. Let $\{E_j \suchthat 1 \le j \le J\}$ be a set of representatives of these orbits, such that the terminal endpoint of $E_j$ is attached to $V_0$. The terminal endpoint of $a \cdot E_j$ is therefore attached to $a \cdot V_0$. 

Construct an $F_n$-forest $\wh \F$ by detaching the terminal endpoint of $a \cdot E_j$ from the point $a \cdot V_0$, for each $a \in F_n$ and each $j=1,\ldots,J$. The components of $\wh \F$ are of two types: its discrete components are the individual points of $F_n \cdot V_0$; its nondiscrete components are the metric completions of the components of $\wh T - F_n \cdot V_0$.

Now we say how to replace each point of $F_n \cdot V_0$ by a copy of $T_0$. Let $F_n \cdot V_0 = \{W_i\}_{i \in I}$, in particular $V_0 = W_{i_0}$ for some $i_0 \in I$. The sets $\{a \in F_n \suchthat a \cdot V_0 = W_i\}$, one for each $i \in I$, are precisely the left cosets of $\pi_1(C_0)$, and this defines bijections between $F_n \cdot V_0$, the index set $I$, and the set of left cosets of $\pi_1(C_0)$. For each $i$ choose a coset representative $a_i$, and so $a_i V_0 = W_i$. The action of $F_n$ on $F_n \cdot V_0$ lifts to an action on the forest $I \cross T_0$ defined as follows: for each $a \in F_n$ and each $(i,t) \in I \cross T_0$, there is a unique $i' \in I$ such that the left cosets $a a_i \pi_1(C_0)$ and $a_{i'} \pi_1(C_0)$ are equal, and we define
$$a \cdot (i,t) = (i',a^\inv_{i'} a a_i \cdot t)
$$
Let $\F$ be the $F_n$ forest obtained from $\wh \F$ by removing the discrete $F_n$-orbit $F_n \cdot V_0$ and inserting the $F_n$-forest $I \cross T_0$.

Finally, construct $T$ from $\F$ as follows. For each $j=1,\ldots,J$, the terminal endpoint of $E_j$ in $\wh T$ was attached to $V_0 = W_{i_0}$. Choose an arbitrary point $x_j \in (i_0,T_0)$ and reattach the terminal endpoint of $E_j$ to $x_j$. Then, for each $a \in F_n$, reattach the terminal endpoint $a \cdot E_j$ to $a \cdot x_j \in I \cross T_0$, and note that this is well-defined because the action of $F_n$ on the orbit of $E_j$ is free.

The $F_n$ action on $\F$ descends to an $F_n$ action on $T$. Any two points $p \ne q \in T$ are endpoints of a unique arc in $T$ which may be written uniquely as a finite concatenation of maximal subarcs in nondiscrete components of $\wh F$ and maximal subarcs in components of $I \cross T_0$. Summing the lengths of these subarcs defines the metric $d(p,q)$, which is clearly an $F_n$-invariant $\reals$-tree metric on $T$. By combining minimality of the $F_n$ action on $\wh T$ with minimality of the $\pi_1(C_0)$ action on $T_0$, it follows that the $F_n$-tree $T$ is minimal. Since the action of $\pi_1(C_0)$ on $T_0$ is very small, and since the $F_n$ stabilizers of all edges in $\wh T$ are trivial, it follows that the $F_n$-action on $T$ is very small. By construction of $T$, the subgroups system $\A_\na(\Lambda;\phi) = \A_0 \union \{[\pi_1(C_1)],\ldots,[\pi_1(C_K)]\}$ is realized in $T$, and is therefore a vertex group system.
\end{proof}

By combining Lemma~\ref{PropVerySmallTree} with Lemma~\ref{LemmaVSElliptics} and with Corollaries~\ref{NA is well defined} and~\ref{NA is independent of plusminus}, we obtain:
 
\begin{corollary} \label{CorPMna}
For any rotationless $\phi \in \Out(F_n)$ and any lamination $\Lambda \in \L(\phi)$, the vertex group system $\A_\na(\Lambda)$ depends only on $\phi$ and $\Lambda$, not on the choice of a \ct\ representing $\phi$. Furthermore, for any dual lamination pair $\Lambda^+ \in \L(\phi)$, $\Lambda^- \in \L(\phi^\inv)$ we have $\A_\na(\Lambda^+;\phi) = \A_\na(\Lambda^-;\phi^\inv)$.
\end{corollary}

Based on this corollary, henceforth we will write $\A_\na \Lambda^\pm$ for the vertex group system $\A_\na(\Lambda^+;\phi) = \A_\na(\Lambda^-;\phi^\inv)$ (continuing to suppress $\phi$ from the notation unless otherwise needed).

\paragraph{Remark.} In the context of Proposition~\ref{PropVerySmallTree}, one can prove that if $\Lambda^\pm \in \L^\pm(\phi)$ is a geometric lamination pair then $\A_\na \Lambda^\pm$ is not a free factor system. Using the Stallings surgery argument from the proof of Fact~\ref{FactContrComp}, one sees that the free factor support of the set of conjugacy classes $\{[\bdy_0 S],[\bdy_1 S],\ldots,[\bdy_m S]\}$ also supports every conjugacy class that is carried by $[\pi_1(S)]$, but the only conjugacy classes that are carried by both of the subgroup systems $[\pi_1(S)]$ and $\A_\na \Lambda^\pm$ are $[\bdy_0 S],[\bdy_1 S],\ldots,[\bdy_m S]$ and their powers.

\bigskip

We conclude this section with a result that generalizes Lemma~\ref{FactWeakLimitLines}, which we will need in the proof of Proposition~\ref{prop:WA1}. Given an end $\E$, we say that $\E$ is \emph{carried by $\<Z,\hat\rho_r\>$} if some ray representing $\E$ is in $\<Z,\hat\rho_r\>$.

\begin{lemma} \label{ItemNotZPLimit} Assume the notation of Definitions~\ref{defn:Z}. 
\begin{enumerate}
\item Every sequence $\gamma_i$ of lines in $G$ not carried by $\<Z,\hat \rho_r\>$ has a subsequence that weakly converges to a line $\gamma$ not carried by $\<Z,\hat \rho_r\>$. 
\item The weak accumulation set of every end not carried by $\<Z,\hat \rho_r\>$ contains a line not carried by $\<Z,\hat \rho_r\>$.
\end{enumerate}
\end{lemma}

\begin{proof} If $H_r$ is non-geometric then $\A_\na(\Lambda)$ is a free factor system and the lemma follows from Fact~\ref{FactWeakLimitLines}. 

Suppose that $H_r$ is geometric and so $\<Z,\hat \rho_r\> = \<Z, \rho_r\>$. Let $\rho_r = \alpha * \beta$ be concatenated at the unique illegal turn $t = \{d_\alpha,d_\beta\}$ in $H_r$, where $d_\alpha$ is the terminal direction of $\alpha$ and $d_\beta$ is the initial direction of $\beta$. Let 
$$L = \max\{\Length(\alpha),\Length(\beta)\}
$$
Let $\Sigma$ be the set of all paths of length $\le 2L$ that occur as subpaths of an element of $\<Z,\rho_r\>$. Note that for any path of the form $\alpha' * \beta' \in \Sigma$ such that $\Length(\alpha') = \Length(\alpha)$, $\Length(\beta') = \Length(\beta)$, $d_\alpha$ is the terminal direction of $\alpha'$, and $d_\beta$ is the initial direction of $\beta'$, we have $\alpha' = \alpha$ and $\beta' = \beta$ so $\alpha' * \beta' = \rho_r$.

We claim that if $\gamma \in \wh\B(G)$ and if every subpath of $\gamma$ of length $\le 2L$ is contained in $\Sigma$ then there is a subpath $\gamma'$ of $ \gamma$ that is contained in $\<Z, \rho_r\>$ and that contains all of $\gamma$ with the possible exception of initial and terminal subpaths of length $\le L$. 
 
To prove the claim, let $\ti \gamma$ be a lift of $\gamma$ to the universal cover of $G$. Given $\ti t = \{\ti d_\alpha, \ti d_\beta\}$ an illegal turn in $\ti\gamma$ that projects to $t=\{d_\alpha,d_\beta\}$, we say that $\ti t$ is \emph{buffered} if $\ti\gamma$ contains at least $L$ edges on each side of the turn~$\ti t$; if this is the case then $\ti\gamma$ has a subpath which contains the turn~$\ti t$ and is a lift of $\rho_r$. Suppose that $\ti t_1$ and $\ti t_2$ are buffered illegal turns in $\ti \gamma$ that project to $t$, and that there are no other illegal turns that project to $t$ between $\ti t_1$ and $\ti t_2$. Let $\ti \gamma_i \subset \ti\gamma$ be the lift of $\rho_r$ or $\rho_r^{-1}$ that contains $\ti t_i$, and let $\sigma$ be the subpath of $\ti \gamma$ which starts at the turn $\ti t_1$ and ends at the turn $\ti t_2$. If $\sigma$ has length $\le 2L$ then $\sigma$ is a subpath of a path contained in $\<Z, \rho_r\>$, in which case it follows that $\ti\gamma_1$ and $\ti\gamma_2$ intersect in at most an endpoint; and the same evidently follows if $\sigma$ has length $> 2L$.

We may therefore decompose $\ti \gamma$ into subpaths $\ti \gamma_i$ as follows: any subpath of $\ti \gamma$ that projects to either $\rho_r$ or $\rho_r^{-1}$ is a $\ti \gamma _i$; each remaining edge is a $\ti\gamma_i$. If $\ti \gamma$ has an initial vertex, then remove the initial segment preceding the first subpath that projects to $\rho_r$ or $\rho_r^{-1}$ if this segment has length $<L$ and remove the initial segment of length $L$ otherwise. Treat the terminal vertex, if there is one, similary. The resulting subpath of $\ti \gamma$ projects to the desired subpath $ \gamma'$ of $\gamma$. This completes the proof of the claim. 

 If $\gamma_i \in \B(G)$ is a sequence of lines not in $\<Z, \rho_r\>$ then it follows from the claim that each $\gamma_i$ has a subpath $\alpha_i$ of edge length $\le 2L$ that is not in $\Sigma$. There are only finitely many paths of edge length $\le 2L$, so by passing to a subsequence we may assume that $\alpha=\alpha_i$ is independent of $i$. It follows that this subsequence has a weak limit which is a line containing the path $\alpha$, and this line is not in $\<Z,\rho_r\>$.
 
 If $r$ is a ray in $G$ no subray of which is contained in $\<Z,\rho_r\>$ then, focussing on subrays obtained by deleting initial paths of length $> L$, it follows from the claim that $r$ has infinitely many pairwise nonoverlapping subpaths $\alpha_i$ of edge length $\le 2L$ that are not in $\Sigma$. Again, passing to a subsequence, we may assume that $\alpha=\alpha_i$ is independent of $i$. It follows that $r$ has a weak limit which is a line containing the path $\alpha$, and this line is not in $\<Z,\rho_r\>$.
\end{proof}

 \subsection{Statements of weak attraction results}
\label{SectionWAResults}

Our main weak attraction results are the following two propositions, the first of which applies to any attracting lamination $\Lambda$, and the second to any $\Lambda$ whose nonattracting vertex group system $\A_\na(\Lambda)$ is trivial.

\begin{proposition}\label{prop:WA1} Suppose that $\phi \in \Out(F_n)$ is rotationless and that $\Lambda^+ \in \L(\phi)$ is paired with $\Lambda^- \in \L(\phi^{-1})$. For each line $\gamma \in \B$ at least one of the following holds:
\begin{enumerate}
\item $\gamma$ is weakly attracted to $\Lambda^+$.
\item $\gamma$ is carried by $\A_{na}(\Lambda^+)$. 
\item The weak closure of $\gamma$ contains $\Lambda^-$.
\end{enumerate}
Moreover, if $V^+$ and $V^-$ are neighborhoods of generic leaves of $\Lambda^+$ and $\Lambda^-$ respectively, then there exists an integer $m \ge 1$ such that each $\gamma \in \B$ satisfies at least one of the following: \ $\gamma \in V^-$; \ $\phi^m(\gamma) \in V^+$; \ or $\gamma$ is carried by $\A_{na}(\Lambda^+)$. 
\end{proposition}

For the next proposition, given a rotationless $\phi \in \Out(F_n)$ and $\Lambda^+ \in \L(\phi)$ such that $\A_\na(\Lambda^+)$ is trivial, define the set of lines
$$\LS(\phi) = S_\phi \union \bigcup_{\Lambda \in \L(\phi)} (\text{generic leaves of $\Lambda$})
$$


\begin{proposition} \label{prop:WA2} Suppose that $\phi \in \Out(F_n)$ is rotationless, that $\Lambda^\pm \in \L^\pm(\phi)$, and that $\A_{na}(\Lambda^\pm)$ is trivial. For each line $\gamma \in \B$ exactly one of the following occurs:
\begin{enumerate}
\item $\gamma$ is weakly attracted to $\Lambda^+$.
\item $\gamma \in \LS(\phi)$. 
\end{enumerate}
\end{proposition}

\begin{remark} The Weak Attraction Theorem, Theorem~6.0.1 of \BookOne, has stronger hypotheses and stronger conclusions than Proposition~\ref{prop:WA1}. The additional hypotheses are that $\gamma$ is birecurrent and that $\Lambda^+$ is topmost.
The additional conclusions are that the three items are mutually exclusive and (3) is replaced by ($3'$): $\gamma$ is a generic leaf of $\Lambda^-$. 

On the other hand, under the added assumption that $\A_\na \Lambda^\pm$ is trivial, Proposition~\ref{prop:WA2} implies the Weak Attraction Theorem. To see why, first note that every birecurrent element of $S_{\phi^\inv}$ is a generic leaf of some $\Lambda' \in \L(\phi^\inv)$. Furthermore, assuming that $\Lambda^+$ is topmost it must be the unique element of $\L(\phi)$ and so its paired lamination $\Lambda^-$ is the unique element of $\L(\phi^\inv)$: if there were a lamination in $\L(\phi)$ different from $\Lambda^+$ then its leaves would be carried by $\A_\na \Lambda^\pm$, contradicting that $\A_\na \Lambda^\pm$ is trivial.
\end{remark}

\subsection{Nonattracted lines: the full height case.}
\label{SectionNonattrFullHeight}

The proofs of Propositions~\ref{prop:WA1} and~\ref{prop:WA2} follow the proof of the Weak Attraction Theorem (Theorem~6.0.1 of \BookOne). 

Lemma~\ref{defining Z} above, which characterizes finite paths with fixed endpoints that are not weakly attracted to $\Lambda^+$, is the analog of Proposition~6.0.4 of \BookOne. 

We now turn to the analog of Proposition~6.0.8 of \BookOne\ in which, given $\phi \in \Out(F_n)$ and $\Lambda^+ \in \L(\phi)$, one is concerned with birecurrent lines $\gamma$ of height $\le$ the height of $\Lambda^+$ which are not weakly attracted to $\Lambda^+$. The proof of Proposition~6.0.8 in \BookOne\ is separated into geometric and non-geometric cases. We drop the hypothesis that $\gamma$ be birecurrent, and to make the statement more digestible we separate the two cases, non-geometric in Lemma~\ref{nonGeometricFullHeightCase} and geometric in Lemma~\ref{geometricFullHeightCase}. The conclusions of these two lemmas describe $\gamma$ in very explicit terms. The key to understanding the conclusions is that one can have a path $\mu$ of height $\le r$ whose intersection with $H_r$ stays constant under iteration, and which can interpolate between various rays, either rays of the dual lamination $\Lambda^-$ or rays of lower height, in order to produce a line that is not weakly attracted to $\Lambda^+$.


\bigskip

Lemmas \ref{nonGeometricFullHeightCase} and \ref{geometricFullHeightCase} share the following setup. Let $\phi,\phi^\inv \in \Out(F_n)$ be rotationless with dual lamination pair $\Lambda^\pm \in \L^\pm(\phi)$. Let $f \from G \to G$ be a \ct\ representing $\phi$ with \eg\ stratum $H_r$ associated to $\Lambda^+$. Applying Theorem~\ref{TheoremCTExistence}, let $f' \from G' \to G'$ be a \ct\ representing $\phi^\inv$ with \eg\ stratum $H'_s$ associated to $\Lambda^-$ such that $[G_r] = [G'_s]$.

\begin{lemma} \label{nonGeometricFullHeightCase} With the notation above, suppose that the stratum $H_r$ is not geometric.  If $\gamma \subset G_r$ is a line of height $r$ that is not weakly attracted to $\Lambda^+$ then the realization of $\gamma$ in $G'$, denoted $\gamma' \subset G'_s$, satisfies at least one of the following:
\begin{enumerate}
\item $\gamma'$ is a generic leaf of $\Lambda^-$.
\item \label{ItemIntermediatePath}
$\gamma'$ decomposes as $\overline R_1 \mu R_2$ where $R_1$ and $R_2$ are singular rays for $\Lambda^-$ and $\mu$ is either the trivial path or a nontrivial path of one of the forms $\alpha$, $\beta$, $\alpha\beta$, $\alpha\beta\bar\alpha$, such that $\alpha$ is a height $s$ indivisible Nielsen path and $\beta$ is a nontrivial path of height $<s$. 
\item $\gamma'$ or $\gamma'{}^\inv$ decomposes as $\overline R_1 \mu R_2$ where $R_1$ is a singular ray for $\Lambda^-$, $R_2$ is a ray of height $<s$, and $\mu$ is either trivial or a height $s$ Nielsen path. 
\end{enumerate}
\end{lemma}

\begin{proof} By Proposition~\ref{PropGeomEquiv}, since $H_r$ is not a geometric stratum, neither is $H'_s$. Let $\alpha$ denote the indivisible Nielsen path of height $s$ in $G'_s$, if it exists, so $\alpha$ is not closed; by Fact~\ref{FactEGStrata} item~\pref{ItemSomeOnce} we may orient $\alpha$ so that its initial endpoint $v$ is an interior point of $H'_s$.

We first show that $\gamma$ has infinitely many edges in $H_r$ by proving that if a line $\gamma$ has height $r$ and only finitely many edges in $H_r$ then $\gamma$ is weakly attracted to $\Lambda^+$. 
To prove this, write $\gamma = \gamma_- \gamma_0 \gamma_+$ where $\gamma_-,\gamma_+ \subset G_{r-1}$ and $\gamma_0$ is a finite path whose first and last edges are in $H_r$. Since $f_\#$ restricts to a bijection on lines of height $r-1$, it follows that $f^k_\#(\gamma)$ has height $r$ for all $k$, and so $f^k_\#(\gamma_0)$ is a nontrivial path of height $r$ for all $k$. By Fact~\ref{FactEvComplSplit} there exists $K \ge 0$ such that $f^K_\#(\gamma_0)$ completely splits into terms each of which is either an edge or Nielsen path of height $r$ or a path in $G_{r-1}$, with at least one term of height $r$. Let $f^K_\#(\gamma_0) = \gamma'_- \gamma'_0 \gamma'_+$ where $\gamma'_-,\gamma'_+$ are in $G_{r-1}$ and $\gamma'_0$ is the maximal subpath of $f^K_\#(\gamma_0)$ whose first and last edges are in $H_r$, so $\gamma'_0$ is nontrivial and completely split. Since $f^K_\#(\gamma) = [f^K_\#(\gamma_-) \, \gamma'_- \, \gamma'_0 \, \gamma'_+ \, f^K_\#(\gamma_+)]$ then, using RTT-(i), it follows that there is a splitting $f^K_\#(\gamma) = \gamma''_- \cdot \gamma'_0 \cdot \gamma''_+$ where $\gamma''_- = [f^K_\#(\gamma_-) \, \gamma'_-]$ and $\gamma''_+= [\gamma'_+ \, f^K_\#(\gamma_+)]$ are in $G_{r-1}$. By Fact~\ref{FactEGStrata} item~\pref{ItemSomeOnce}, each term in this splitting of $f^K_\#(\gamma)$ which is an indivisible Nielsen path of height $r$ is adjacent to a term that is an edge in $H_r$. It follows that at least one term in the splitting of $f^K_\#(\gamma)$ is an edge in $H_r$, implying that $\gamma$ is weakly attracted to $\Lambda^+$. 

Since $\gamma$ contains infinitely many edges in $H_r$, and since $[G_r]=[G'_s]$ and the graphs $G_r, G'_s$ are both core subgraphs, the line $\gamma'$ contains infinitely many edges in $H'_s$. 

In the part of the proof of the nongeometric case of Proposition~6.0.8 of \BookOne\ that does not use birecurrence and so is true in our context, it is shown that there exists $M' > 0$ so that for every finite subpath $\gamma_i'$ of $\gamma'$ there exists a line or circuit $\tau_i'$ in $G'$ that contains at most $M'$ edges of $H'_{s}$ such that $\gamma_i'$ is a subpath of $g_\#^{k_i}(\tau_i')$ for some $k_i\ge 0$. If $G_{r-1} = \emptyset$ then $\tau_i'$ is a circuit; otherwise $\tau_i'$ is a line. (This is proved in two parts. First, in what is called step 2 of that proof, an analogous result is proved in $G$. Then the bounded cancellation lemma is used to transfer this result to $G'$; the case that $G_{r-1} = \emptyset$ is considered after the case that $G_{r-1} \ne \emptyset$.)

Choose a sequence of finite subpaths $\gamma'_i$ of $\gamma'$ that exhaust $\gamma'$ and let $\tau'_i$ and $k_i$ be as above so that $\gamma'_i$ is a subpath of $g^{k_i}_\#(\tau'_i)$ and so that $\tau'_i$ contains at most $M'$ edges of~$H'_s$. Since $\gamma'$ contains infinitely many $H'_s$ edges, we have $k_i \to +\infinity$ as $i \to +\infinity$. 

By Lemma~\ref{LemmaEGUnifPathSplitting} there exists $d>0$ depending only on the bound $M'$ such that $g^d_\#(\tau'_i)$ has a splitting into terms each of which is either an edge or indivisible Nielsen path of height $s$ or a path in $G'_{s-1}$. By taking $i$ so large that $k_i \ge d$ we may replace each $\tau'_i$ by $g^d_\#(\tau'_i)$ and each $k_i$ by $k_i - d$, and hence we may assume that $\tau_i'$ has a splitting 
$$ \tau_i' = \tau'_{i,1} \cdot \ldots \cdot \tau'_{i,l_i}
$$
each of whose terms is an edge or Nielsen path of height $s$ or a path in $G'_{s-1}$. The number of edges that $\tau_i'$ has in $H_s$ is still uniformly bounded, and so $l_i$ is uniformly bounded. Passing to a subsequence, we may assume that $l_i$ and the ordered sequence of height $t$ terms in $\tau'_i$ are independent of~$i$. We may also assume that $l = l_i$ is minimal among all such choices of $\gamma_i'$ and $\tau_i'$.

\subparagraph{Case 1: $l = 1$.} In this case $\tau'_i = E$ is a single edge of $H'_s$ and so, by Fact~\ref{FactAttractingLeaves}~\pref{ItemLeafAsLimit}, $\gamma'$ is a leaf of $\Lambda^-$. If both ends of $\gamma'$ have height $s$ then $\gamma'$ is generic by Fact~\ref{FactTwoEndsGeneric} and case (1) is satisfied. 

Suppose one end of $\gamma'$, say the positive end, has height $\le s-1$. We have $\gamma' = \overline R_1 R_2$ where the ray $R_1$ starts with an edge of $H'_s$ and the ray $R_2$ is contained in $G'_{s-1}$. By dropping finitely many terms of the sequence we may assume $\gamma'_i = \overline R_{1i} R_{2i}$ where $R_{1i}, R_{2i}$ are nontrivial initial segments of $R_1,R_2$. We may subdivide $E = \bar e_1 e_2$ at an interior point so that if $1 \le j \le k_i$ then $g^j_\#(E) = g^j_\#(\bar e_1) g^j_\#(e_2)$, and so that $R_{1i}$ is an initial segment of $g^{k_i}_\#(e_1)$ and $R_{2i}$ is an initial segment of $g^{k_i}_\#(e_2)$. Let $j_i$ be the first exponent such that the subdivision point of $g^{j_i}_\#(E) = g^{j_i}_\#(\bar e_1) g^{j_i}_\#(e_2)$ is a vertex. By RTT-(i), the first edge of $g^{j_i}_\#(e_1)$ is an oriented edge $E'_i$ of $H'_s$, and the first 
term of the splitting of $g^{j_i}_\#(e_2)$ into edges of $H'_s$ and maximal subpaths of $G'_{s-1}$ is a subpath $\sigma'_i$ of $G'_{s-1}$. Furthermore, letting $E''_i$ be the edge of $H'_s$ whose interior contains the subdivision point of $g^{j_i-1}_\#(E) = g^{j_i-1}_\#(\bar e_1) g^{j_i-1}_\#(e_2)$, it follows that $\sigma'_i$ is a term in the splitting of $g_\#(E''_i)$ into edges of $H'_s$ and maximal subpaths in $G'_{s-1}$; this shows that there are only finitely many possibilities for the path $\sigma'_i$. Passing to a subsequence, we may assume that $E'_i = E'$ and $\sigma'_i = \sigma'$ are independent of $i$. Since the length of $R_{2i}$ goes to $\infinity$ as $i \to \infinity$, and since $R_{2i}$ is an initial subpath of $g^{k_i-j_i}(\sigma')$, we have $k_i - j_i \to \infinity$. It follows that $R_1$ is the singular ray corresponding to the principal direction $E'$. This shows that (3) is satisfied with trivial $\mu$.

\subparagraph{Case 2: $l \ge 2$.} Choose a subpath $\nu_i' \subset \tau_i'$, with endpoints not necessarily at vertices, such that $g^{k_i}_\#(\nu_i') = \gamma_i'$. Let $\tau''_i$ be the subpath obtained from $\tau'_i$ by removing the initial segment $\tau'_{i,1}$ and the terminal segment $\tau'_{i,l}$, so either $\tau''_i = \tau'_{i,2} \cdot \ldots \cdot \tau'_{i,l-1}$ or, when $l=2$, $\tau''_i$ is the the trivial path at the common vertex along which $\tau'_{i,1}$ and $\tau'_{i,2}$ are concatenated. After passing to a subsequence, we may assume that $\tau''_i \subset \nu_i'$; if not then we could reduce~$l$ by removing either $\tau'_{i,1}$ or $\tau'_{i,l}$ from $\tau'_i$. For the same reason, we may assume that $\gamma'$ has a finite subpath that contains $g^{k_i}_\#(\tau''_i)$ for all $i$. After passing to a subsequence, we may assume that $\mu = g^{k_i}_\#(\tau''_i)$ is independent of $i$. In particular $\mu$ is either trivial or has a splitting into terms each of which is either $\alpha$, or $\bar\alpha$, or a path in $G'_{s-1}$. Since the endpoints of $\alpha$ are distinct, no two adjacent terms in this splitting can both be $\alpha$ or $\bar\alpha$, and so each subdivision point of the splitting is in $G'_{s-1}$. Since $v$ is an interior point of $H'_s$, for any occurence of $\alpha$ or $\bar\alpha$ as a term of $\mu$ the endpoint $v$ must be an endpoint of $\eta$. It follows that $\mu$ can be written in one of the forms given in item~\pref{ItemIntermediatePath}.

Write $\gamma'$ as $\overline R_1 \mu R_2$. If $\tau'_{1,1}$ is an edge $E$ in $H'_s$ then $E = \tau'_{i,1}$ for all $i$, and the ray $R_1$ is the increasing union of $g^{k_1}_\#(\bar E) \subset g^{k_2}_\#(\bar E) \subset \cdots$, so $R_1$ is a singular ray for $\Lambda^-$. Otherwise $\tau'_{i,1}$ is a path in $G'_{s-1}$ for all $i$ and $R_1$ is a ray in $G'_{s-1}$. Using $\tau'_{i,l}$ similarly in place of $\tau'_{i,1}$, $R_2$ is either a singular ray for $\Lambda^-$ or a ray in $G'_{s-1}$. At least one of $R_1$ and $R_2$ is a singular ray. If they are both singular rays then we are in case (2), otherwise we are in case (3). 
\end{proof}

\begin{lemma} \label{geometricFullHeightCase} Continuing with the notation laid out before Lemma~\ref{nonGeometricFullHeightCase}, suppose that the stratum $H_r$ is geometric. Let $\rho_r$ be the height $r$ closed indivisible Nielsen path in $G_r$. Knowing by Proposition~\ref{PropGeomEquiv} that $H'_s$ is geometric, let $\rho'_s$ be the closed indivisible Nielsen path of height $s$ in $G'_s$. If $\gamma \subset G_r$ is a line of height $r$ that is not weakly attracted to $\Lambda^+$ then the realization of $\gamma$ in $G'$, denoted $\gamma' \subset G'_s$, has at least one of the following forms:
\begin{enumerate}
\item \label{ItemGFHCIterate}
$\gamma'$ or $\bar\gamma'$ is the bi-infinite iterate of $\rho'_s$.
\item \label{ItemGFHCLeaf}
$\gamma'$ is a generic leaf of $\Lambda^-$.
\item \label{ItemGFHCTwoRays}
$\gamma'$ decomposes as $\overline R_1 \mu R_2$ where $R_1$ and $R_2$ are singular rays for $\Lambda^-$ and $\mu$ is either the trivial path, or a finite iterate of $\rho'_s$, or a nontrivial path of height $<s$.
\item \label{ItemGFHCOneRay}
$\gamma'$ or $\bar\gamma'$ decomposes as $\overline R_1 R_2$ where $R_1$ is a singular ray for $\Lambda^-$, and the ray $R_2$ either has height less than~$s$ or is the singly infinite iterate of $\rho'_s$.
\end{enumerate}
\end{lemma}

\begin{proof} For the proof of this lemma we continue with the notation used in Section~\ref{SectionGeometric}. Since $\gamma$ has height $r$ we know that it is not the bi-infinite iterate of any of $\alpha_1,\ldots,\alpha_M$. Also, we may assume that \pref{ItemGFHCIterate} does not hold, and so $\gamma$ is not the bi-infinite iterate of $\rho_r$. There exists, therefore, an $S$-proper quasiline $\gamma^*$ in $Y$ corresponding to~$\gamma$. Since $\gamma$ has height $r$, the decomposition of $\gamma^*$ contains at least one piece which is a proper geodesic in $S$. Since $\gamma$ is not weakly attracted to $\Lambda^+$ under iteration of~$\phi$, it follows that $\gamma^*$ is not weakly attracted to $\Lambda^u$ under iteration of $h_\#$, and so no piece of $\gamma^*$ which is a proper geodesic in $S$ is weakly attracted to $\Lambda^u$ under iteration of $\psi \in \MCG(S)$. Applying Fact~\ref{FactPsAnWeakAttr} to each such piece, we see that none of them is a finite proper geodesic, so either $\gamma^*$ is entirely contained in $S$ or its pieces which are proper geodesics in $S$ consist of one or two proper geodesic rays, and furthermore one of the following cases holds. 

In one case, $\gamma^*$ is a leaf of $\Lambda^s$, in which case conclusion~\pref{ItemGFHCLeaf} holds. 

In another case, $\gamma^*$ is a bi-infinite geodesic contained in a principal region $P$ of~$\Lambda^s$. If $P$ is an ideal polygon then conclusion~\pref{ItemGFHCTwoRays} holds with $\mu$ trivial. Suppose that $P$ is a peripherial crown, in which case at least one end of $\gamma^*$ converges to an ideal point of~$P$. If $\bdy_0 S \subset P$ then: if two ends of $\gamma^*$ converge to ideal points then \pref{ItemGFHCTwoRays} holds with $\mu$ trivial or a finite iterate of $\rho'_s$; if one end converges to an ideal point then \pref{ItemGFHCOneRay} holds with $R_2$ a singly infinite iterate of $\rho'_s$. If one of $\bdy_1 S, \ldots,  \bdy_m S$ is contained in $P$ then: if two ends of $\gamma^*$ converge to ideal points then \pref{ItemGFHCTwoRays} holds with $\mu$ trivial or of height $<s$; and if one end converges to an ideal point then \pref{ItemGFHCOneRay} holds with $R_2$ a singly infinite iterate of one of $\alpha_1,\ldots,\alpha_M$ and so of height $<s$.

In another case $\gamma^*$ has two pieces which are proper geodesics in $S$, each a ray contained in a peripheral crown of $\Lambda^s$, connected by a piece which is a finite geodesic in $G_r \union A$. In this case conclusion~\pref{ItemGFHCTwoRays} holds with $\mu$ of height $<s$.

In the remaining case $\gamma^*$ has one piece which is a proper geodesic ray in a peripheral crown of $\Lambda^s$, the remaining piece being a ray in $G_r \union A$, in which case conclusion~\pref{ItemGFHCOneRay} holds with $R_2$ of height $<s$.

\end{proof}

\subsection{Nonattracted lines: the general case}
\label{SectionNonattrGeneral}

In this section we prove Propositions~\ref{prop:WA1} and~\ref{prop:WA2}.

\begin{proof}[Proof of Proposition~\ref{prop:WA1}] Choose a \ct\ $\fG$ representing $\phi$ with \eg\ stratum $H_r$ corresponding to $\Lambda^+$. Let $Z$ and $\hat \rho_r$ in $G$ be as in Definition~\ref{defn:Z}. Recall from Lemma~\ref{ZP is closed}\pref{item:ZP=NA} that the set of lines carried by $\<Z,\hat \rho_r\>$ is the same as the set of lines carried by $\A_{na}(\Lambda^\pm)$. 

First we show that the uniform version of the proposition, stated in the sentence ``Moreover\ldots'', is a consequence of the nonuniform version stated before that sentence. Arguing by contradiction, if the uniform version fails then there are neighborhoods $V^+,V^-$ of generic leaves of $\Lambda^+,\Lambda^-$ respectively, a sequence of lines $\gamma_i \in \B$, and a sequence of positive integers $m_i \to +\infinity$, such that for all $i$ we have: $\gamma_i \not\in V^-$; $f_\#^{m_i}(\gamma_i) \not\in V^+$; and $\gamma_i$ is not carried by $\A_\na(\Lambda^+)$. We may assume that $V_+$ has the property $f_\#(V_+) \subset V_+$, because generic leaves of $\Lambda^+$ have a neighborhood basis of such sets. By Lemma~\ref{ItemNotZPLimit}, some subsequence of $\gamma_i$ has a weak limit $\gamma'$ that is not carried by $\A_\na(\Lambda^+)$. Since $V_-$ is open, $\gamma' $ is not contained in $V_-$ and so the weak closure of $\gamma'$ does not contain $\Lambda^-$. To obtain a contradiction with the nonuniform version, it remains to show that $\gamma'$ is not weakly attracted to $\Lambda^+$, i.e.\ that the sequence $f_\#^m(\gamma')$ does not weakly converge to $\Lambda^+$. If it does then $f_\#^M(\gamma') \in V^+$ for some~$M$. Since $V^+$ is open there exists $I$ such that $f_\#^M(\gamma_i) \in V^+$ for all $i \ge I$. Since $f_\#(V^+) \subset V^+$, it follows that $f_\#^m(\gamma_i) \in V^+$ for all $m \ge M$ and $i \ge I$. We can choose $i \ge I$ so that $m_i \ge M$, and it follows that $f_\#^{m_i}(\gamma_i) \in V^+$, a contradiction. 

\bigskip

We turn now to the proof that an arbitrary line $\gamma$ satisfies one of the conclusions (1), (2) or (3), arguing by induction on the height of $\gamma$. 

Applying Theorem~\ref{TheoremCTExistence}, choose a \ct\ $g \from G' \to G'$ representing $\phi^\inv$ such that the free factor systems represented by the filtration elements of $G$ and of $G'$ are the same. Letting $H'_s$ be the \eg\ stratum corresponding to $\Lambda^-$, we have $[G_r] = [G'_s]$. By Proposition~\ref{PropGeomEquiv}, the stratum $H_r$ is geometric if and only if the stratum $H'_s$ is geometric, and if so then their closed indivisible Nielsen paths $\rho_r$, $\rho'_s$ represent the same conjugacy class up to orientation reversal. 

The basis step of the induction is the case where $\gamma$ has height $\le r$, in which case the realization of $\gamma$ in $G'$ is contained has height $\le s$. Suppose that conclusion (1) does not hold: $\gamma$ is not weakly attracted to $\Lambda^+$. In the case that the strata $H_r$ and $H'_s$ are not geometric, by applying Lemma~\ref{nonGeometricFullHeightCase} it follows that either conclusion (3) holds, i.e.\ the weak closure of $\gamma$ contains $\Lambda^-$, or $\gamma$ has height $<r$ in~$G$, i.e.\ $\gamma$ is carried by $G_{r-1} = Z \intersect G$ and so conclusion (2) holds. In the case that the strata $H_r$ and $H'_s$ are geometric, by applying Lemma~\ref{geometricFullHeightCase} it follows that either conclusion (3) holds or the realization of $\gamma$ in $G$ is either in $G_{r-1}$ or is the bi-infinite iterate of $\rho_r$ or $\bar\rho_r$ and so conclusion (2) holds.

We assume now that $\gamma$ has height $u>r$ and that the proposition holds for all lines of height $< u$. Since no line has height corresponding to a zero stratum, the stratum $H_u$ is irreducible.

Let $\E_-$ and $\E_+$ be the ends of $\gamma$. 

\paragraph{Case 1: $\E_-$ and $\E_+$ are carried by $\<Z,\hat \rho_r\>$.} It follows that $\gamma$ has the form $\overline R_- \gamma_0 R_+$ where $R_-, R_+\in \<Z,\hat \rho_r\>$ are rays and $\gamma_0$ is a finite subpath. After replacing $\gamma$ with $f^m_\#(\gamma)$ for some $m \ge 0$, we may assume that the endpoints of $\gamma_0$ are fixed. In the case that $\gamma_0 \in \<Z,\hat\rho_r\>$, it follows that $\gamma \in \<Z,\hat\rho_r\>$ by Lemma~\ref{ZP is closed}~\pref{Item:Groupoid}, and so conclusion (2) holds by Lemma~\ref{defining Z}~\pref{ItemZPAnyPaths}. In the case that $\gamma_0 \not \in \<Z,\hat \rho_r\>$, it follows that $\gamma_0$ is weakly attracted to $\Lambda^+$ by Lemma~\ref{defining Z}~\pref{ItemZPFinitePaths}, and so $f^k_\#(\gamma_0)$ contains longer and longer subpaths of $\Lambda^+$ as $k$ increases. If conclusion (1) does not hold then these long subpaths of $\Lambda^+$ must be cancelled when $f^k_\#(R_-) f_\#^k(\gamma_0) f^k_\#( R_+)$ is tightened to $f^k_\#(\gamma)$. It follows that at least one of $R_-$ or $R_+$ must be weakly attracted to $\Lambda^+$, but by Lemma~\ref{defining Z}~\pref{ItemZPAnyPaths} this contradicts the assumption that $R_-, R_+\in\<Z,\hat \rho_r\>$, thereby demonstrating that conclusion (1) indeed holds.

\paragraph{Case 2: $\E_+$ (or $\E_-$) is not carried by $\<Z,\hat \rho_r\>$ and has height less than $u$.} By Lemma~\ref{ItemNotZPLimit}, the accumulation set of the end $\E_+$ of $\gamma$ contains a line $\sigma \not \in \<Z,\hat \rho_r\>$, and $\sigma$ has height less than~$u$. The inductive hypothesis therefore applies to $\sigma$. It follows that either $\sigma$, and hence $\gamma$ by Remark~\ref{basin is open}, is weakly attracted to $\Lambda^+$; or the weak closure of $\sigma$, and hence the weak closure of $\gamma$, contains~$\Lambda^-$.

\paragraph{Case 3: $\E_+$ (or $\E_-$) is not carried by $\<Z,\hat \rho_r\>$ and has height $u$.} Suppose that at least one end of $\gamma$, say $\E_+$, has height $u$ and is not in $\<Z,\hat \rho_r\>$. There are three subcases. 

In the first subcase, $H_u$ is an \noneg\ stratum. Applying \BookOne\ Lemma~4.1.4 it follows that $\gamma$ has a splitting in which ${\cal E}_+$ splits as a concatenation of finite subpaths with fixed endpoints. Lemma~\ref{defining Z} implies that $\gamma$ is weakly attracted to $\Lambda^+$. 


In the second subcase, $H_u$ is \eg\ and $H_u \not\subset Z$. Let $\Lambda'{}^+ \in \L(\phi)$ correspond to $H_u$ and let $\Lambda'{}^- \in \L(\phi^\inv)$ be its dual lamination. It follows that $\Lambda^+ \subset \Lambda'{}^+$, and then by Lemma~\ref{containmentSymmetry} that $\Lambda^- \subset \Lambda'{}^-$. Applying Lemma~\ref{ItemNotZPLimit}, let $\sigma$ be a line in the accumulation set of $\E_+$ such that $\sigma \not\in \<Z,\hat \rho_r\>$. If $\sigma$ has height $<u$ then we proceed as in Case~2, applying the inductive hypothesis to $\sigma$ to obtain the desired conclusion about $\gamma$. Suppose then that $\sigma$ has height~$u$. Letting $H'_v$ be the \eg\ stratum of $G'$ corresponding the $\Lambda'{}^-$, we have $[G_u] = [G'_v]$. By Proposition~\ref{PropGeomEquiv}, $G_u$, $G'_v$ are either both geometric or both nongeometric. Applying Lemma~\ref{nonGeometricFullHeightCase} or Lemma~\ref{geometricFullHeightCase} as appropriate, the conclusion is that either $\sigma$ is weakly attracted to $\Lambda'{}^+$, implying that $\gamma$ is weakly attracted to $\Lambda'{}^+$ and so also to $\Lambda^+$, or the weak closure of $\sigma$ contains~$\Lambda'{}^-$, implying that the weak closure of $\gamma$ contains $\Lambda'{}^-$ and so also contains~$\Lambda^-$. 

In the third and final subcase, $H_u$ is \eg\ and $H_u \subset Z$. By case 2, we may assume that ${\cal E}_-$ either has height $u$ or is carried by $\<Z,\hat \rho_r\>$, and so there are two subsubcases. 

In the first subsubcase, ${\cal E}_-$ is carried by $\<Z,\hat \rho_r\>$. We decompose $\gamma$ as follows. First, write $\gamma = \overline R_- R_+$ where $\overline R_-$ is the maximal ray representing $\E_-$ such that $\overline R_- \in \<Z,\hat \rho_r\>$. Next, write $R_+ = \alpha'_1 \beta'_1 \alpha'_2 \beta'_2 \ldots$ where each $\alpha'_j$ has height $<u$ and each $\beta'_j$ is in $H_u \subset Z$; note that by maximality of $R_-$, $\alpha'_1$ is nontrivial and $\alpha'_1 \not\in \<Z,\hat \rho_r\>$. Every other $\alpha'_j$ and $\beta'_j$ is also nontrivial, since $\E_+$ has height $u$, $H_u \subset Z$, and $\E_+$ is not carried by $\<Z,\hat\rho_r\>$. Next, for each $j$, if $\alpha'_j \in \<Z,\hat \rho_r\>$ then incorporate it into the adjacent~$\beta$'s, resulting in a nontrivial decomposition $R_+ = \alpha_1 \beta_1 \alpha_2 \beta_2 \ldots$ where each $\alpha_j \not\in \<Z,\hat \rho_r\>$ has height $<u$, and each $\beta_j \in \<Z,\hat \rho_r\>$ has height~$u$ and has initial and final directions in~$H_u$. 

By (Zero Strata) and the definition of $Z$, none of the paths $\alpha_j$ is contained in a zero stratum, and so by Fact~\ref{FactUsuallyPrincipal} each concatenation point of an $\alpha$ and a $\beta$ is principal. Also, by maximality of $\overline R_-$ and Fact~\ref{FactUsuallyPrincipal}, the concatenation point of $\overline R_-$ and $\alpha_1$ is principal. Applying (Rotationless), the endpoints of each $\alpha_j$ are fixed.

In each of the paths $f^k_\#(R_-)$ and $f^k_\#(\beta_j)$ for $k \ge 0$ and $j \ge 1$, there is a uniform upper bound for the number of height $r$ edges in any subpath that occurs as a subpath of a leaf of $\Lambda^+$ --- indeed each of these paths is in $\<Z,\hat\rho_r\>$ by Lemma~\ref{defining Z}~\pref{ItemZPPathsInv}, and so any occurence of a height $r$ edge $E$ in any of these paths is a subpath of an occurrence of $\rho_r$, and therefore the height $r$ illegal turn occurs within a bounded distance of the occurrence of $E$. On the other hand, for each $j$ the sequence $f^k_\#(\alpha_j)$ has height $<u$ and weakly limits on~$\Lambda^+$, by Lemma~\ref{defining Z}~\pref{ItemZPFinitePaths}, and so as $k \to +\infinity$ the path $f^k_\#(\alpha_j)$ contains segments of~$\Lambda^+$ with arbitrarily many height $r$ edges.

We claim that when $f^k_\#(\overline R_-) f^k_\#(\alpha_1) f^k_\#(\beta_2) f^k_\#(\alpha_2) \ldots$ is tightened to $f^k_\#(\gamma)$, if $j < j'$ then no subsegments of $f^k_\#(\alpha_j)$ and $f^k_\#(\alpha'_j)$ cancel with each other. It follows from this claim that the long leaf segments of $\Lambda^+$ that occur in the $f^k_\#(\alpha_j)$ subpaths are not cancelled when $f^k_\#(\overline R_-) f^k_\#(\alpha_1) f^k_\#(\beta_2) f^k_\#(\alpha_2) \ldots$ is tightened to $f^k_\#(\gamma)$, and so $\gamma$ is weakly attracted to $\Lambda^+$, completing the proof in the first subsubcase. 

To prove the claim, suppose on the contrary that $f^k_\#(\alpha_j)$ and $f^k_\#(\alpha_{j'})$ have cancelling subsegments, with $j<j'$. Letting $\delta = \beta_{j} \alpha_{j+1} \cdots \alpha_{j'-1} \beta_{j'-1}$, we can write
$\alpha_j = \alpha^1_j \alpha^2_j $ and 
$\alpha_{j'} = \alpha^1_{j'} \alpha^2_{j'} $ 
so that
$f^k_\#(\alpha_j^2\delta \alpha_{j'}^1)$ is trivial. It follows that $f^k_\#(\delta)$ has height $<s$. Since the restriction of $f$ to the component $C$ of $G_{s-1}$ containing $f^k_\#(\delta)$ is a homotopy equivalence and since the endpoints of $f^k_\#(\delta)$ are fixed, there is a path $\zeta$ in $C$ with the same endpoints as $f^k_\#(\delta)$ and such that $f^k_\#(\zeta) = f^k_\#(\delta)$. But then $[\delta \zeta]$ is a non-trivial closed path whose image under $f^k_\#$ is trivial. This contradiction proves the claim.

In the second subsubcase, ${\cal E}_-$ is not carried by $\<Z,\hat \rho_r\>$ and so has height $u$. There is a bi-infinite alternating decomposition into $\alpha_j$'s and $\beta_j$'s as above and the same argument shows that $\gamma$ is weakly attracted to $\Lambda^+$. 
\end{proof}

\begin{proof}[Proof of Proposition~\ref{prop:WA2}] Recall the hypotheses: $\phi \in \Out(F_n)$ is rotationless, $\Lambda^\pm \in \L^\pm(\phi)$, and $\A_\na \Lambda^\pm$ is trivial. Passing to a power of $\phi$ we may also assume that $\phi^\inv$ is rotationless. Corollaries~\ref{NA is well defined} and~\ref{NA is independent of plusminus} together imply that every nontrivial conjugacy class is weakly attracted to $\Lambda^+$ by iteration of $\phi$ and to $\Lambda^-$ by iteration of $\phi^\inv$. Since no nontrivial conjugacy class is fixed by $\phi$ or $\phi^\inv$, which is the standing assumption of Section~\ref{SectionSingular}, all the results of that section apply. In particular, for each $\Phi \in P(\phi)$ we have $\Fix_+(\wh\Phi) = \Fix_N(\wh\Phi)$, and for each $\Psi \in P(\phi^\inv)$ we have $\Fix_+(\wh\Psi) = \Fix_N(\wh\Psi)$. Also, no \ct\ representing a power of $\phi$ or $\phi^\inv$ has a geometric \eg\ stratum.

Choose a \ct\ $f \from G \to G$ representing $\phi$, and let $Z$ and $\hat\rho_r$ be as in Definition~\ref{defn:Z} with respect to $\Lambda^+$. Consider the lowest stratum $H_1$ of $G$.


\begin{lemma} \label{trivial Ana} 
With notation as above we have:
\begin{enumerate}
\item \label{ItemEGLowest}
$H_1$ is \eg\ and corresponds to $\Lambda^+$.
\item \label{ItemEGLamInAll}
$\Lambda^+$ is contained in each element of $\L(\phi)$.
\item \label{ItemEndLimitLambda}
For each $\Phi \in \Aut(F_n)$ representing $\phi$ the accumulation set of each $P \in \Per_+(\wh\Phi)$ contains $\Lambda^+$.
 \end{enumerate}
 \end{lemma}

\begin{proof} To prove \pref{ItemEGLowest}, the stratum $H_1$ is a core graph by (Filtration), so it is noncontractible, and it is not a zero stratum, nor a fixed stratum because that would lead to a fixed conjugacy class. Neither can $H_1$ be a nonfixed \neg\ stratum because $H_0 = \emptyset$. It follows that $H_1$ is an \eg\ stratum. If $\Lambda^+$ corresponded to some \eg\ stratum higher than $H_1$ then circuits in $H_1$ would not be weakly attracted to $\Lambda^+$, a contradiction. 

To prove~\pref{ItemEGLamInAll}, by Lemma~\ref{ZP is closed}~\pref{item:ZP=NA}, $\<Z,\hat\rho_r\>$ contains no lines. For each edge $E$ of each \eg\ stratum $H_r \subset G$, it follows that $E \not\subset Z$, for otherwise by Lemma~\ref{defining Z}~\pref{ItemZPPathsInv} any leaf of the lamination associated to $H_r$ would be a line in $\<Z,\hat\rho_r\>$. Applying the definition of $Z$ and Proposition~\ref{PropInclusion} it follows that $\Lambda^+$ is a subset of each element of~$\L(\phi)$.

To prove \pref{ItemEndLimitLambda}, let $\ti f \from \Gamma \to \Gamma$ be the lift of $f$ corresponding to $\Phi$, and consider $P \in \Per_+(\hat f) = \Per_+(\wh\Phi)$.

Suppose first that $\ti f$ is a principal lift, so $P \in \Fix_+(\hat f)$. Apply Fact~\ref{FactPrincipalRay} to get a principal direction $E$ with corresponding singular ray $R$, and a lift $\wt R \subset \wt G$ of $R$ that terminates at $P$. If $E$ belongs to an \eg\ stratum $H_s$ then, by \recognition\ Lemma~3.26~(2), the accumulation set of $P$ equals the accumulation set of $E$ which equals the attracting lamination $\Lambda^+_s$ corresponding to $H_s$, and so by item~\pref{ItemEGLamInAll} the accumulation set contains $\Lambda^+$. If $E$ is \noneg\ then, by Fact~\ref{FactNEGEdgeImage}, $f(E) = E \cdot u$ where $u$ is a closed path that forms a circuit and where the turn $(u, \bar u)$ is legal. Corollary~\ref{NA is well defined} implies that $\Lambda^+$ is in the accumulation set of the sequence $f^k_\#(u)$. Item~\pref{ItemEndLimitLambda} in the \noneg\ case therefore follows from $R = E \cdot u \cdot f_\#(u) \cdot f^2_\#(u) \cdot \ldots$. 

Suppose now that $\ti f$ is not a principal lift, so $\Per_+(\hat f)$ has two points or one point. If $\Per_+(\hat f)$ has two points then they are endpoints of a lift of a generic leaf of some lamination in $\L(\phi)$, and so by item~\pref{ItemEGLamInAll} the accumulation set of each of these two points contains $\Lambda^+$. 

We may therefore assume that $\Per_+(\hat f) = \Fix_+(\hat f) = \{P\}$. It follows that $\ti f$ is fixed point free, because if it had a fixed point then that point would have at least two directions fixed by $\ti f^q$ for some $q \ge 1$, which leads to at least two points in $\Fix_+(\hat f^q)$; this follows from \recognition\ Lemma~3.26~(1) applied to $f^q$, where we rule out the possibility $\Fix(\hat f^q) = \{T_c^\pm\}$ because no conjugacy class is $\phi$-periodic.

We next extract a fact from the proof of \BookOne\ Proposition~5.4.3:

\begin{fact}\label{FactEigenray} Under the notation above, there exists a periodic vertex $v \in G$, a path $u$ from $v$ to $f(v)$ with periodic initial and terminal directions, and a lift $\ti v$ of $v$, such that the concatenation $R = u f_\#(u) f^2_\#(u) f^3_\#(u) \ldots$ is a ray, and such that the path $\ti u = [\ti v, \ti f(\ti v)]$ is a lift of $u$, and so $\wt R = \ti u \ti f_\#(\ti u) \ti f^2_\#(\ti u) \ti f^3_\#(\ti u) \ldots$ is a lift of $R$.
\end{fact}
\begin{proof} The role of the current map $\ti f \from \wt G \to \wt G$ is played by the map $h \from \Gamma \to \Gamma$ in the proof of \BookOne\ Proposition~5.4.3. Starting from the first paragraph on page 582 where it says ``We assume now that $h$ is fixed-point free'', and continuing through the second paragraph on page 583, what is proved is the existence of a vertex $v \in G$ and a lift $\ti v \in \wt G$ such that the sequence of points $\ti v, \ti f(\ti v), \ti f^2(\ti v), \ldots$ is an ordered subset of an infinite path in $\wt G$. After replacing $v$ and $\ti v$ by $f^i(v)$ and $\ti f^i(\ti v)$ for some $i$, the path $\ti u = [\ti v, \ti f(\ti v)]$ and its projection $u$ clearly satisfy the stated conclusions.
\end{proof}

To apply Fact~\ref{FactEigenray}, let $\xi$ be the ideal endpoint of $\wt R$. Since $\ti f_\#(\wt R) \subset \wt R$ it follows that $\xi \in \Fix(\hat f)$. By \recognition\ Lemma~3.15~(2) we have $\xi \in \Fix_N(\hat f)$, and so $\xi=P$. Letting $k$ be a common period of $v$ and the initial and terminal directions of~$u$, consider the circuit $\sigma = u f_\#(u) \ldots f^{k-1}_\#(u)$. Note that $f^i_\#(\sigma)$ lifts to a subpath $\ti f^i_\#(\ti u) f^{i+1}_\#(u) \ldots f^{i+k-1}_\#(u)$ of $\wt R$ for each $i$. It follows that the weak accumulation set of $P$ contains every line to which $\sigma$ is weakly attracted, which by Corollary~\ref{NA is well defined} includes every line in $\Lambda^+$. This completes the proof of Lemma~\ref{trivial Ana}.
\end{proof}

To prove Proposition \ref{prop:WA2}, we first prove that items (1) and (2) are mutually exclusive by assuming that $\gamma \in \LS(\phi^\inv) = S_{\phi^\inv} \union \bigcup_{\Lambda \in \L(\phi^\inv)} \text{(generic leaves of $\Lambda$)}$, and proving that $\gamma$ is not weakly attracted to~$\Lambda^+$. If $\gamma \in S_{\phi^\inv}$, choose $\Psi \in P(\phi^\inv)$ and a lift $\ti \gamma$ whose endpoints belong to $\Fix_+(\wh\Psi)$. Choosing an element of $F_n$ whose axis has endpoints arbitrarily close to the endpoints of $\ti\gamma$, the projection of that axis gives a closed path $\alpha$ that is weakly attracted to $\gamma$ under iteration by~$\phi^{-1}$. Like any closed path, $\alpha$ is not a generic leaf of $\Lambda^+$, by \BookOne\ Lemma 3.1.16. Since $\Lambda^+$ has arbitrarily small neighborhoods that are mapped into themselves by~$\phi$, it can not be that $\alpha$ is weakly attracted to $\Lambda^+$ under iteration by $\phi^{-1}$. It follows that $\gamma$ is not a leaf of $\Lambda^+$ and hence, since it is fixed by $\phi$, is not weakly attracted to $\Lambda^+$ under iteration by $\phi$. This proof also applies to any $\phi$-periodic generic leaf $\gamma$ of any $\Lambda_j^- \in \L(\phi^{-1})$, after passing to a positive power of~$\phi$. Since any two generic leaves of $\Lambda_j^-$ have the same weak closure, under iteration they also have the same weak accumulation set, and so no generic leaf of $\Lambda_j^-$ is weakly attracted to~$\Lambda^+$.

\bigskip

Now we prove that each $\gamma \in \B$ satisfies (1) or (2). Our proof uses the theorem of Levitt and Lustig \cite{LevittLustig:PeriodicDynamics}, specialized to the present situation where $\phi \in \Out(F_n)$ has no periodic conjugacy classes, and so by applying Fact~\ref{LemmaFixPhiFacts} to any of its representatives $\Phi \in \Aut(F_n)$ we have $\Per(\Phi) = \Per_-(\Phi) \union \Per_+(\Phi)$.

\begin{theorem}\label{LevittLustig} \cite[Theorem I]{{LevittLustig:PeriodicDynamics}} 
If $\phi \in \Out(F_n)$ has no periodic conjugacy classes in $F_n$ then there exists an integer $q \ge 1$ such that for each $\Phi \in \Aut(F_n)$ representing $\phi$ and each $\xi \in \bdy F_n$, one of the following holds:
\begin{enumerate}
\item \label{ItemRepeller}
$\xi \in \Fix_-(\wh\Phi^q)$.
\item \label{ItemAttracted}
The sequence $\wh\Phi^{qi}(\xi)$ converges to a point in $\Fix_+(\wh\Phi^q)$.
\end{enumerate}
\end{theorem} 

\bigskip

Let $\gamma \in \B$ be realized in $G$ with height $s$. Consider the stratum $H_s$, which cannot be a zero stratum.

\subparagraph{Case A: $H_s$ is \neg.} The stratum $H_s$ is a single edge $E_{s}$ with principal initial vertex $v$. Choose a lift $\ti \gamma$ and an edge $\wt E_{s}$ in $\ti \gamma$ that projects to $E_{s}$. The lift $\ti f : \Gamma \to \Gamma$ that fixes the initial vertex $\ti v$ of $\wt E_{s}$ is principal by Fact~\ref{FactPrincipalLift}; let $\Phi \in P(\phi)$ correspond to $\ti f$. Let $P$ and $Q$ be the endpoints of $\ti \gamma$. The decomposition of $\ti \gamma$ into two rays that meet at $\ti v$ is a splitting for $\ti f$ by Lemma~4.1.4 of \BookOne. It follows that $\ti f^i_\#(\ti \gamma)$ contains $\wt E_s$ for all $i \ge 0$,
and so the sequences $\wh\Phi^i(P)$ and $\wh\Phi^i(Q)$ share no accumulation points in $\bdy F_n$. We apply Theorem~\ref{LevittLustig} to these two sequences, with one of the following outcomes. If either $P$ or $Q$ is not in $\Per_-(\wh\Phi)$ then by Theorem~\ref{LevittLustig} the sequence $\ti f^i_\#(\ti\gamma)$ weakly accumulates on a line $\ti\sigma$ that has an endpoint in $\Fix_+(\wh\Phi^q)=\Fix_+(\wh\Phi)$. Lemma~\ref{trivial Ana}~\pref{ItemEndLimitLambda} implies that the closure of $\ti \sigma$ contains $\Lambda^+$ and hence $\gamma$ is weakly attracted to $\Lambda^+$. Suppose instead that conclusion~\pref{ItemRepeller} of Theorem~\ref{LevittLustig} applies to both of $P,Q$. If $\Phi^\inv$ is principal then $\Per_-(\Phi) = \Fix_-(\Phi)$ and so $\gamma \in S_{\phi^\inv}$. If $\Phi^\inv$ is not principal then the set $\Per_-(\Phi)$, which contains both of $P$ and $Q$, cannot contain any other points, and moreover the line $\ti\gamma$ between those two points is lift of a generic leaf of a lamination in $\L(\phi^\inv)$.


\subparagraph{Case B: $H_s$ is \eg.} By Lemma~\ref{trivial Ana}~\pref{ItemEGLamInAll}, the lamination $\Lambda_{s}^+$ associated to $H_{s}$ contains $\Lambda^+$. If $\gamma$ is weakly attracted to $\Lambda^+_s$ then it is weakly attracted to $\Lambda^+$ and we are done. We are also done if $\gamma$ is a generic leaf of~$\Lambda^-_s$. Letting $f' \from G' \to G'$ be a \ct\ representing $\phi^\inv$ with \eg\ stratum $H'_t$ such that $[G_s] = [G'_t]$, and letting $\gamma'$ be the realization of $\gamma$ in $G'_t$, we can apply Lemma~\ref{nonGeometricFullHeightCase} because $H_s$ is nongeometric. We are therefore reduced to the case that $\gamma'$ contains a singular ray $R$ of~$\phi^\inv$. Choose a lift $\wt R$ of~$R$, a lift $\ti\gamma'$ of $\gamma'$ containing $\wt R$, and a principal lift $\ti f'$ of $f'$ such that $\ti f'(\wt R) = \wt R$. Let $\Psi \in P(\phi^\inv)$ correspond to $\ti f'$, and let $\Phi = \Psi^\inv$. Let $P \in \Fix_+(\wh\Psi)$ be the endpoint of $\wt R$. Let $Q$ be the endpoint of $\ti\gamma'$ opposite $P$. If $Q \in \Fix_+(\wh \Psi)$ then $\gamma \in S_{\phi^\inv}$ and we are done. If $Q \not\in \Fix_+(\wh\Psi) = \Fix_-(\wh\Phi)$ then by Theorem~\ref{LevittLustig} the sequence $\wh\Phi^{qi}(Q)$ converges to some $Q' \in \Fix_+(\wh\Phi^q)$ for some $q \ge 1$. By Lemma~\ref{trivial Ana}~\pref{ItemEndLimitLambda} the accumulation set of $Q'$ contains $\Lambda^+$. Since $P \in \Fix_-(\wh\Phi^q)$ and $Q' \in \Fix_+(\wh\Phi^q)$ are not equal it follows that $\ti f^{qi}_\#(\ti\gamma)$ accumulates on the line $\wt\sigma$ with endpoints $P,Q'$, which projects to a line $\sigma$ whose weak limit set contains $\Lambda^+$, and so $\gamma$ is weakly attracted to~$\Lambda^+$.
\end{proof}

The next corollary is used in Step~3 of Section~\ref{SectionLooking}.
 
\begin{corollary} \label{minimal set} For any rotationless $\phi \in \Out(F_n)$ and $\Lambda^\pm \in \L^\pm(\phi)$ such that $\A_\na \Lambda^\pm$ is trivial, the set $\LS(\phi)$ is closed and is equal to $S_\phi \union \bigcup_{\Lambda \in \L(\phi)} \Lambda$. Furthermore, the closure of every element of $\LS(\phi)$ contains $\Lambda^+$.
\end{corollary}
 
\begin{proof} The set of lines weakly attracted to~$\Lambda^-$ under iteration of $\phi^\inv$ is an open subset of $\B$ which, by Proposition~\ref{prop:WA2}, is the complement of $\LS(\phi)$, proving that $\LS(\phi)$ is closed. Clearly $\LS(\phi) \subset S_\phi \union \bigcup_{\Lambda \in \L(\phi)} \Lambda$, and the opposite inclusion follows since $\LS(\phi)$ is closed and each $\Lambda \in \L(\phi)$ is the closure of any of its generic leaves. By Lemma~\ref{trivial Ana}~\pref{ItemEGLamInAll} and~\pref{ItemEndLimitLambda}, the closure of each element of $\LS(\phi)$ contains $\Lambda^+$.
\end{proof}

The next corollary, a uniform version of Proposition~\ref{prop:WA2}, is used in Step~2 of Section~\ref{SectionLooking}.

\begin{corollary} \label{cor:WA2} Suppose that $\phi \in \Out(F_n)$ is rotationless, that $\Lambda^\pm \in \L^\pm(\phi)$ and that $\A_{na}(\Lambda^\pm)$ is trivial. Let $V_+$ be an attracting neighborhood of $\Lambda^+$. If $(\gamma_n)_{n \ge 1}$ is a sequence of lines and $i_n \to +\infinity$ is a sequence such that $\phi^{i_n}(\gamma_n) \not \in V_+$, then every weak limit of the sequence $(\gamma_n)$ is in $\LS(\phi^\inv)$. 
\end{corollary}

\begin{proof} Consider a weak limit $\gamma$ of the sequence $(\gamma_n)$, and suppose that $\gamma$ weakly attracted to $\Lambda^+$, so there exists $k \ge 0$ such that $\phi^k(\gamma) \in V^+$. If $n$ is sufficiently large it then follows that $\phi^k(\gamma_n) \in V^+$. In particular, choosing $n$ so that $i_n \ge k$, it follows that $\phi^{i_n}(\gamma_n) \in V_+$, in contradiction to the hypothesis. We conclude that $\gamma$ is not weakly attracted to $\Lambda^+$ and so the corollary follows from Proposition~\ref{prop:WA2}.

 \end{proof}

\section{Invariant vertex groups and invariant free factors} \label{transversality}
In certain situations we have an element of $\Out(F_n)$ with only an invariant vertex group system of the form $\A_\na\Lambda^\pm$, but we want an invariant free factor system, which is supplied by Proposition~\ref{PropVertToFree}. Corollary~\ref{ZP is not invariant} will be used in the proof of Proposition~\ref{PropUniversallyAttracting}.  

Recalling that $\Out(F_n)$ acts on conjugacy classes of subgroups of $F_n$, it also acts on subgroup systems. The stabilizer of a subgroup system is therefore a well-defined subgroup of $\Out(F_n)$.

\begin{proposition} 
\label{PropVertToFree}
If $\phi \in \Out(F_n)$ is rotationless and reducible, if $\Lambda^\pm \in \L^\pm(\phi)$, and if $\A_\na\Lambda^\pm$ is nontrivial, then there exists a proper, nontrivial free factor system $\B$ and a finite index subgroup $K \subgroup \Stab(\A_\na\Lambda^\pm)$ such that:
\begin{enumerate}
\item $K \subgroup \Stab(\B)$ 
\item \label{ItemNonGeomFF}
If $\Lambda^\pm$ is not geometric then $\B = \A_\na \Lambda^\pm$.
\item \label{ItemGeomFF}
If $\Lambda^\pm$ is geometric then, choosing a \ct\ $f \from G \to G$ representing $\phi$ with geometric \eg\ stratum $H_r$ corresponding to $\Lambda^+$, and 
adopting the notation of Definition~\ref{DefGeometricStratum}, we have:
\begin{enumerate}
\item \label{ItemOnlyOne}
If $m = 0$ then $\B = \A_\supp [\pi_1 S]$.
\item \label{ItemTwoOrMore}
If $m \ge 1$ then $\B = \A_\supp\{[\bdy_1 S],\ldots,[\bdy_m S]\}$
\end{enumerate}
\end{enumerate}
\end{proposition}

It immediately follows that:

\begin{corollary} \label{ZP is not invariant} If $H \subgroup \Out(F_n)$ is fully irreducible, if $H_0 \subgroup H$ has finite index, if $\phi \in H$ is rotationless and reducible, if $\Lambda^\pm \in \L^\pm(\phi)$, and if $\A_{\na} \Lambda^\pm$ is nontrivial, then there exists $\psi \in H_0$ such that $\psi(\A_\na \Lambda^\pm) \ne \A_\na \Lambda^\pm$. 
\qed\end{corollary} 

\begin{proof}[Proof of Proposition \ref{PropVertToFree}] When $\Lambda^\pm$ is nongeometric then by Proposition~\ref{PropVerySmallTree} the vertex group system $\A_\na\Lambda^\pm$ is a free factor system, and it is clearly proper, so the proposition follows with $\B = \A_\na\Lambda^\pm$ and $K = \Stab(\A_\na\Lambda^\pm)$.

Henceforth we assume that $\Lambda^\pm$ is geometric. Choose a \ct\ $f \from G \to G$ with \eg\ geometric stratum $H_r$ corresponding to $\Lambda^+$, and adopt the notation of Definition~\ref{DefGeometricStratum}. For the proof we use $[\bdy S]$ as a shorthand notation for the finite set of conjugacy classes $\{[\bdy_0 S],[\bdy_1 S],\ldots,[\bdy_m S]\}$.

We shall prove:
\begin{itemize}
\item The action of $\Stab(\A_\na\Lambda^\pm)$ on conjugacy classes of elements of $F_n$ preserves the set $[\bdy S]$, and the action on conjugacy classes of subgroups fixes $[\pi_1 S]$.
\end{itemize}
We use this to finish the proof as follows.

In case \pref{ItemOnlyOne}, $\pi_1 S$ is a nontrivial free factor of $F_n$, and so $\A_\supp[\pi_1 S]=[\pi_1 S]$. Moreover $\pi_1 S$ is a proper free factor, because the restriction of $\phi$ to $\Out(\pi_1 S)$ is fully irreducible, whereas $\phi$ is not fully irreducible as an element of $\Out(F_n)$. Since $\Stab(\A_\na\Lambda^\pm)$ preserves $[\pi_1 S]$, we are finished.

In case \pref{ItemTwoOrMore} clearly the free factor system $\A_\supp\{[\bdy_1 S],\ldots,[\bdy_m S]\}$ is nontrivial, and it is proper since each of the conjugacy classes $[\bdy_1 S],\ldots,[\bdy_m S]$ is supported in $[G_{r-1}]$. Since the group $\Stab(\A_\na\Lambda^\pm)$ preserves the set $\{[\bdy_0 S],[\bdy_1 S],\ldots,[\bdy_m S]\}$, the kernel of its action on this finite set is a finite index subgroup $K$ that preserves the set $\{[\bdy_1 S],\ldots,[\bdy_m S]\}$. It follows that $K$ preserves the free factor system $\A_\supp\{[\bdy_1 S],\ldots,[\bdy_m S]\}$ and we are done.

\bigskip

The remainder of the proof is mostly focussed on showing that $\Stab(\A_\na\Lambda^\pm)$ fixes $[\pi_1 S]$, after which it will quickly follow that the set $[\bdy S]$ is preserved as well.

Recall the geometric model for $\A_\na\Lambda^\pm$ given in Section~\ref{SectionAsubNA}, consisting a 2-complex $X$ obtained by gluing $Y$ and $G$ along $G_r$, and an extension of the deformation restriction $d \from Y \to G_r$ to a deformation retraction $d \from X \to G$, allowing us to identify $K = Z \union \bdy_0 S \subset X$, and to identify the immersion $K \to G$ with the restriction of $d$. We therefore have an equation of subgroup systems $\A_\na\Lambda^\pm = [d_* \pi_1(K)]$.

Fix $\psi \in \Stab(\A_\na\Lambda^\pm)$. Since $\A_\na \Lambda^\pm = [d_* \pi_1(K)]$ it follows that $\psi$ is represented by a homotopy equivalence of pairs $\Psi \from (X,K) \to (X,K)$. As the proof proceeds, we will apply various homotopies to $\Psi$, always homotoping through maps of pairs $(X,K) \mapsto (X,K)$.

Our goal being to prove that $\psi$ stabilizes $[S]$, we shall study the map $S \to X \xrightarrow{\Psi} X$. It will be easier to work with the surface $S'$ that is obtained from $S \union \A$ by making the same identification of $\bdy_i S$ with a component of $\bdy \A_i$ as is made under the quotient map $S \union \A \union G_{r-1} \mapsto Y$. Each annulus $\A_i$ embeds as a collar neighborhood of a component of $\bdy S'$ denoted $\bdy_i S'$, and there is one more component of $\bdy S'$, denoted $\bdy_0 S'$, that coincides with $\bdy_0 S$. Denote $\bdy_- S' = \union_{1 \le i \le m} \bdy_i S'$. Let $j \from S' \to X$ be the map induced by composing the quotient map to $Y$ with the inclusion $Y \subset X$. Note that $j(\bdy S') \subset K$. With appropriate choice of base points, the image of the inclusion $\pi_1(S) \inject F_n$ is equal to $j_*(\pi_1 S')$, and so the conjugacy classes of these subgroups are equal, an equation which we notate by writing $[\pi_1 S] = [\pi_1 S']$. Also, for $i=1,\ldots,m$ the restriction of $j$ to $\bdy_i S'$ is a parameterization of the circuit $\alpha_i$, so the subgroup system denoted either $[\alpha_i]$ or $[\bdy_i S]$ can also be denoted $[\bdy_i S']$. 

Let $L = X - \interior (j(S')) = (G \setminus H_r) \union \bdy_0 S$, a subcomplex of $X$ containing $K=Z\union \bdy_0 S$. Note also that $L$ contains every free edge of $X$, meaning an edge $E$ such that $E - \bdy E$ is an open subset of $X$; these are precisely the edges that are not contained in $j(S')$.

Notice that although the set 
$$j(\interior(S')) = (\text{the $j$-image of the manifold interior of $S'$})
$$
contains the set 
$$\interior(j(S')) = (\text{the interior of $j(S')$ relative to $X$})
$$
these sets need not be equal. In fact we have
$$j(\interior(S')) - \interior(j(S')) = j(\interior(S')) \intersect L = j(\interior(S')) \intersect (\union_{s > r} H_s)
$$
which is a finite subset of the vertex set of $H_r$ that we call the \emph{attaching set} of $j(\interior(S'))$, it being where the higher strata of $G$ are attached to $j(\interior(S'))$ in $X$.

We shall need a metric on $X$ whose restriction to $Y$ is slightly different from the metric previously described in Section~\ref{SectionGeometric}. Start with any geodesic metric on $G_{r-1}$. This induces a metric on each component of $\bdy_- S'$, which extends to a hyperbolic structure with totally geodesic boundary on $S'$. Now extend to a geodesic metric on $\union_{s>r} H_s$. The advantage of this metric is that it is locally $\CAT(-1)$ (unlike the previous metric on $Y$, in which the annuli $\A$ were Euclidean products), and so every nontrivial conjugacy class in $F_n$ is represented by a unique closed geodesic in $Y$ \cite{BridsonHaefliger}. 

 We shall repeatedly make use of the following observations. First, for any conjugacy class of $F_n$ that is represented by a loop in~$L$, the corresponding closed geodesic in $X$ is contained in $L$; the same is true for any subcomplex of $L$, such as $K$. Second, for any conjugacy class that is in $[S] - \union_{0 \le i \le m} [\bdy_i S]$ the corresponding closed geodesic is not contained in~$L$, because that conjugacy class is represented by a loop that lifts to $S'$ and is not homotopic into $\bdy S'$, and can therefore be straightened to a geodesic in the hyperbolic structure on $S'$ which is entirely contained in $\interior(S')$, and then mapped back via $j$ to a geodesic in $X$ whose intersection with $L$ is contained in the attaching set of $j(\interior(S'))$. These observations follow from the fact that each of $L$ and $j(\interior(S'))$ are locally convex with respect to the $\CAT(-1)$ metric on $X$.

Let $S'' \subset S$ be a surface obtained by removing from $S$ a collar neighborhood of $\bdy_0 S$, and so $S''$ is obtained from $S'$ by removing a collar neighborhood of $\bdy S'$.

There are now two stages in our strategy for proving that $\psi[S]=[S]$: 
\begin{description}
\item[(I)] We may choose $\Psi$ within its homotopy class so that $\Psi\composed j(S'') \subset j(S')$.
\item[(II)] We may homotope the map $\Psi \composed j \from S'' \to j(S')$ to a map $S'' \mapsto j(\interior(S'))$.
\end{description}
Clearly (II) implies that $\psi[S] \sqsubset [S]$, using that $S''$ is a deformation retract of $S$ which is in turn a deformation retract of $S'$, and that $j$ restricts to a $\pi_1$-injective embedding of $\interior(S')$. By a result of Peter Scott found in Lemma~6.0.6 of \BookOne, it follows that $\psi[S]=[S]$.

Note that (II) does not follow immediately from (I), because even though $j(\interior(S')) \intersect j(\bdy S') = \emptyset$, nonetheless we do not expect either of the maps $j \restrict \bdy S'$ or $\Psi \composed j \restrict S''$ to be injective, so even if (I) is known to be true there is still work to do to pull the map $\Psi \composed j \restrict S''$ off of $j(\bdy S')$.

\bigskip

First we prove (I). Let $\E$ be the set of midpoints of all free edges of $X$. By standard differential topology methods \cite{Hirsch:DiffTop} we may perturb the map $\Psi$ by a homotopy so that the map $\Psi \composed j \from S' \to X$ is transverse to the set $\E$, which implies that $\tau = (\Psi \composed j)^\inv(\E)$ is a properly embedded compact 1-manifold in the compact surface~$S'$.

We claim that if $\Psi \from (X,K) \to (X,K)$ is chosen in its homotopy class so $\Psi \composed j \from S' \to X$ is transverse to $\E$ and $\tau$ has the minimal number of components, then each component of $\tau$ is a peripheral arc, meaning a properly embedded arc in $S'$ that is isotopic rel endpoints into $\bdy S'$. 

Once this claim is proved, we can then homotope $\Psi$ to arrange that $\tau$ is contained in any collar neighborhood of $\bdy S'$, in particular in $S' - S''$. After this homotopy, $\Psi(j(S''))$ is a compact subset of $X$ disjoint from $\E$, and therefore disjoint from some regular neighborhood $N(\E)$. Then we use the fact that $j(S')$ is a deformation retract of $X-N(\E)$; in fact there is a map of triples $(X,K,X-N(\E)) \mapsto (X,K,j(S'))$ which restricts to the identity on $j(S')$ and which is homotopic to the identity. Composing $\Psi$ with this map, we obtain a further homotopy of $\Psi$ after which $\Psi(j(S'')) \subset j(S')$, thereby proving~(I).

We prove the claim by descent: assuming that $c$ is a component of~$\tau$ which is not a peripheral arc, we shall homotope $\Psi$ to eliminate $c$ and reduce the number of components of $\tau$. This component $c$ is either a circle or a nonperipheral arc, and we consider these two cases separately. Let $e \in \E$ be the point for which $\Psi(j(c))=e$. 

\subparagraph{Case 1:} Suppose that $c$ is a circle. Since $\Psi \composed j \from S' \to X$ is $\pi_1$-injective, $c$ must be homotopically trivial, and so it bounds a closed disc $D$ contained in the interior of $S'$. Since $c \intersect \bdy S' = \emptyset$, the disc $D$ embeds in $X$ and we identify it with its image in $X$. By perturbing $\Psi$ we may assume that $c$ is disjoint from the attaching points of $j(\interior(S'))$, and so any regular neighborhood of $c$ in the 2-complex $X$ is an annulus disjoint from the attaching points of $j(\interior(S'))$. By transversality we may choose such a regular neighborhood $N(c)$, and we may choose a regular neighborhood $N(e)$ of $e$, such that $N(e)$ is an arc and the restriction of $\Psi$ is a fibration of $N(c)$ over $N(e)$ with circle fibers, the fiber over $e$ being $c = \bdy D$. Consider the disc $D' = D \union N(c)$, so $\Psi \restrict \bdy D'$ is a constant, equal to one of the two endpoints of $N(e)$. Since the 2-complex $X$ is aspherical, the map $\Psi \restrict D'$ is homotopic rel $\bdy D'$ to a constant. Furthermore, by the homotopy extension lemma \cite{Spanier}, this homotopy extends to a homotopy of all of $\Psi$ which is stationary outside of the union of $D'$ with an arbitrarily small neighborhood of those attaching points of $j(\interior S'))$ that lie in $\interior(D)$. The homotoped map $\Psi'$ is still transverse to $\E$, and $(\Psi' \composed j)^\inv(\E)$ is the union of those components of $\tau=(\Psi \composed j)^\inv(\E)$ that are not contained in $D$, a collection of components that does not include $c$. This completes the proof by descent in Case~1.

Note that $(\Psi')^\inv(\E)$ may have more points in the interiors of free edges of $X$ than $\Psi^\inv(\E)$ has, these new points being located in a neighborhood of those attaching points of $j(\interior S'))$ that lie in $\interior(D)$, but this is inconsequential to our proof by descent. 


\subparagraph{Case 2:} Suppose that $c$ is nonperipheral arc. Construct a closed curve $\gamma$ in $S'$ as follows. Let $\bdy c = \{x_1,x_2\}$ and orient $c$ from $x_1$ to $x_2$. If $x_1,x_2$ are in the same component of $\bdy S'$, let $\beta$ be an oriented arc in $\bdy S'$ from $x_2$ to $x_1$, and let $\gamma = c * \beta$; there are two choices for $\beta$, and for at least one of those choices the curve $\gamma$ is not freely homotopic to a closed curve in $\bdy S'$, because $S'$ supports a pseudo-Anosov homeomorphism and so is not a 3-holed sphere. If $x_1,x_2$ are in different components of $\bdy S'$, let $\beta_i$ be a loop based at $x_i$ parameterizing a component of $\bdy S'$, and let $\gamma = \{c * \beta_2 * \bar c * \beta_1\}$, and again since $S'$ is not a 3-holed sphere it follows that $\gamma$ is not freely homotopic to a closed curve in $\bdy S'$. In either case, $\gamma$ is freely homotopic to a closed geodesic contained in $\interior(S')$. It follows that $j(\gamma)$ is homotopic to a closed geodesic not contained in $L$.

On the other hand, 
\begin{align*}\Psi(j(\gamma)) \quad = \quad &\Psi(j(c)) * \Psi(j(\beta)) \\
 \text{or} \quad &\Psi(j(c)) * \Psi(j(\beta_2)) * \overline{\Psi(j(c))} * \Psi(j(\beta_1))
\end{align*}
but $\Psi(j(c))$ is a constant in $K$ and $\Psi(j(\beta))$, $\Psi(j(\beta_i))$ are closed curves in~$K$, so $\Psi(j(\gamma))$ is a closed curve in $K$. Since $\Psi$ restricts to a homotopy equivalence of $K$, there is a circuit $\gamma'$ in $K$ such that $\Psi(\gamma')$ and $\Psi(j(\gamma))$ are freely homotopic in $K$. Since $\Psi$ is a homotopy equivalence of $X$ it follows that $\gamma'$ and $j(\gamma)$ are freely homotopic in $X$. But $\gamma'$ is a closed geodesic in $K \subset L$, and $j(\gamma)$ is homotopic to a closed geodesic not contained in $L$, contradicting that $X$ is locally $\CAT(-1)$. This contradiction shows that Case~2 cannot occur.

This completes the proof of (I).
%


\bigskip

Now we prove (II). First we homotope $\Psi$ so that $\Psi \composed j(\bdy S'') \subset K$; this can be done by a homotopy of $\Psi$ supported in a neighborhood of $j(\bdy S'')$, given that each component $j(\bdy_i S'')$ is homotopic to the closed curve $j(\bdy_i S')$ in $K$ and that $\Psi(K) \subset K$. Next we subdivide $X$ so that $S''$ is a subcomplex, and then we further subdivide the domain and range of $\Psi$ (using not necessarily the same subdivision) and homotope $\Psi$ so that $\Psi$ is a simplicial map; this can be done by applying the simplicial approximation theorem \cite{Spanier}. Pulling back via $j$ we obtain a subdivision of $S''$ such that $\Psi \composed j \restrict S''$ is a simplicial map. Let
$$R = (\Psi \composed j \restrict S'')^\inv(j(\bdy S')) = (\Psi \composed j \restrict S'')^\inv(K) = (\Psi \composed j\restrict S'')^\inv(L) 
$$
where the equality of these sets follows by combining (I) with the fact that $j(\bdy S') = j(S') \intersect K = j(S') \intersect L$. Note that $R$ is a subcomplex of $S''$ containing $\bdy S''$. 

The key fact about $R$ is that any closed curve $\gamma$ contained in $R$ is either homotopically trivial or peripheral in $S''$. For if not then $j(\gamma)$ is freely homotopic in $X$ to a closed geodesic contained in $j(\interior(S'))$. However, $\Psi(j(\gamma)) \subset K$, and as in the proof of Case~3 of (I) above, using that $\Psi \from (X,K) \to (X,K)$ is a homotopy equivalence, it follows that $j(\gamma)$ is freely homotopic in $X$ to a closed geodesic contained in $K \subset L$, a contradiction.

Let $N(R)$ be a regular neighborhood of $R$ in the surface $S''$, and so $N(R)$ is a compact subsurface of $S''$ containing $\bdy S''$. Notice that each component of $\bdy N(R)$ is homotopic to a closed curve in $R$ and so is either homotopically trivial or peripheral in~$S''$. Let $N'(R)$ be the union of $N(R)$ with the discs bounding the homotopically trivial components of $\bdy N(R)$, so each component of $N'(R)$ is either a disc or a peripheral annulus in $S''$. 

We exhibit a collar neighborhood $C_0$ of $\bdy S''$ whose interior contains the subsurface $N'(R)$. First construct a collar neighborhood $C_2$ of $\bdy S''$ which contains each peripheral annulus component $A$ of $N'(R)$: each $A$ is contained in a unique annulus $C_2(A)$ having one boundary circle in $\bdy S''$ and the other in $\bdy A$, any two of the annuli $C_2(A)$ are disjoint or nested, and so the union of the $C_2(A)$ is the desired collar neighborhood $C_2$ of $\bdy S''$. Next consider the disc components $D$ of $N'(R)$ that are not already contained in $C_2$, let $C_1$ be the union of these discs $D$ with $C_2$, and let $\alpha_D$ be a pairwise disjoint collection of arcs properly embedded in the subsurface $\overline{S''-C_1}$ such that $\alpha_D$ has one endpoint on $\bdy (C_1)$ and the other on $\bdy D$. Then take $C_0$ to be a regular neighborhood of the union of $C_1$ and the arcs $\alpha_D$.

Now precompose the map $\Psi \composed j \restrict S''$ by a deformation retraction $S'' \mapsto \overline{S'' - C_0}$ to get the desired map $S'' \mapsto j(\interior(S'))$, finishing the proof of~(II).

\bigskip

The only task remaining in the proof of Proposition~\ref{PropVertToFree} is to show that the set $[\bdy S]$ is preserved by $\psi$. For any closed geodesic $\gamma$ in $X$ representing a conjugacy class $[\gamma]$, we have $[\gamma] \in [S]$ if and only if $\gamma$ lifts via $j \from S' \to X$ to a closed geodesic in $S'$, whereas $[\gamma] \in [K]$ if and only if $\gamma$ is a closed curve in $K$. The only geodesics that satisfy both of these conditions are the ones which lift to $\bdy S'$. In other words, a conjugacy class is an element of both of the subgroup systems $[S]$ and $[K]$ if and only if it is an element of the set $[\bdy S]$. Since $\psi$ preserves both $[S]$ and $[K]$ it follows that $\psi$ preserves $[\bdy S]$.
\end{proof}

\vfill\break

\section{Ping pong on geodesic lines}
\label{SectionPingPong}

In this section we develop ping pong methods for constructing new exponentially growing outer automorphisms from old ones, and in particular we prove Proposition~\ref{PropUniversallyAttracting}. 

\subsection{Finding attracting laminations}
\label{SectionFindingAttrLams}

The main result of this section is Lemma~\ref{FindingEG} which gives a method for finding attracting laminations that is independent of relative train track technology. We also prove Lemma~\ref{independence of plus or minus} which gives an extension of the duality relation of lamination pairs.

\begin{definition} \label{def:DoubleSharp} Suppose that $f \from G_1 \to G_2$ is a homotopy equivalence and that $\beta$ is a path. Choose lifts $\ti f : \ti G_1 \to \ti G_2$ and $\ti \beta \subset \ti G_1$ to the universal covers. Define $\ti f_{\#\#}(\ti \beta) \subset \ti f_\#(\beta) \subset \ti G_2$ to be the intersection of all paths $\ti f_\#(\ti \gamma)$ as $\gamma$ ranges over all paths in $\ti G_1$ that contain $\ti \beta$ as a subpath and define $f_{\#\#}(\beta)\subset G_2 $ to be the projected image of $\ti f_{\#\#}(\ti \beta)$. Since $\ti f_{\#\#}(\ti \beta)$ depends equivariantly on the choices of $\ti f$ and $\ti \beta$, it easily follows that $f_{\#\#}(\beta)$ does not depend on the choice of lifts $\ti f$ and $\ti \beta$. The bounded cancellation lemma implies that there is a uniform bound to the length of the initial and terminal segments that are removed from $f_\#(\beta)$ to form $f_{\#\#}(\beta)$.
\end{definition}

\begin{lemma}\label{lem:DoubleSharp} Suppose that $f \from G_1 \to G_2$ and $g \from G_2 \to G_3$ are homotopy equivalences. 
\begin{enumerate}
\item \label{item:double sharp containment} If $\alpha$ is a subpath of $ \beta$ in $G_1$ then $f_{\#\#}(\alpha)$ is a subpath of $f_\#(\beta)$. 
\item \label{item: double sharp composition} If $\beta$ is a path in $G_1$ then $g_{\#\#}(f_{\#\#}(\beta))$ is a subpath of $(g \composed f)_{\#\#}(\beta)$.
\item \label{item:disjoint copies} If $\sigma$ is a path in $G_1$ that decomposes into (possibly trivial) subpaths as 
$$\sigma = \alpha_1 \, \beta_1 \, \alpha_2 \, \beta_2\, \ldots \, \beta_m \, \alpha_{m+1}
$$
then there is a decomposition 
$$f_\#(\sigma) = \gamma_1 \, f_{\#\#}(\beta_1) \, \gamma_2 \, f_{\#\#}(\beta_2) \, \ldots \, f_{\#\#}(\beta_m) \, \gamma_{m+1}
$$
for some (possibly trivial) subpaths $\gamma_i$.
\end{enumerate}
 \end{lemma}

\begin{proof} \pref{item:double sharp containment} is an immediate consequence of the definition of $f_{\#\#}(\alpha)$. If $\ti \beta$ is a subpath of $ \ti \gamma$ then $\ti f_{\#\#}(\beta)$ is a subpath of $\ti f_\#(\ti \gamma)$ and so $g_{\#\#}(\ti f_{\#\#}(\ti \beta))$ is a subpath of $\ti g_\#(\ti f_\#(\ti \gamma)) = (\ti g\composed \ti f)_\#(\ti \gamma) $. This proves \pref{item: double sharp composition}. 

\pref{item:double sharp containment} follows from the fact that if $\ti \alpha \subset \ti \beta$ and $\ti \gamma$ contains $\ti \beta$ then $\ti \gamma$ contains $\ti \alpha$. \pref{item: double sharp composition} follows from the fact that $g_{\#}\composed f_{\#} = (g \composed f)_{\#}$. 

For \pref{item:disjoint copies}, choose a lift $\ti \sigma = \ti \alpha_1 \ti \beta_1 \ti \alpha_2 \ti \beta_2\ldots \ti \beta_m \ti \alpha_{m+1}$. Given $i < j$, write $\ti \sigma = \ti \sigma_1 \ti \sigma_2$ where $\ti \sigma_1$ is the initial subpath of $\ti \sigma$ that terminates with $\ti \beta_i$. It is an immediate consequence of the definition of $f_{\#\#}$ that $\ti f_\#(\ti \sigma_1)$ contains $\ti f_{\#\#}(\ti \beta_i)$ and that $\ti f(\ti \sigma_1)$ is disjoint from the interior of $\ti f_{\#\#}(\ti \beta_j)$. Similarly $\ti f_\#(\ti \sigma_2)$ contains $\ti f_{\#\#}(\ti \beta_j)$ and $\ti f(\ti \sigma_2)$ is disjoint from the interior of $\ti f_{\#\#}(\ti \beta_i)$. It follows that $\ti f_{\#\#}(\ti \beta_i)$ and $\ti f_{\#\#}(\ti \beta_j)$ are subpaths of $\ti f_\#(\ti \sigma)$ with disjoint interiors and that former precedes the latter. Since $i < j$ are arbitrary this proves \pref{item:disjoint copies}.
\end{proof}
 
 \begin{lemma} \label{BufferedSplitting} Suppose that $f:G \to G$ is a relative train track map and that $\Lambda^+$ is an attracting lamination for $f$ corresponding to the \eg\ stratum $H_r$. Then there is a positive constant $C$ so that if $\tau \subset G$ is a subpath of a leaf of $\Lambda^+$ and if $\tau = \tau_1 \tau_2 \tau_3$ is a decomposition into subpaths where $\tau_1$ and $\tau_3$ contain at least $C$ edges of $H_r$ and where the initial and terminal edges of $\tau_2$ are contained in $H_r$ then 
$ f^k_\#(\tau_2)$ is a subpath of $f^k_{\#\#}(\tau) $ for all $k \ge 0$.
 \end{lemma}
 
 \begin{proof} Suppose that $\tau = \tau_1 \tau_2 \tau_3$ is as in the statement of the lemma for an as yet unspecified constant $C$. Choose lifts $\ti f :\ti G \to \ti G$ and $\ti \tau = \ti \tau_1 \ti \tau_2 \ti \tau_3$. Given a path $\ti \gamma \subset \ti G$ that contains $\ti \tau$ as a subpath, decompose $\ti \gamma$ into subpaths $\ti \gamma = \ti \gamma_1 \ti \tau_2 \ti \gamma_3$ where $\ti \tau_1$ is a terminal subpath of $\ti \gamma_1$ and $\ti \tau_3$ is an initial subpath of $\ti \gamma_3$. By Lemma~4.2.2 of \BookOne\  we may choose $C$, independent of both $\tau$ and $\gamma$, so that $\ti f^k_\#(\ti \gamma) = \ti f^k_\#(\ti \gamma_1)\ti f^k_\#(\ti \tau_2) \ti f^k_\#(\ti \gamma) $ is a decomposition into subpaths for all $k \ge 0$. The lemma then follows from the definition of $f^k_{\#\#}(\tau) $.
 \end{proof}
 
Recall that if $\beta$ is a finite path in a marked graph $G$ then $N(G,\beta)$ is the weak basis element of $\B$ consisting of all lines whose realization in $G$ contains $\beta$ as a subpath.

\begin{lemma} \label{FindingEG} 
Let $f \from G \to G$ be a homotopy equivalence representing $\phi \in \Out(F_n)$, and suppose that $\beta \subset G$ is a finite path such that $f_{\#\#}(\beta)$ contains three disjoint copies of $\beta$. Then there exists $\Lambda \in \L(\phi)$ such that $\Lambda$ is $\phi$-invariant, each generic leaf of $\Lambda$ contains $\beta$ as a subpath, and $N(G,\beta)$ is an attracting neighborhood for $\Lambda$.
\end{lemma}

\begin{proof} For inductive reasons we write $\beta_0 = \beta$. Choose a lift $\ti \beta_0$ of $\beta_0$ to the universal cover $\ti G$ and a lift $\ti f: \ti G \to \ti G$ such that 
$$
\ti f_{\#\#}(\ti \beta_0) = \ti \alpha_{0,1} \, \ti \beta_{0,L} \, \ti \alpha_{0,2} \, \ti \beta_{0} \, \ti \alpha_{0,3} \, \ti \beta_{0,R} \, \ti \alpha_{0,4}
$$
where $\ti \beta_{0,L}$ and $\ti \beta_{0,R}$ are translates of $\ti \beta_0$. (The notation is chosen to emphasize the fact that $\ti \beta_{0,L}$ and $\ti \beta_{0,R}$ are to the left and right of $\ti \beta_0$ in $\ti f_{\#\#}(\ti \beta_0) $.) Define 
$$
\ti \beta_1 = \ti \beta_{0,L} \, \ti \alpha_{0,2}  \, \ti \beta_{0} \, \ti \alpha_{0,3}  \, \ti \beta_{0,R} \subset \ti f_{\#\#}(\ti \beta_0). 
$$
By Lemma~\ref{lem:DoubleSharp}~\pref{item:disjoint copies} we may write $\ti f_{\#\#}(\ti\beta_1)$ as
$$
\ti f_{\#\#}(\ti \beta_1) = \ti \alpha_{1,1}  \, \ti \beta_{1,L} \, \ti \alpha_{1,2}  \, \ti \beta_{1 \, }\ti \alpha_{1,3}  \, \ti \beta_{1,R} \, \ti \alpha_{1,4}
$$
where $\ti \beta_{1,L} \subset \ti f_{\#\#}(\ti \beta_{0,L})$ and $\ti \beta_{1,R} \subset \ti f_{\#\#}(\ti \beta_{0,R})$ are translates of $\ti \beta_1 \subset \ti f_{\#\#}(\ti \beta_0)$. Assuming by induction that 
 $$
\ti f_{\#\#}(\ti \beta_i) = \ti \alpha_{i,1}  \,  \ti \beta_{i,L}  \, \ti \alpha_{i,2}   \, \ti \beta_{i}  \, \ti \alpha_{i,3}  \,  \ti \beta_{i,R}  \, \ti \alpha_{i,4}
$$ 
where $\ti \beta_{i,L} \subset \ti f_{\#\#}(\ti \beta_{i-1,L})$ and $\ti \beta_{i,R} \subset \ti f_{\#\#}(\ti \beta_{i-1,R})$ are translates of $\ti \beta_i \subset \ti f_{\#\#}(\ti \beta_{i-1})$, define 
$$
\ti \beta_{i+1} = \ti \beta_{i,L}\ti \alpha_{1,2} \ti \beta_{i}\ti \alpha_{i,3} \ti \beta_{i,R} \subset \ti f_{\#\#}(\ti \beta_i)
$$ 
and apply Lemma~\ref{lem:DoubleSharp}~\pref{item:disjoint copies} to complete the induction step. 

The union of the nested sequence $ \ti \beta_0 \subset \ti \beta_1\subset \ti \beta_2 \subset\cdots$ is an $\ti f_\#$-invariant line $\ti \lambda$. Each ray $\ti R$ in $\ti \lambda$ contains a translate of $\ti \beta_i$ for all sufficiently large $i$ and so contains a translate of $\ti \beta_i$ for all $i$. Thus $\ti \lambda$ is birecurrent. If a line $\ti \gamma$ contains $\ti \beta_0$ as a subpath then $\ti f_\#(\ti \gamma)$ contains $\ti f_{\#\#}(\ti \beta_0)$ by Lemma~\ref{lem:DoubleSharp}~\pref{item:double sharp containment} and so contains $\ti \beta_1$. The obvious induction argument shows that $\ti f^i_\#(\ti \gamma)$ contains $\ti \beta_i$ for all $i$. This proves that $N(G,\beta)$ is an attracting neighborhood for $\lambda$ in the space of lines with respect to the action of $\phi$. Since the length of $\ti f_\#(\ti \beta_i)$ is at least three times the length of $\ti \beta_i$, \ $\ti \lambda$ is not the axis of a covering translation. By Definition~3.1.5 of \BookOne\ the closure $\Lambda$ of $\lambda$ is an attracting lamination for $\phi$ and $\lambda$ is a generic leaf of $\Lambda$. Since $\ti \lambda$ is $\ti f_\#$-invariant and $N(G,\beta)$ is an attracting neighborhood for $\lambda$, $\Lambda$ is $\phi$-invariant and $N(G,\beta)$ is an attracting neighborhood for $\Lambda$. 
 \end{proof}
 
Recall that for any $\phi \in \Out(F_n)$, dual elements of $\L(\phi)$ and $\L(\phi^\inv)$ are supported by the same free factor system. The following lemma extends this fact in certain situations, replacing free factor systems by the vertex groups systems constructed in Section~\ref{SectionVertexGroups}.

\begin{lemma} \label{independence of plus or minus} Suppose that $\phi, \xi \in \Out(F_n)$ are rotationless, and that $\Lambda^\pm_\phi \in \L^\pm(\phi)$ and $\Lambda^\pm_\xi \in \L^\pm(\xi)$. Suppose further that the following holds: 
\begin{itemize}
\item[$(*)$] $\Lambda^\pm_\phi$ is nongeometric \, \emph{or} \, $\Lambda^\pm_\xi$ is geometric.
\end{itemize}
Then $\Lambda^+_\xi$ is carried by the subgroup system $\A_\na(\Lambda^\pm_\phi)$ if and only if $\Lambda^-_\xi$ is carried by $\A_\na(\Lambda^\pm_\phi)$.
\end{lemma}

The proof of this lemma in the case that $\Lambda^\pm_\phi$ and $\Lambda^\pm_\xi$ are both geometric makes use of the span construction from Lemma~7.0.6 of \BookOne\ applied to $\Lambda^\pm_\xi$. We do not know of any analogue of the span construction for nongeometric laminations, and we do not know if the lemma is true when $\Lambda^\pm_\phi$ is geometric and $\Lambda^\pm_\xi$ is nongeometric.

\begin{proof} If $\Lambda^\pm_\phi$ is non-geometric then $\A_\na(\Lambda^\pm_\phi)$ is a free factor system by Proposition~\ref{PropVerySmallTree}, and the lemma follows from the fact (Definition 3.2.3 and Lemma~3.2.4 of \BookOne) that $\A_\supp(\Lambda^+_\xi) = \A_\supp(\Lambda^-_\xi)$ is the smallest free factor system carrying $\Lambda^+_\xi$ and is the smallest carrying~$\Lambda^-_\xi$. 
 
We may therefore assume that both $\Lambda^\pm_\phi$ and $\Lambda^\pm_\xi$ are geometric. By symmetry it suffices to assume that $\Lambda^+_\xi$ is carried by $\A_\na(\Lambda^\pm_\phi)$ and prove that $\Lambda^-_\xi$ is carried by $\A_\na(\Lambda^\pm_\phi)$. Let $\fG$ be a \ct\ representing $\xi$ with geometric \eg\ stratum $H_r$ corresponding to $\Lambda^+_\xi$, and adopt the notation of Definition~\ref{DefGeometricStratum} for a geometric model $Y$ of $f \restrict G_r$. A circuit in $G_r$ whose geodesic representative in $Y$ is contained in the surface $S$ is said to be a \emph{$\xi$-geometric circuit}. Let $f' \from G' \to G'$ be a \ct\ representing $\phi$ with geometric \eg\ stratum $H'_s$ corresponding to $\Lambda^+_\phi$.

Let $\lambda^+_\xi$ be a generic leaf of $\Lambda^+_\xi$. A circuit $\alpha$ in $G$ is in the \emph{span} of $\lambda^+_\xi$ if for every integer $L > 0$, $\alpha$ is freely homotopic to a concatenation of the form $\mu_1 * .... * \mu_k$, where each $\mu_i$ is a subsegment of $\lambda^+_\xi$, and for each $i=1,...,k$ the subsegment of $\lambda^+_\xi$ of length $L$ centered on the terminal endpoint of $\mu_i$ and the subsegment of $\lambda^+_\xi$ of length $L$ centered on the initial endpoint of $\mu_{i+1}$ (with indices taken mod~$k$) represent the same path in $G$ up to possible reversal of orientation. Lemma 7.0.7 of \BookOne\ says that every $\xi$-geometric circuit $\alpha$ is in the span of $\lambda^+_\xi$. Knowing that $\lambda^+_\xi$ is carried by $\A_\na(\Lambda^\pm_\phi)$, Lemma 7.0.6 of \BookOne, applied to the restriction of $f'$ to the component of $G'_s$ containing $H'_s$, says that any circuit in the span of $\lambda^+_\xi$ is also carried by $\A_\na(\Lambda^\pm_\phi)$. Combining these it follows that every $\xi$-geometric circuit is carried by $\A_\na(\Lambda^\pm_\phi)$. Items \pref{item:ZP=NA} and \pref{item:closed} of Lemma~\ref{ZP is closed} imply that the set of lines carried by $\A_\na(\Lambda^\pm_\phi)$ is closed in the weak topology. Since $\Lambda^-_\xi$ is a weak limit of $\xi$-geometric circuits, it is carried by $\A_\na(\Lambda^\pm_\phi)$.
 \end{proof}
 
\subsection{The ping-pong argument}
\label{SectionPingPongArgument}
 
In this section we state and prove Proposition~\ref{PropSmallerComplexity}, a technical statement in which our ping-pong arguments are packaged. Proposition~\ref{PropSmallerComplexity} will be used in the inductive step of the proof of Proposition~\ref{PropUniversallyAttracting} given in Section~\ref{SectionProofUnivAttr}, and in the proof of Theorem~\ref{thm:main} given in Section~\ref{SectionLooking}.

When Proposition~\ref{PropSmallerComplexity} is applied, the outer automorphisms $\theta$ and $\omega$, and all hypotheses and conclusions concerning them, are used only if an element that we construct in the subgroup $\<\phi,\psi\>$ has a lamination with a geometric stratum. They could be removed entirely if we knew that the conclusions of Lemma~\ref{independence of plus or minus} held without assuming hypothesis $(*)$. 

\begin{proposition}
\label{PropSmallerComplexity} Suppose that $ \phi,\psi,\theta,\omega \in \Out(F_n)$ are rotationless and exponentially growing with lamination pairs $\Lambda^\pm_\phi$, $\Lambda^\pm_\psi$, $\Lambda^\pm_\theta$, $\Lambda^\pm_\omega$, respectively, and suppose that $W^-_\theta$ is an attracting neighborhood of $\Lambda^-_\theta$ and that $W^+_\omega$ is an attracting neighborhood of $\Lambda^+_\omega$ such that the following hold:
\begin{itemize}
\item $\A_\na\Lambda^\pm_\phi \subgroup \A_\na\Lambda^\pm_ \theta$ and $\Lambda^-_\phi \subset W^-_\theta$.
\item $\A_\na\Lambda^\pm_\psi \subgroup \A_\na\Lambda^\pm_ \omega$ and $\Lambda^+_\psi\subset W^+_\omega$.
\item Either every lamination pair of every element of $\< \phi, \psi\>$ is a geometric lamination pair, or both of the pairs $\Lambda^\pm_\omega$ and $\Lambda^\pm_\theta$ are non-geometric.
\end{itemize}
 Suppose further that the following properties are satisfied. 
\begin{itemize}
\item [(i)]$\Lambda^+_\psi$ is weakly attracted to $\Lambda^+_\phi$ under iteration by $\phi$. 
\item [(ii)]$\Lambda^-_\psi$ is weakly attracted to $\Lambda^-_\phi$ under iteration by $\phi^{-1}$. 
\item [(iii)] $\Lambda^+_\phi$ is weakly attracted to $\Lambda^+_\psi$ under iteration by $\psi$. 
\item [(iv)]$\Lambda^-_\phi$ is weakly attracted to $\Lambda^-_\psi$ under iteration by $\psi^{-1}$. 

\end{itemize}
 Then there exists an integer $M>0$, attracting neighborhoods $V^\pm_\phi$ of $\Lambda^\pm_\phi$, and attracting neighborhoods $V^\pm_\psi $ of $\Lambda^\pm_\psi$, such that for any integers $m,n \ge M$, the outer automorphism $\xi = \psi^m \phi^n$ is exponentially growing and has a lamination pair $\Lambda^\pm_\xi$ that satisfies the following:
\begin{description}
\item [$(1)$] Any conjugacy class carried by $ \A_\na \Lambda^\pm_\xi $ is carried by both $\A_\na \Lambda^\pm_\phi$ and $\A_\na \Lambda^\pm_\psi $. 
\item [$(2^+)$]$\psi^m(V^+_\phi) \subset V^+_\psi$.
\item [$(2^-)$] $\phi^{-n}(V^-_\psi) \subset V^-_\phi$.
\item [$(3^+)$]$\phi^n(V^+_\psi) \subset V^+_\phi$.
\item [$(3^-)$] $\psi^{-m}(V^-_\phi) \subset V^-_\psi$.
\item [$(4^+)$]$ V^+_\xi:= V^+_\psi $ is an attracting neighborhood of $\Lambda^+_{\xi}$ and $ V^+_\xi\subset W^+_\omega$.
\item [$(4^-)$] $ \label{item:4-} V^-_\xi := V^-_\phi$ is an attracting neighborhood for $\Lambda^-_\xi$ and $V^-_\xi \subset W^-_\theta$.
 \end{description}
\end{proposition}

\begin{proof} Choose \cts\
$$g_\phi \from G_\phi \to G_\phi \quad\text{and}\quad g_\psi \from G_\psi \to G_\psi
$$
representing $\phi$ and $\psi$, respectively. Let $H_\phi \subset G_\phi$, $H_\psi \subset G_\psi$ denote the \eg\ strata corresponding to $\Lambda_\phi$, $\Lambda_\psi$, respectively. Let $C_1$ be the constant of Lemma~\ref{BufferedSplitting} applied to $g_\phi \from G_\phi \to G_\phi$ and to $H_{\phi}$. Let $\lambda^+_{\phi}$ and $\lambda^+_{\psi}$ be realizations of generic leaves of $\Lambda^+_\phi$ and $\Lambda^+_\psi$ in $G_{\phi}$ and $ G_{\psi}$ respectively. Pick homotopy equivalences $h_\psi : G_\phi \to G_\psi$ and $h_\phi : G_\psi \to G_\phi$ that respect the markings. 

By \emph{(i)} we may choose a finite subpath $\alpha_0 \subset G_{\phi}$ of $\lambda^+_{\phi}$ that contains at least $2C_1+1$ edges of $H_\phi$ and such that ${h_\psi}_\#(\alpha_0) \subset G_\psi$ is weakly attracted to $\lambda^+_{\psi}$ under iteration by~$g_\psi$. Decompose $\alpha_0$ as an initial subpath with $C_1$ edges in $H_\phi$ followed by a central subpath $\alpha$ that begins and ends in $H_\phi$ followed by a terminal subpath with $C_1$ edges in $H_\phi$. 

By \emph{(iii)} there exists $m_1$ so that $(g_{\phi}^{m_1}h_{\phi})_\#(\lambda^+_{\psi})$ contains $\alpha_0$ as a subpath. Let $C_2$ be the bounded cancellation constant for $g_{\phi}^{m_1}h_\phi$. Then choose a subpath $\beta$ of $\lambda^+_{\psi}$ that begins and ends in $H_\psi$ and such that $(g_{\phi}^{m_1}{h_\phi})_\#(\beta)$ contains a subpath that decomposes as an initial subpath with at least $C_2$ edges followed by $\alpha_0$ followed by a terminal subpath with at least $C_2$ edges. By lengthening $\beta$ if necesssary, we may assume that 
 $V^+_\psi = N(G_\psi,\beta) $ is an attracting neighborhood for the action of $\psi$ on $\Lambda^+_\psi$ and that $V^+_\psi \subset W^+_\omega$. 
 
 If $\sigma$ is any path in $G_{\psi}$ that contains $\beta$ as a subpath then ${(g_{\phi}^{m_1}h_\phi})_\#(\sigma)$ contains $\alpha_0$ as a subpath. Lemma~\ref{BufferedSplitting} implies that
 ${(g_{\phi}^{m_1+k}h_\phi})_\#(\sigma) = {g_{\phi}^{k}}_\#((g_{\phi}^{m_1}h_\phi)_\#(\sigma))$ contains ${g_\phi^k}_\#(\alpha)$ for all $k \ge 0$. In other words, ${(g_{\phi}^{m_1+k}h_\phi})_{\#\#}(\beta) \supset {g_\phi^k}_\#(\alpha)$ for all $k \ge 0$. 
 
 By a symmetric argument with the roles of roles of $\phi$ and $\psi$ reversed there exists $m_2$ and a subpath $\gamma$ of $\lambda^+_{\phi}$ such that ${(g_{\psi}^{m_2+k}h_\psi})_{\#\#}(\gamma) \supset {g_\psi^k}_\#(\beta)$ for all $k \ge 0$ and such that $V^+_\phi = N(G_\phi,\gamma)$ is an attracting neighborhood for the action of $\phi$ on $\Lambda^+_\phi$.

Now choose $m_3$ so that ${g^k_\phi}_\#(\alpha)$ contains three disjoint copies of $\gamma$ and ${g_\psi^k}_\#(\beta)$ contains $\beta$ for all $k \ge m_3$. If $n \ge m_1 + m_3$ then ${(g_{\phi}^{n} \, h_\phi})_{\#\#}(\beta)$ contains three disjoint copies of $\gamma$. If in addition, $m \ge m_2 + m_3$ then $(g_{\psi}^{m} \, {h_\psi} \, g_{\phi}^{n} \, {h_\phi})_{\#\#}(\beta)$ contains three disjoint copies of $\beta$ by Lemma~\ref{lem:DoubleSharp}~\pref{item: double sharp composition}. Let $f= g_{\psi}^{m} \, {h_\psi} \, g_{\phi}^{n} \, {h_\phi}: G_{\psi} \to G_{\psi}$. Then $f$ represents $\xi = \psi^m \phi^n$ and by applying Lemma~\ref{FindingEG} it follows that $V^+_\xi := V^+_\psi= N(G_\psi,\beta)$ is an attracting neighborhood of an attracting lamination $\Lambda_\xi^+$ of $\xi$ each of whose generic leaves, when realized in $G_\psi$, contains $\beta$ as a subpath. 

Taking any integer $M \ge \max\{m_1+m_3,m_2+m_3\}$, for any $m,n \ge M$ we have now produced an exponentially growing $\xi = \psi^m \phi^n$ and an attracting lamination $\Lambda^+_\xi$ satisfying~$(2^+), (3^+)$ and $(4^+)$. 

By a completely symmetric argument, with the roles of $\phi,\psi$ played by $\psi^\inv$, $\phi^\inv$ respectively, we may increase $M$ if necessary and find attracting neighborhoods $V^-_\phi \subset W^-_\theta$ of $\Lambda^-_\phi$ and $V^-_\psi$ of $\Lambda^-_\psi$, so that if $m,n \ge M$ then $\xi^\inv = \phi^{-n} \psi^{-m}$ is exponentially growing and has an attracting lamination $\Lambda^-_\xi$ that satisfies ~$(2^-), (3^-)$ and $(4^-)$. 

 By Proposition~\ref{prop:WA1} there exists $m_4$ so that if $\nu$ is a line that is neither an element of $V^-_\phi = V^-_\xi$ nor carried by $\A_\na \Lambda^\pm_\phi$ then $\phi^k_\#(\nu) \in V^+_\phi$ for all $k \ge m_4$. In addition to the constraints previously placed on $M$, we now assume also that $m,n\ge M >m_4$. Then $\xi_\#(\nu) = \psi^m_\# \phi^n_\#(\nu) \in \psi^m_\#(V^+_\phi) \subset V^+_\xi$ which proves that $\nu$ is weakly attracted to~$\Lambda_\xi^+$. We will use this in the following form. Every line that is not contained in $V^-_\xi$ and is not weakly attracted to $\Lambda^+_\xi$ is carried by $\A_\na\Lambda^\pm_\phi$. 
 
 By a completely symmetric argument, we may assume, after increasing $M$ further if necessary, that every line that is not contained in $V^+_\xi$ and is not weakly attracted to $\Lambda^-_\xi$ is carried by $\A_\na\Lambda^\pm_\psi$. 
 
Any line that is weakly attracted to neither $\Lambda_\xi^+$ nor $\Lambda_\xi^-$ is disjoint from both $V^+_\xi$ and $V^-_\xi$ and so is carried by both $\A_\na\Lambda^\pm_\phi$ and $\A_\na\Lambda^\pm_\psi$. Restricting to periodic lines we conclude that a conjugacy class carried by both $\A_\na \Lambda^+_\xi $ and $\A_\na \Lambda^-_\xi $ is carried by both $\A_\na \Lambda^\pm_\phi$ and $\A_\na \Lambda^\pm_\psi$. If we knew that $\Lambda^-_\xi$ and $\Lambda^+_\xi$ were dual laminations then by Corollary~\ref{CorPMna} it would follow that $\A_\na\Lambda^-_\xi = \A_\na\Lambda^+_\xi$, so~(1) would follow and the proposition would be proved.
 
It therefore remains to show that $\Lambda_\xi^+$ and $\Lambda_\xi^-$ are dual laminations for $\xi$. 

Let the set of all lamination pairs of $\xi$ be indexed as $\{\Lambda^\pm_i\}_{i \in I}$. Assuming that $\Lambda^+_\xi$ and $\Lambda^-_\xi$ are not dual, we have 
$$\Lambda^+_\xi = \Lambda^+_i \quad\text{and}\quad \Lambda^-_\xi = \Lambda^-_j \quad\text{for some}\quad i \ne j \in I
$$
From this we derive a contradiction. There are two cases, depending on whether $\Lambda^+_i \subset \Lambda^+_j$.

Consider the case that $\Lambda^+_i \not\subset \Lambda^+_j$. Let $\sigma$ be a generic leaf of $\Lambda^+_j$, so $\sigma$ is not weakly attracted to $\Lambda^+_i=\Lambda^+_\xi$ by forward iteration of~$\xi$. The lamination $\Lambda_j^+$ is closed, $\xi$-invariant, and by Fact~\ref{FactDualDifferent} it does not contain $\Lambda_j^-$, so $\sigma$ is not weakly attracted to $\Lambda^-_j=\Lambda^-_\xi$ by backward iteration of~$\xi$. It follows that $\sigma$ is carried by $\A_\na\Lambda^\pm_\phi$ and hence is carried by $\A_\na\Lambda^\pm_\theta$. Lemma~\ref{independence of plus or minus} implies that a generic leaf $\tau$ of $\Lambda^-_j = \Lambda_\xi^-$ is carried by $\A_\na\Lambda^\pm_\theta$ which contradicts the fact that $\tau$ is contained in $V_\xi^- =V_\phi^- \subset W^-_\theta$ which is an attracting neighborhood of $\Lambda^-_\theta$. 

In the remaining case, $\Lambda^+_i \subset \Lambda^+_j$. By Lemma~\ref{containmentSymmetry} we have $\Lambda^-_i \subset \Lambda^-_j$ which implies that a generic leaf $\sigma$ of $\Lambda^-_i$ is not weakly attracted to $\Lambda^-_j$ under iteration of~$\xi^\inv$. Since $\Lambda_i^-$ is closed, $\xi$-invariant and does not contain $\Lambda_i^+$, $\sigma$ is not weakly attracted to $\Lambda^+_i$ under iteration of~$\xi$. It follows that $\sigma$ is carried by $\A_\na\Lambda^\pm_\psi$ and hence is carried by $\A_\na\Lambda^\pm_\omega$. Lemma~\ref{independence of plus or minus} implies that a generic leaf $\tau$ of $\Lambda^+_i = \Lambda_\xi^+$ is carried by $\A_\na\Lambda^\pm_\omega$ which contradicts the fact that $\tau$ is contained in $V_\xi^+=V_\psi^+ \subset W^+_\omega$ which is an attracting neighborhood of $\Lambda^+_\omega$. 
\end{proof}

\subsection{The proof of Proposition~\ref{PropUniversallyAttracting}}
\label{SectionProofUnivAttr}

The next two propositions are used in producing an appropriate conjugator, an outer automorphism that conjugates one element to second one so that the ping-pong arguments can be applied to the resulting pair.

\begin{proposition}\label{PropConjugator} For any subgroup $H \subgroup \Out(F_n)$, any element $\phi \in H$, and any dual lamination pair $\Lambda^\pm_\phi \in \L^\pm(\phi)$, there exists a finite index subgroup $H_0 \subgroup H$ such that for any $\theta \in H_0$ and any generic lines $\gamma^\pm$ for $\Lambda^\pm_\phi$, neither $\theta(\lambda^+)$ nor $\theta(\lambda^-)$ is carried by $\A_\na\Lambda^\pm_\phi$.
\end{proposition}

\begin{proof} Passing to a positive power we assume $\phi$ is rotationless. Applying Theorem~\ref{TheoremCTExistence}, let $\fG$ be a \ct\ representing $\phi$ with \eg\ stratum $H_r$ corresponding to $\Lambda^\pm_\phi$ chosen so that $[G_r] = \A_\supp(\Lambda^\pm_\phi)$. Assume the notation of Definition~\ref{defn:Z}. By Lemma~\ref{ZP is closed}~\pref{item:ZP=NA}, it suffices to show that the realizations of $\theta(\lambda^+)$ and $\theta(\lambda^-)$ in $G$ are not carried by~$\<Z,\hat\rho_r\>$. Lemma~7.0.3 of \BookOne\ is the special case of this under the additional hypothesis that the lamination pair $\Lambda^\pm_\phi$ is topmost, that being a requirement for defining the subgraph $Z$ in \BookOne. In our present setting, using our general definition of $Z$, the exact same proof works: the only property needed of $Z$ is that $Z \intersect G_r = G_{r-1}$, which holds here as it does in \BookOne.
\end{proof}

 We also need the following simple result about group actions on sets.

\begin{lemma} \label{action on sets} Suppose that a group $H$ acts on sets $X_1,\ldots,X_M$ and that $x_m \in X_m$ are points whose stabilizers in $H$ have infinite index. Then there is an infinite sequence $\{g_l\}$ in $H$ such that $g_l(x_m) \ne g_k(x_m)$ for $1 \le m \le M$ and for $l \ne k$. 
\end{lemma}

\proof The proof is by induction on $M$ with the $M=1$ case being an immediate consequence of the assumption that the stabilizer of $x_i$ in $H$ has infinite index. 

For the inductive step, assume that there is an infinite sequence $\{\hat g_l\}$ in $H$ such that $\hat g_l(x_m) \ne \hat g_k(x_m)$ for $1 \le m \le M-1$ and for $l \ne k$. If $\{\hat g_l(x_M)\}$ is an infinite set then, after passing to a subsequence, we may assume that $\hat g_l(x_M) \ne \hat g_k(x_M)$ for $l \ne k$. In this case we choose $g_l = \hat g_l$. 
 
We may therefore assume, after passing to a subsequence, that $\hat g_l(x_M) = \bar x_M$ is independent of $l$. Since the $H$-orbit of $\bar x_M$ equals the $H$-orbit of $x_M$, there is an infinite sequence $\{h_s\}$ in $H$ such that $h_s(\bar x_M) \ne h_t(\bar x_M)$ for $s \ne t$. 

We define by induction an increasing function $\alpha \from \N \to \N$ such that $h_{J} \hat g_{\alpha(J)}(x_m) \ne h_j \hat g_{\alpha(j)}(x_m)$ for $j<J$ and $1 \le m \le M-1$. Assume that $\alpha(1),\ldots,\alpha(J-1)$ are defined. For $1 \le m \le M-1$, the points $\hat g_{\alpha(J-1)+k}(x_m)$ take infinitely many values as $k \ge 1$ varies, and so the points $h_{J} \hat g_{\alpha(J-1)+k}(x_m)$ take infinitely many values. We may therefore pick $k \ge 1$ and set $\alpha(J) = \alpha(J-1)+k$ so that for $1 \le m \le M-1$, the point $h_J \hat g_{\alpha(J)}(x_m)$ is different from each of $h_1 \hat g_{\alpha(1)}(x_m),\ldots,h_{J-1} \hat g_{\alpha(J-1)}(x_m)$. This completes the definition of $\alpha$.

Setting $g_j = h_{j} \hat g_{\alpha(j)}$ completes the proof.
\endproof

\begin{proof}[Proof of Proposition \ref{PropUniversallyAttracting}] Consider a subgroup $H \subgroup \Out(F_n)$ that is fully irreducible. Applying the main result of \BookTwo, fix once and for all an exponentially growing $\theta \in H$ and a lamination pair $\Lambda^\pm_\theta$ for $\theta$. We may assume that either $\Lambda^\pm_\theta$ is non-geometric or that every lamination pair for every element of $H$ is geometric. Choose attracting neighborhoods $W_\theta^\pm$ of $\Lambda^\pm_\theta$.

Let $X$ be the set of lamination pairs $\Lambda^\pm_\phi$ for elements $\phi \in H$ with the following properties.
\begin{itemize}
\item [(a)] $\Lambda^-_\phi \subset W^-_\theta$ and $\A_\na\Lambda^\pm_\phi \subgroup \A_\na\Lambda^\pm_ \theta$.
\item [(b)] There exists $\mu \in H$ such that $\Lambda^+_\phi \subset \mu(W^+_\theta)$ and $\A_\na\Lambda^\pm_\phi \subgroup \mu( \A_\na\Lambda^\pm_ \theta)$. 
\end{itemize}
Note that $X$ contains $\Lambda^\pm_\theta$, using $\mu = \Id$, and so is non-empty. Note also that (a) is the same as the first bullet in the hypotheses of Proposition~\ref{PropSmallerComplexity}.

We claim that for any $\phi \in H$ with lamination pair $\Lambda^\pm_\phi \in X$ such that $\A_\na\Lambda^\pm_\phi$ is non-trivial, there exists $\xi \in H$ with a lamination pair $\Lambda^\pm_\xi \in X $ for $\xi$ that is properly contained in $\A_\na\Lambda^\pm_\phi$, that is, $\A_\na\Lambda^\pm_\xi \sqsubset \A_\na\Lambda^\pm_\phi$ and $\A_\na\Lambda^\pm_\xi \ne \A_\na\Lambda^\pm_\phi$. Proposition~\ref{PropVDCC} and the obvious induction argument imply that $H$ contains an element which has a lamination pair $\Lambda^\pm$ such that $\A_\na\Lambda^\pm = \emptyset$, thereby completing the proof of Proposition~\ref{PropUniversallyAttracting}.

To prove the claim, let $\gamma^\pm_\phi$ be generic lines of $\Lambda^\pm_\phi$, respectively. By Proposition~\ref{PropConjugator}, there is a finite index subgroup $H_0 \subgroup H$ such that for all $\zeta \in H_0$, neither $\zeta(\gamma_\phi^+)$ nor $\zeta(\gamma_\phi^-)$ is carried by $\A_\na\Lambda^\pm_\phi$. After passing to a further finite index subgroup, still called $H_0$, we may assume that the stabilizer of $\Lambda^ +_\phi$ either contains $H_0$ or intersects $H_0$ in a infinite index subgroup and the same is true of the stabilizer of $\Lambda^ -_\phi$. By Corollary~\ref{ZP is not invariant} and Lemma~\ref{action on sets}, we may choose $\zeta \in H_0$ so that $\zeta(\A_\na\Lambda^\pm_\phi) \ne \A_\na\Lambda^\pm_\phi$ and so that $\zeta(\Lambda^+_\phi) \ne \Lambda^-_\phi$ and $\zeta(\Lambda^-_\phi) \ne \Lambda^+_\phi$.
 
 Define $\psi = \zeta \phi \zeta^\inv$ and $W^+_\psi = \zeta(W^+_\theta)$. Then $\gamma^+_\psi = \zeta(\gamma^+_\phi)$ and $\gamma^-_\psi = \zeta(\gamma^-_\phi)$ are generic lines for a lamination pair $\Lambda^+_\psi =\zeta(\Lambda^+_\phi) \ne \Lambda^-_\phi$ and $\Lambda^-_\psi =\zeta(\Lambda^-_\phi) \ne \Lambda^+_\phi$ for $\psi$, $W^+_\psi$ is an attracting neighborhood for $\Lambda^+_\psi$ and $\A_\na\Lambda^\pm_\psi = \zeta(\A_\na\Lambda^\pm_\phi) \ne \A_\na\Lambda^\pm_\phi$. 
Since the maximum length of a strictly decreasing sequence of vertex group systems beginning with $\A_\na\Lambda^\pm_\phi$ is the same as the maximum length of a strictly decreasing sequence of vertex group systems beginning with $\A_\na\Lambda^\pm_\psi$, \ $\A_\na\Lambda^\pm_\phi \not \sqsubset \A_\na\Lambda^\pm_\psi$. It follows that any vertex group system that is contained in both $\A_\na\Lambda^\pm_\psi$ and $\A_\na\Lambda^\pm_\phi$ is properly contained in $\A_\na\Lambda^\pm_\phi$.

 If the weak closure of $\gamma^+_\psi$ contains a generic line of $\Lambda^-_\phi$ then $\Lambda^-_\phi \subset \Lambda^+_\psi$. Since $\A_\supp \Lambda^-_\phi$ and $\A_\supp \Lambda^+_\psi$ have the same rank, Lemma~3.1.15 of \BookOne\ implies that $ \Lambda^-_\phi = \Lambda^+_\psi$, which is a contradiction. We conclude that the weak closure of $\gamma^+_\psi$ does not contain a generic line of $\Lambda^-_\phi$. Since $\xi \in H_0$, $\gamma^+_\psi$ is not carried by $\A_\na\Lambda^\pm_\phi$. Proposition~\ref{prop:WA1} implies that $\gamma^+_\psi$ is weakly attracted to $\Lambda^+_\phi$ under iteration by $\phi$. This verifies ($i$) of Proposition~\ref{PropSmallerComplexity}. 
 The symmetric arguments show that items ($ii$), ($iii$) and ($iv$) of Proposition~\ref{PropSmallerComplexity} are also satisfied. 
 
 Choose $\mu$ as in (b), let $\omega= (\zeta\mu)\theta(\zeta\mu)^{-1}$ and let $W^+_\omega =(\zeta\mu)(W^+_\theta)$. Then $W^+_\omega$ is an attracting neighborhood for $\Lambda^+_\omega = (\zeta\mu)\Lambda^+_\theta$ that contains $ \zeta(\Lambda^+_\phi) = \Lambda^+_\psi $ and $\A_\na\Lambda^\pm_\psi = \zeta(\A_\na\Lambda^\pm_\phi) \subgroup \zeta \mu \A_\na\Lambda^\pm_\theta = \A_\na\Lambda^\pm_\omega $. Thus the second bullet in the hypotheses of Proposition~\ref{PropSmallerComplexity} is satisfied. Since $\omega= (\zeta\mu)\theta(\zeta\mu)^{-1}$ and $\Lambda^\pm_\omega = (\zeta\mu)\Lambda^\pm_\theta$, the lamination pair $\Lambda^\pm_\omega$ is geometric if and only if the lamination pair $\Lambda^\pm_\theta$ is geometric. This implies that the third bullet in the hypotheses of Proposition~\ref{PropSmallerComplexity} is satisfied.
 
 We have now verified all of the hypotheses of Proposition~\ref{PropSmallerComplexity}. From the conclusion of that proposition, there exists $\xi \in H$ satisfying \begin{itemize}
 \item [(1)] $ \A_\na \Lambda^\pm_\xi \subgroup \A_\na \Lambda^\pm_\phi $ and $\A_\na \Lambda^\pm_\xi \subgroup \A_\na \Lambda^\pm_\psi$
 \item [(4$^-$)] $\Lambda^-_\xi \subset W^-_\theta$.
 \item [(4$^+$)] $\Lambda^+_\xi \subset W^+_\omega= (\zeta\mu)(W^+_\theta)$. 
\end{itemize}
Combining (1) with (a) and (b) for $\phi$, we have
\begin{itemize}
 \item [(5)] $ \A_\na \Lambda^\pm_\xi \subgroup \A_\na \Lambda^\pm_\phi \subgroup \A_\na \Lambda^\pm_\theta$
\item [(6)] $ \A_\na \Lambda^\pm_\xi \subgroup \A_\na \Lambda^\pm_\psi =\zeta (\A_\na\Lambda^+_\phi) \subgroup \zeta \mu (\A_\na\Lambda^\pm_\theta)$
\end{itemize}
Items (4$^-$) and (5) show that $\Lambda^\pm_\xi $ satisfies (a). Items (4$^+$) and (6) show that $\Lambda^\pm_\xi $ satisfies (b). This proves that $\Lambda^\pm_\xi \in X$. Item (1) completes the proof of the claim as noted above, and so completes the proof of Proposition~\ref{PropUniversallyAttracting}.
\end{proof}

\section{Finding fully irreducible outer automorphisms}
\label{SectionLooking}

In this section we prove Theorem~\ref{thm:main}. The strategy is to consider a certain pair of elements $\phi,\psi \in \Out(F_n)$ with which to carry out a ping-pong argument. If ping-pong is successful then we produce a fully irreducible element of the form $\psi^m \phi^n$ for sufficiently large $m,n$. If ping-pong fails, instead we produce outer automorphisms of the form $\psi^m \phi^n$ for arbitrarily large values of $m,n$, each of which has a positive power having a nontrivial proper invariant free factor. In the latter situation, using our main ping-pong argument, Proposition~\ref{PropSmallerComplexity}, we prove that these proper free factors exhibit stronger and stronger filling properties, as $m,n \to +\infinity$. But then we will show that these filling properties are impossibly strong, and so there is indeed a fully irreducible element of the form $\psi^m \phi^n$. 

This strategy is carried out under the following assumptions on the pair $\phi,\psi \in \Out(F_n)$:
\begin{description}
\item[(I)] $\phi,\psi$ are rotationless and reducible.
\item[(II)] No proper, nontrivial free factor system is invariant under both $\phi$ and $\psi$.
\item[(III)] There exist lamination pairs $\Lambda^\pm_\phi \in \L^\pm(\phi)$, $\Lambda^\pm_\psi \in \L^\pm(\psi)$ such that $\A_\na(\Lambda^\pm_\phi)$ and $\A_\na(\Lambda^\pm_\psi))$ are both trivial.
\end{description}

\paragraph{Step 1:} Every fully irreducible subgroup $H \subgroup \Out(F_n)$ either contains a fully irreducible element or contains a pair $\phi,\psi \in H$ satisfying (I)--(III).

\bigskip

To prove this, first apply Theorem~1.1 of \BookTwo\ to obtain an element of $H$ with exponential growth. Proposition~\ref{PropUniversallyAttracting} now implies that there is a rotationless $\phi \in H$ and a lamination pair $\Lambda_\phi^{\pm} \in \L^\pm(\phi)$ such that $\A_\na(\Lambda_\phi^{\pm}) $ is trivial. The group $H$ acts on the set $X$ of proper nontrivial free factors (up to conjugacy), and the stabilizer of each point of this set has infinite index in $H$. The set $\{x_1,\ldots,x_M\}$ of conjugacy classes of $\phi$-invariant proper, nontrivial free factors is finite by Lemma~\ref{finite singular set}. Applying Lemma~\ref{action on sets} with $X=X_1=\ldots=X_M$, there exists $\eta \in H$ such that $\eta\{x_1,\ldots,x_M\} \cap \{x_1,\ldots,x_M\} = \emptyset$. Defining $\psi = \eta \phi \eta^{-1} \in H$, properties (II), (III) follow. If $\phi$ is fully irreducible then we are done. Otherwise, applying Fact~\ref{FactPeriodicIsFixed}, $\phi$ is reducible since it is rotationless, and so (I) follows, completing the proof of Step~1.
 
\paragraph{Step 2:} For any $\phi,\psi \in \Out(F_n)$ satisfying (I) --- (III), one of the following two properties holds:

\begin{enumerate}
\item [(a)] There exists an integer $M>0$ such that for any integers $m,n \ge M$ the outer automorphism $\xi = \psi^m \phi^n$ is fully irreducible.
\item [(b)] There exist filling sets $B_{\phi^{-1}} \subset S_{\phi^{-1}}$ and $B_{\psi} \subset S_\psi$ so that for any choice of weak neighborhoods $U(b) \subset \B$, one for each $b \in B_{\phi^\inv} \union B_\psi$, there exists a proper free factor $K$ that carries an element of each $U(b)$. 
 \end{enumerate}

To start the proof, if (a) fails then for all $M$ there exist integers $m,n \ge M$ such that some positive power of $\xi_M = \psi^m \phi^n$ has an invariant proper free factor $K_M$ (the dependence of $m$, $n$ on $M$ will be mostly suppressed through the proof, but where emphasis is needed we write $m=m(M)$, $n=n(M)$). Assuming this we will prove~(b). 



We first verify the hypotheses of 
Proposition~\ref{PropSmallerComplexity}. Since $\A_\na(\Lambda_\phi^{\pm})$ and $\A_\na(\Lambda_\psi^{\pm})$ are trivial, $\Lambda_\phi^{\pm}$ and $\Lambda_\psi^{\pm}$ are non-geometric. Let $\theta = \phi$ and $\omega = \psi$, let $\Lambda_\theta^{\pm} = \Lambda_\phi^{\pm}$ and $\Lambda_\omega^{\pm}= \Lambda_\psi^{\pm}$ and let $W_\theta^-$ be an attracting neighborhood of $\Lambda_\theta^-$ and $W_\omega^+$ an attracting neighborhood of $\Lambda_\omega^+$. The three bulleted items in the hypotheses of 
Proposition~\ref{PropSmallerComplexity} are satisfied. Since $\A_\na(\Lambda_\psi^{\pm})$ is trivial and $\psi$ is reducible, $\A_\supp(\Lambda_\psi^{\pm})$ is a proper free factor and so is not $\phi$-invariant. It follows that $\Lambda^+_\psi$ is not $\phi$ -invariant and hence that a generic leaf of $\Lambda^+_\psi$ is neither a generic leaf of an element of $\L(\phi^{-1})$ nor a singular leaf of $\phi^{-1}$. Proposition~\ref{prop:WA2} therefore implies that $\Lambda^+_\psi$ is weakly attracted to $\Lambda^+_\phi$ under iteration by $\phi$ which is item ($i$) in the hypotheses of 
Proposition~\ref{PropSmallerComplexity}. Items ($ii$), ($iii$) and ($iv$) are proved similarly. 

Having verified the hypotheses of Proposition~\ref{PropSmallerComplexity}, we now adopt notation from the conclusions of Proposition~\ref{PropSmallerComplexity}: $V^+_\phi$ and $V^+_\psi$ are attracting neighborhoods of $\Lambda^+_\phi$ and $\Lambda^+_\psi$ respectively; for all sufficiently large $M$, $\Lambda_{\xi_M}^\pm$ is an attracting lamination pair for $\xi_M$ such that $\A_\na(\Lambda_{\xi_M}^{\pm})$ trivial; and $V^+_{\xi_M}$ is an attracting neighborhood of $\Lambda^+_{\xi_M}$ such that
 \begin{equation} 
 \psi^{m(M)}(V^+_\phi) \subset V^+_\psi \subset V^+_{\xi_M}
 \end{equation} 
These two set inclusions are contained in conclusions $(2^+)$ and $(4^+)$ of Proposition~\ref{PropSmallerComplexity}.
 
No conjugacy class in $F_n$ is carried by every neighborhood of $\Lambda^+_{\xi_M}$. It follows that no free factor is carried by every neighborhood of $\Lambda^+_{\xi_M}$. Let $R$ be the rank of the ambient free group. The free factor $\xi^R_M K_M$, being invariant under some positive power of $\xi_M$, is therefore not carried by the attracting neighborhood $V^+_{\xi_M}$ of $\Lambda^+_{\xi_M}$. We may therefore choose a line $\sigma_M$ that is carried by $K_M$ such that $\xi^R_M(\sigma_M) \not\in V^+_{\xi_M}$. It follows that 
\begin{equation} 
 \xi^{p+1}_M(\sigma_M) \not\in V^+_{\xi_M} \mbox{ for all } 0 \le p \le R-1 
 \end{equation}
 which along with (1) implies that
 \begin{equation} 
 \phi^{n(M)}(\xi^p_M\sigma_M) \not\in V^+_\phi \mbox{ for } 0 \le p \le R-1
 \end{equation}

Suppose first that $p=0$. By combining (3) with Corollary~\ref{cor:WA2} it follows that some subsequence of $(\sigma_M)$ weakly converges to some line $\ell$ which is an element of $S_{\phi^\inv}$ or a generic leaf of some lamination in $\L(\phi^\inv)$. But in the latter case, by Lemma~\ref{trivial Ana}~\pref{ItemEGLamInAll}
 the weak closure of $\ell$ contains each leaf of $\Lambda^-_\phi$, one of which is an element of $S_{\phi^\inv}$ by Lemma~\ref{FactSingularRayDense}. In either case, therefore, some element of $S_{\phi^\inv}$ is a weak limit of a subsequence of $(\sigma_M)$. After passing to a subsequence of $(\sigma_M)$ we may assume that this element is a weak limit of the entire sequence $(\sigma_M)$. If some other element of $S_{\phi^{-1}}$ is a weak limit of some subsequence of $(\sigma_M)$ then after passing to a further subsequence we may assume that both elements of $S_{\phi^{-1}}$ are weak limits of $(\sigma_M)$. Continuing on in this manner, using that $S_{\phi^\inv}$ is finite, we construct a nonempty subset $B_{0,\phi^\inv} \subset S_{\phi^\inv}$ such that each element of $B_{0,\phi^\inv}$ is a weak limit of $(\sigma_M)$, and no element of $S_{\phi^\inv} - B_{0,\phi^\inv}$ is a weak limit of any subsequence of $(\sigma_M)$. Furthermore, these properties of $B_{0,\phi^\inv}$ remain true if $(\sigma_M)$ is replaced by any subsequence. 

By the same argument applied inductively to $p=1,\ldots,R-1$, we may pass to a further subsequence of $(\sigma_M)$ and find nonempty subsets $B_{p,\phi^{-1}} \subset S_{\phi^{-1}}$ such that each element of $B_{p,\phi^{-1}}$ is a weak limit of the sequence $({\xi^p_M}\sigma_M)$ and such that no element of $S_{\phi^{-1}} - B_{p,\phi^\inv}$ is a weak limit of some subsequence of $({\xi^p_M}\sigma_M)$. 


By Lemma~\ref{LemmaSingularSingleton}, the free factor support of any nonempty subset of $S_{\phi^\inv}$ is a singleton. Applying this to $B_{p,\phi^\inv}$, we may write its free factor support as $\{[F_p]\}$. 

We claim that if the nontrivial free factor $F_p$ is proper then the rank of $F_{p+1}$ is strictly less than the rank of $F_{p}$. Putting off for a moment the proof of this claim, we use it to prove conclusion~(b). 

If $F_0$ were proper then, from the claim, we would have $R > \rank(F_0) > \rank(F_1) > \ldots > \rank(F_{R-1}) > 0$, which is impossible. Therefore, the free factor $F_0$ is not proper, and so $B_{0,\phi^{-1}}$ is filling. Define $B_{\phi^{-1}} = B_{0,\phi^{-1}}$. Note that for any $b \in B_{\phi^\inv}$ and any weak neighborhood $U(b)$, if $M$ is sufficiently large then $U(b)$ contains $\sigma_M$, which is carried by the proper free factor $K_M$. 

A completely symmetric argument, using $\xi_M^{-1}= \phi^{-n}\psi^{-m}$ instead of $\xi_M = \psi^m \phi^n$ and items $(2^-)$ and $(4^-)$ of Proposition~\ref{PropSmallerComplexity} instead of items $(2^+)$ and $(4^+)$, shows that after passing to a further subseqence of $(\sigma_M)$ there exist lines $\mu_M$ that are carried by $K_M$ and non-empty finite subsets $B_{p,\psi} \subset S_\psi$ such that for each $p=0,\ldots,R-1$ each element of $B_{p,\psi}$ is a weak limit of the sequence $({\xi^{-p}_M} \mu_M)$ and such that no element of $S_\psi - B_{p,\psi}$ is a weak limit of any subsequence of $({\xi^{-p}_M}\mu_M)$. Define $B_\psi = B_{0,\psi}$ and argue as above that $B_\psi$ is filling. Also note as above that for any $b \in B_\psi$ and any weak neighborhood $U(b)$, if $M$ is sufficiently large then $U(b)$ contains $\mu_M$. This proves that $B_{\phi^{-1}} , B_{\psi}$ satisfy conclusion~(b).

\bigskip

It remains to prove the claim: assuming that $F_p$ is a proper free factor, we prove that $\rank(F_p) > \rank(F_{p+1})$. For $p=0,\ldots,R-1$ consider the $\psi$-invariant free factor system which is the support of the set of all elements of $S_{\psi^\inv}$ that are carried by $[F_p]$. By Lemma~\ref{LemmaSingularSingleton} this free factor system is a singleton which we write as $\{[F'_p]\}$. By construction we have $[F_p] \sqsupset [F'_p]$. Since $[F_p]$ is $\phi$-invariant and $[F'_p]$ is $\psi$-invariant, and since $\phi$ and $\psi$ have no common invariant free factor system, we have $\rank(F_p) > \rank(F'_p)$.

To finish proving the claim it suffices to prove that $[F'_p] \sqsupset [F_{p+1}]$. For notational simplicity we restrict our attention to $p = 0$.  

We show that $[F_0]$ carries $\sigma_M$ for sufficiently large $M$. To prove this, by Fact~\ref{FactWeakLimitLines} it suffices to show that $[F_0]$ carries every line $\sigma$ that is a weak limit of a subsequence of $(\sigma_M)$. By Corollary~\ref{cor:WA2}, $\sigma$ is either an element of $ S_{\phi^{-1} }$ or a generic leaf of some attracting lamination $\Lambda^-_s$ of~$\phi^{-1}$. In the former case, $\sigma \in B_{\phi^{-1}}$ by construction and so is carried by $[F_0]$. Suppose then that $\sigma$ is a generic leaf of $\Lambda^-_s$. Every leaf of $\Lambda^-_s$ is in the weak closure of $\sigma$, and so is a weak limit of a subsequence of $(\sigma_M)$. By Fact~\ref{FactSingularRayDense}, some leaf $\ell$ of $\Lambda^-_s$ is a singular line, at least one of whose ends is dense in $\Lambda^-_s$. Since $\ell$ is a weak limit of a subsequence of $(\sigma_M)$, it follows that $\ell \in B_{\phi^\inv}$, and so $\ell$ is carried by $[F_0]$. Furthermore, $\sigma$ is in the weak closure of $\ell$, so $\sigma$ is carried by $[F_0]$. This completes the proof that $[F_0]$ carries $\sigma_M$ for sufficiently large~$M$. Note that $[F_0]$ can be characterized as the smallest free factor that carries $\sigma_M$ for sufficiently large~$M$ (and, similarly, $[F_p]$ is the smallest free factor that carries $\xi^p_M\sigma_M$ for sufficiently large $M$, for any $p=0,\ldots,R-1$).

We next show, by a similar proof, that $[F'_0]$ carries the line $\tau_M = \phi^{n(M)}\sigma_M$ for sufficiently large~$M$. Again, by Fact~\ref{FactWeakLimitLines}, it suffices to show that  $[F'_0]$ carries every line $\tau$ that is a weak limit of a subsequence of $(\tau_M)$. Note that $\tau_M$ is carried by $[F_0] = \phi^{n(M)}[F_0]$ for sufficiently large $M$, and so $\tau$ is carried by $F_0$. By (2) and the second inclusion in (1), $\psi^{m(M)} \tau_M = {\xi_M} \sigma_M$ is not contained in~$V^+_\psi$. By Corollary~\ref{cor:WA2}, $\tau$ is either an element of $S_{\psi^\inv}$ or a generic leaf of some attracting lamination $\Lambda^-_t$ of $\psi^\inv$ that is carried by $[F_0]$. In the former case, $[F'_0]$ carries $\tau$ by definition of $[F'_0]$. Suppose then that $\tau$ is a generic leaf of $\Lambda^-_t$,
and so $\A_\supp(\Lambda^-_t) = \A_\supp(\tau) \sqsubset [F_0]$. By Fact~\ref{FactSingularRayDense} some leaf $\ell'$ of $\Lambda^-_t$ is an element of $S_{\psi^\inv}$ such that at least one  end of $\ell'$ is dense in $\Lambda^-_t$. It follows that $[F_0]$ carries $\ell'$, and so $[F'_0]$ carries $\ell'$, again by definition of $[F'_0]$. Since $\tau$ is in the weak accumulation set of $\ell'$, it follows that $[F'_0]$ carries $\tau$. This completes the proof that $[F'_0]$ carries $\tau_M$ for sufficiently large $M$.

Since $\psi^{m(M)}[F'_0] = [F'_0]$ it follows that $[F'_0]$ carries ${\xi_M}(\sigma_M) = \psi^{m(M)}(\tau_M)$, for sufficiently large $M$. But $[F_1]$ is the smallest free factor carrying $\xi_M(\sigma_M)$ for all sufficiently large~$M$, which implies that $[F'_0] \sqsupset [F_1]$, completing the proof of the claim. 

This concludes the proof of Step 2.

 \paragraph{Step 3:} For any $\phi,\psi \in \Out(F_n)$ satisfying (I)---(III), property~(b) of Step~2 is impossible.
 
 \bigskip
 
To prove this, first we show that $\cl(S_{\phi^{-1}}) \cap \cl(S_{\psi})= \emptyset$. Recall the closed set of lines $\LS(\psi)$, the union of the set $S_\psi$ with the set of generic leaves of each lamination in $\L(\psi)$, and similarly recall $\LS(\phi^\inv)$. If $\cl(S_{\phi^{-1}}) \cap \cl(S_{\psi})$ is not empty then, applying Corollary~\ref{minimal set} to any leaf in $\cl(S_{\phi^{-1}}) \cap \cl(S_{\psi})$, it follows that $\LS(\psi) \intersect \LS(\phi^\inv)$ contains both $\Lambda^-_\phi$ and $\Lambda^+_\psi$. Applying Corollary~\ref{minimal set} to a generic leaf of $\Lambda^-_\phi$ regarded as a leaf of $\LS(\psi)$ it follows that $\Lambda^+_\psi \subset \Lambda^-_\phi$, and a symmetric argument yields the opposite inclusion. Both $\phi$ and $\psi$ therefore fix $\A_\supp(\Lambda^-_\phi) = \A_\supp(\Lambda^+_\psi)$ which, by~(I), (III), and Corollary~\ref{CorollaryCTFullIrr}, is a proper, nontrivial free factor system. This contradicts~(II).
 
\medskip
 
The proof of Step 3 is completed by the following result: 
 
\begin{proposition} \label{two filling sets} Suppose that $B_1, B_2 \subset \B$ are filling sets of lines and that $\cl(B_1) \cap \cl(B_2) = \emptyset$. Then there exist weak neighborhoods $U(b) \subset \B$, one for each $b \in B_1 \cup B_2$, such that for each proper free factor $F$ there exists $b_F \in B_1 \cup B_2$ such that $[F]$ does not carry a line in $U(b_F)$. 
 \end{proposition}
 
\begin{proof} Let $L(\cdot)$ denote the edge length of paths in the rose $R_n$. Let each $b \in B_1 \union B_2$ be realized by a line in $R_n$ with a chosen base point. Let $\wh U(b,C) \subset \wh\B(R_n)$ denote the weak neighborhood of $b$ consisting of paths in $R_n$ that have a subpath which is equal to the subpath of $b$ of edge length~$2C$ centered on the base point. Let $U(b,C) = \wh U(b,C) \intersect \B(R_n) \subset \B(R_n) \approx \B$.

Since $\cl(B_1) \cap \cl(B_2) = \emptyset$, there exists a constant $K > 0$ such that for each $b_1 \in B_1$, $b_2 \in B_2$, the lines $b_1,b_2$ do not have a common subpath of edge length $> K$ in~$R_n$. It follows that for each path $\gamma$ in $R_n$ with $L(\gamma) > K$ there exists $i=1,2$ such that $\gamma$ is not a subpath of the realization in $R_n$ of any line in $B_i$.

Given a proper free factor $F$, for each $i=1,2$, since $B_i$ is filling we may choose $b_i(F) \in B_i$ which is not carried by $F$. Since the set of lines carried by $F$ is closed, we may then choose a constant $C_i(F) > 0$ such that the neighborhood $U(b_i(F),C_i(F))$ does not contain any line that is carried by $F$. The same is therefore true of the neighborhood $U(b_i(F),C)$ for any constant $C \ge C_i(F)$.

Consider triples $(G,S,\rho)$ consisting of a marked graph $G$, a proper connected subgraph $S \subset G$ with no valence~1 vertices, and a map $\rho \from G \to R_n$ which is a homotopy inverse of the marking of $G$, such that $\rho$ takes vertices to vertices, $\rho$ is an immersion on each edge of $G$, and $\rho$ is an immersion on the subgraph $S$; it follows that the restriction of $\rho$ to any path in $S$ is a path in $R_n$. Such a triple is called a \emph{representative} of a proper free factor $F$ if $[F]=[S]$. We put a metric on each edge of $G$ by pulling back the metric on $R_n$ under the map $\rho$, and we extend the length notation $L(\cdot)$ to this setting, and so for each edge $E \subset G$ we have $L(E)=L(\rho(E))$. Note that a line $\ell \in \B$ is carried by $[F]$ if and only if the realization $\ell_G$ in $G$ is contained in $S$, in which case the restriction of $\rho$ to $\ell_G$ is an immersion whose image is the realization of $\ell$ in $R_n$.

We claim that every proper free factor $F$ has a representative $(G,S,\rho)$. To see why, it is evident that there exists a triple $(G,S,\rho)$ that satisfies all the required properties except that $\rho$ need not be an immersion on $S$. Factor $\rho  \from G \to R_n$ as a composition of folds $G = G_0 \xrightarrow{p_1} G_1 \xrightarrow{p_2} \dots \xrightarrow{p_N}G_N = R_n$. Each intermediate graph $G_j$ is marked by a homotopy inverse of $\rho_j = p_Np_{N-1}\cdots p_{j+1}:G_j \to R_n$. Let $P_j = p_j p_{j-1}\cdots p_1 : G \to G_j$. By giving precedence to folds involving two edges of $S$, we may assume that there exists $J \ge 1$ such that $S_J := P_J(S)$ is a proper subgraph of $G_J$, the map $P_J$ restricts to a homotopy equivalence from $S$ to $P_J(S)$, and the map $\rho_J \restrict S_J$ is an immersion. After replacing $(G,S,\rho)$ by $(G_J,S_J,\rho_J)$ we obtain a representative of $F$.

Consider a proper free factor $F$ and a representative $(G,S,\rho)$ of $F$. Consider also the line $b_i(F)$ and the constant $C_i(F)$, for each $i=1,2$. By applying the Bounded Cancellation Lemma to the map $\rho$, letting $C_i(G,S,\rho)$ equal the sum of $C_i(F)$ and the bounded cancellation constant for $\rho$, and using that any path in $S$ extends to a line in $S$, we may ensure that for any finite path $\sigma \subset S$, no subpath of $\rho(\sigma)$ is equal to the basepoint-centered subpath of $b_i(F)$ with edge length $2C_i(G,S,\rho)$. In other words, $\wh U(b_{i}(F),C_i(G,S,\rho))$ does not contain the path~$\rho(\sigma)$. Let $C(G,S,\rho)$ equal the maximum of $C_i(G,S,\rho)$, $i=1,2$.

Two triples $(G,S,\rho)$, $(G',S',\rho')$ are said to be \emph{equivalent} if there exists a homeomorphism $h : (G,S) \to (G',S') $ such that $\rho' h = \rho$. Clearly the constants $C_i(G,S,\rho)$ and their maximum $C(G,S,\rho)$ depend only on the equivalence class of $(G,S,\rho)$. There are only finitely many equivalence classes of triples $(G,S,\rho)$ for which $L(E) \le 2K$ for all edges $E$ of $G$. Define $C =4K+\max\{2 \, C(G,S,\rho)\}$ where the maximum is taken over these finitely many equivalence classes, and note that $C$ depends only on $B_1$ and~$B_2$. For each $b \in B_1 \union B_2$ take $U(b) = U(b,C)$.

Now fix a free factor $F$, and a representative $(G,S,\rho)$ of $F$. To prove the proposition we consider two cases.

In the case that $L(E) \le 2K$ for all edges $E$ of $G$, it follows that $C > C(G,S,\rho) > C_i(F)$ for $i=1,2$, and the conclusion of the proposition holds for both $b_F = b_{1}(F)$ and $b_F = b_{2}(F)$. 
 
In the remaining case, factor $\rho \from G \to R_n$ as a composition of folds $G = G_0 \xrightarrow{p_1} G_1 \xrightarrow{p_2} \dots \xrightarrow{p_N}G_N = R_n$, and as above denote $\rho_j = p_Np_{N-1}\cdots p_{j+1}:G_j \to R_n$ and $P_j = p_j p_{j-1}\cdots p_1 : G \to G_j$, and so $\rho = \rho_j P_j$. Let $I \in \{1,\ldots,N\}$ be the minimal value such that $L(E) \le 2K$ for all edges $E$ of $G_I$. Note that some edge $E_I \subset G_I$ has length $L(E_I) > K$, because if not then, using that the fold map $p_I \from G_{I-1} \to G_I$ takes each edge of $G_{I-1}$ to a path of at most two edges in $G_I$, each edge of $G_{I-1}$ would have length $\le 2K$, contradicting minimality of~$I$. Choose $i \in \{1,2\}$ so that $\rho_I(E_I)$ is not a subpath of the realization in $R_n$ of any element of~$B_i$. Let $b_F = b_i(F)$. To conclude the proof we assume that the realization $\sigma$ in $G$ of some line in $U(b_F) = U(b_F,C)$ is contained in $S$ and argue to a contradiction. We use the fact that $\rho \restrict \sigma$ is an immersion, which implies that $P_I \restrict \sigma$ and $\rho_I \restrict P_I(\sigma)$ are also immersions. It follows that there is a finite subpath $\sigma_0$ of $\sigma$ (with endpoints not necessarily at vertices) such that $\rho(\sigma_0)$ is the basepoint-centered subpath of $b_F$ in $R_n$ with edge length~$L(\rho(\sigma_0)) = 2C$. Thus $P_I(\sigma_0)$ does not contain $E_I$ as a subpath. It follows that there is a subpath $\sigma_0' \subset \sigma_0$ such that $P_I(\sigma_0')$ is contained in the subgraph $G_I \setminus E_I$ of $G_I$ and such that $\rho(\sigma_0')=\rho_I(P_I(\sigma_0'))$ is the basepoint centered subpath of $b_F$ in $R_n$ obtained from $\rho(\sigma_0)$ by removing initial and or terminal segments of length less than $2K$. Our choice of~$C$ guarantees that 
\begin{align*}
L(\rho(\sigma_0)) &< L(\rho(\sigma_0')) + 4K = L(\rho_I(P_I(\sigma_0')) + 4K \\
&< 2 \, C_i(G_I,G_I \setminus E_I, \rho_I) + 4K \\
&\le 2 \, C(G_I,G_I \setminus E_I,\rho_I) + 4K \\
& \le C
\end{align*}
which is the desired contradiction. 
 \end{proof}
 
 This concludes the proof of Step 3. 
 
 \paragraph{Proof of Theorem~\ref{thm:main}} Combine Steps~1, 2, and~3. \hfill\qed

\bibliographystyle{amsalpha} 
\bibliography{mosher} 

\newcommand{\etalchar}[1]{$^{#1}$}
\def\cprime{$'$} \def\cprime{$'$}
\providecommand{\bysame}{\leavevmode\hbox to3em{\hrulefill}\thinspace}
\providecommand{\MR}{\relax\ifhmode\unskip\space\fi MR }
\providecommand{\MRhref}[2]{%
  \href{http://www.ams.org/mathscinet-getitem?mr=#1}{#2}
}
\providecommand{\href}[2]{#2}
\begin{thebibliography}{FLP{\etalchar{+}}79}
 
\bibitem[BF]{BestvinaFeighn:OuterLimits}
M.~Bestvina and M.~Feighn, \emph{Outer limits}, preprint.

\bibitem[BF08]{BestvinaFeighn:HypComplex}
\bysame, \emph{A hyperbolic {$\text{Out}(F_n)$}-complex}, arXiv:0808.3730,
  2008.

\bibitem[BFH97]{BFH:laminations}
M.~Bestvina, M.~Feighn, and M.~Handel, \emph{Laminations, trees, and
  irreducible automorphisms of free groups}, Geom. Funct. Anal. \textbf{7}
  (1997), 215--244.

\bibitem[BFH00]{BFH:TitsOne}
\bysame, \emph{{The Tits alternative for ${\rm Out}(F\sb n)$. I. Dynamics of
  exponentially-growing automorphisms.}}, Ann. of Math. \textbf{151} (2000),
  no.~2, 517--623.

\bibitem[BFH05]{BFH:TitsTwo}
\bysame, \emph{The {T}its alternative for {${\rm Out}(F\sb n)$}. {II}. {A}
  {K}olchin type theorem}, Ann. of Math. \textbf{161} (2005), no.~1, 1--59.

\bibitem[BH92]{BestvinaHandel:tt}
M.~Bestvina and M.~Handel, \emph{Train tracks and automorphisms of free
  groups}, Ann. of Math. \textbf{135} (1992), 1--51.

\bibitem[BH99]{BridsonHaefliger}
M.~Bridson and A.~Haefliger, \emph{Metric spaces of non-positive curvature},
  Springer, 1999.

\bibitem[CB88]{CassonBleiler}
A.~Casson and S.~Bleiler, \emph{Automorphisms of surfaces after {N}ielsen and
  {T}hurston}, London Mathematical Society Student Texts, vol.~9, Cambridge
  University Press, 1988.

\bibitem[CL95]{CohenLustig:verysmall}
M.~Cohen and M.~Lustig, \emph{Very small group actions on {${\bf R}$}-trees and
  {D}ehn twist automorphisms}, Topology \textbf{34} (1995), no.~3, 575--617.

\bibitem[CM87]{CullerMorgan:Rtrees}
M.~Culler and J.~W. Morgan, \emph{Group actions on {${\bf R}$}-trees}, Proc.
  London Math. Soc. (3) \textbf{55} (1987), no.~3, 571--604.

\bibitem[Coo87]{Cooper:automorphisms}
D.~Cooper, \emph{Automorphisms of free groups have finitely generated fixed
  point sets}, J. Algebra \textbf{111} (1987), no.~2, 453--456.

\bibitem[FH06]{FeighnHandel:recognition}
M.~Feighn and H.~Handel, \emph{{The recognition theorem for
  $\text{Out}(F_n)$}}, Preprint, 2006.

\bibitem[FLP{\etalchar{+}}79]{FLP}
A.~Fathi, F.~Laudenbach, V.~Poenaru, et~al., \emph{Travaux de {Thurston} sur
  les surfaces}, Ast\'erisque, vol. 66--67, Soci\'et\'e {M}ath\'ematique de
  {F}rance, 1979.

\bibitem[GJLL98]{GJLL:index}
D.~Gaboriau, A.~Jaeger, G.~Levitt, and M.~Lustig, \emph{An index for counting
  fixed points of automorphisms of free groups}, Duke Math. J. \textbf{93}
  (1998), no.~3, 425--452.

\bibitem[GL95]{GaboriauLevitt:rank}
D.~Gaboriau and G.~Levitt, \emph{The rank of actions on {${\bf R}$}-trees},
  Ann. Sci. \'Ecole Norm. Sup. (4) \textbf{28} (1995), no.~5, 549--570.

\bibitem[Haw08]{Hawkins:PerronFrobenius}
T.~Hawkins, \emph{Continued fractions and the origins of the
  {P}erron-{F}robenius theorem}, Arch. Hist. Exact Sci. \textbf{62} (2008),
  no.~6, 655--717. \MR{MR2457065 (2009i:01018)}

\bibitem[Hir76]{Hirsch:DiffTop}
M.~Hirsch, \emph{Differential topology}, Springer, 1976.

\bibitem[Iva92]{Ivanov:subgroups}
N.~V. Ivanov, \emph{{Subgroups of Teichm\"uller modular groups}}, Translations
  of Mathematical Monographs, vol. 115, Amer. Math. Soc., 1992.

\bibitem[LL08]{LevittLustig:PeriodicDynamics}
G.~Levitt and M.~Lustig, \emph{Automorphisms of free groups have asymptotically
  periodic dynamics}, J. Reine Angew. Math. \textbf{619} (2008), 1--36.
  \MR{MR2414945}

\bibitem[Mil82]{Miller:Nielsen}
R.~T. Miller, \emph{Geodesic laminations from {N}ielsen's viewpoint}, Adv. in
  Math. \textbf{45} (1982), no.~2, 189--212.

\bibitem[Pau97]{Paulin:AutExt}
F.~Paulin, \emph{Sur les automorphismes ext\'erieurs des groupes
  hyperboliques}, Ann. Sci. \'Ecole Norm. Sup. (4) \textbf{30} (1997), no.~2,
  147--167. \MR{MR1432052 (98c:20070)}

\bibitem[Spa81]{Spanier}
E.~Spanier, \emph{Algebraic topology}, Springer, New York---Berlin, 1981.

\bibitem[Sta83]{Stallings:folding}
J.~Stallings, \emph{Topology of finite graphs}, Inv. Math. \textbf{71} (1983),
  551--565.

\end{thebibliography}

 \end{document}